\documentclass{IEEEtran}
\usepackage{cite}
\usepackage{algorithmic}
\usepackage{graphicx}
\usepackage{textcomp}

\usepackage{etex}
\usepackage{mathrsfs}
\usepackage{amsmath,amssymb,pstricks,color,dsfont,bm}

\usepackage{amsthm}
\usepackage{graphicx}
\usepackage{algorithm, algorithmic}
\usepackage{epstopdf}
\usepackage{multirow}
\usepackage{multicol}
\usepackage{enumitem}
\usepackage{stfloats}
\usepackage{caption}
\usepackage{mathtools}

\def\BibTeX{{\rm B\kern-.05em{\sc i\kern-.025em b}\kern-.08em
    T\kern-.1667em\lower.7ex\hbox{E}\kern-.125emX}}

\newcommand{\MATLABStyle}[1]{{\fontfamily{pcr}\selectfont{#1}}}

\newcommand{\Lim}[1]{\underset{#1}{\mathrm{Lim}}}
\newcommand{\Liminf}[1]{\underset{#1}{\mathrm{Liminf}}}
\newcommand{\Limsup}[1]{\underset{#1}{\mathrm{Limsup}}}
\newcommand{\squeezeup}{\vspace{-2.5mm}}

\newtheorem{theorem}{Theorem}[section]
\newtheorem{lemma}{Lemma}[section]
\newtheorem{corollary}{Corollary}[section]

\newtheorem{assumption}{Assumption}[section]
\newtheorem{definition}{Definition}[section]
\newtheorem{claim}{Claim}[section]

\newtheorem{XxmpX}{Remark}[section] 
\newenvironment{remark} 
  {%
  \pushQED{\qed}\begin{XxmpX}}
  {\popQED\end{XxmpX}}

\renewenvironment{proof}{{\textit{Proof:}}}{\qed}

\newenvironment{proofOf}[1]{{\textit{Proof of {#1}:}}}{\qed}

\begin{document}
\title{Pareto optimal multi-robot motion planning}
\author{Guoxiang Zhao and Minghui Zhu
\thanks{
This work was supported by the grants NSF ECCS-1710859 and NSF CNS-1830390.}
\thanks{G. Zhao and M. Zhu are with the School of Electrical Engineering and Computer Science, Pennsylvania State University, University Park, PA 16802 USA. (e-mail: gfz5014@psu.edu; muz16@psu.edu) }
}

\maketitle

\begin{abstract}
This paper studies a class of multi-robot coordination problems where a team of robots aim to reach their goal regions with minimum time and avoid collisions with obstacles and other robots.
A novel numerical algorithm is proposed to identify the Pareto optimal solutions where no robot can unilaterally reduce its traveling time without extending others'. 
The consistent approximation of the algorithm in the epigraphical profile sense is guaranteed using set-valued numerical analysis. 
Experiments on an indoor multi-robot platform and computer simulations show the anytime property of the proposed algorithm; i.e., it is able to quickly return a feasible control policy that safely steers the robots to their goal regions and it keeps improving policy optimality if more time is given.

\end{abstract}

\begin{IEEEkeywords}
robotic motion planning, multi-robot coordination, Pareto optimality
\end{IEEEkeywords}

\begin{section}{Introduction}\label{sect:intro}

Robotic motion planning is a fundamental problem where a control sequence is found to steer a mobile robot from an initial state to a goal set, while enforcing dynamic constraints and environmental rules. It is well-known that the problem is computationally challenging. For example, the piano-mover problem is shown to be PSPACE-hard in general \cite{reif1979complexity}. Sampling-based algorithms are demonstrated to be efficient in addressing robotic motion planning in high-dimensional spaces. The Rapidly-exploring Random Tree (RRT) algorithm \cite{lavalle2001randomized} and its variants are able to quickly find feasible paths. However, the optimality of returned paths is probably lost. In fact, computing optimal motion planners is much more computationally challenging than finding feasible motion planners \cite{canny1988complexity}. It is shown that computing the shortest path in $\mathbb R^3$ populated with obstacles is NP-hard in the number of obstacles \cite{canny1988complexity}. Recently, RRT* \cite{karaman2011samplingbased} and its variants are shown to be both computationally efficient and asymptotically optimal.

Multi-robot optimal motion planning is even more computationally challenging, because the worst-case computational complexity exponentially grows as the robot number. 
Current multi-robot motion planning mainly falls into three categories: centralized planning \cite{sanchez2002delaying}\cite{xidias2008motion}, decoupled planning \cite{kant1986efficient}\cite{simeon2002path} and priority planning \cite{buckley1989fast}\cite{erdmann1987multiple}. 
Noticeably, none of these multi-robot motion planners are able to guarantee the optimality of returned solutions. 
Recent papers \cite{zhu2014game} and \cite{jha2015game} employ game theory to synthesize open-loop planners and closed-loop controllers to coordinate multiple robots respectively. 
It is shown that the proposed algorithms converge to Nash equilibrium \cite{nash1951noncooperative} where no robot can benefit from unilateral deviations. 
As RRTs, the algorithms in \cite{zhu2014game}\cite{jha2015game} leverage incremental sampling and steering functions, the latter of which require to solve two-point boundary value problems. There are only a very limited number of dynamic systems whose steering functions have known analytical solutions, including single integrators, double integrators and Dubin's cars \cite{karaman2013samplingbased}.
Heuristic methods are needed to compute steering functions when dynamic systems are complicated.

In the control community, distributed coordination of multi-robot systems has been extensively studied in last decades \cite{ren2008distributed,bullo2009distributed,zhu2015distributed}. A large number of algorithms have been proposed to accomplish a variety of missions; e.g., rendezvous \cite{cao2008reaching}, formation control \cite{ren2008distributed}, vehicle routing \cite{frazzoli2004decentralized} and sensor deployment \cite{cortes2004coverage}\cite{schwager2009decentralized}. This set of work is mainly focused on the design and analysis of algorithms, which are scalable with respect to network expansion. To achieve scalability, most algorithms adopt gradient descent methodologies, which are easy to implement. 
Their long-term behavior; e.g., asymptotic convergence, can be ensured but usually there is no guarantee on transient performance; e.g., aggregate costs, due to the myopic nature of the algorithms. 
Another set of more relevant papers is about (distributed) receding-horizon control or model predictive control (MPC) for multi-robot coordination. Representative works include \cite{dunbar2006distributed}\cite{zhu2013distributed} on formation stabilization, \cite{dunbar2011distributed}\cite{li2016distributed} on vehicle platooning and \cite{kuwata2007distributed} on trajectory optimization. 
Model predictive control bears the following benefits \cite{bemporad1999robust, mayne2000constrained, mayne2014model}.
First, it has a unique ability to cope with hard constraints on controls and states.
Second, it can deal with system uncertainties and control disturbances and its robust stability can be formally guaranteed.
Third, it is suitable for control applications requiring rapid computations thanks to its online fashion of implementation.
The infinite-horizon performance of $N$-horizon MPC policy exponentially converges to the optimal value function of the infinite-horizon optimal control problem as the computing horizon $N$ extends to infinity \cite{grune2008infinite}.
In contrast, multi-robot motion planning aims to find controllers which can optimize certain cost functionals over entire missions; e.g., finding collision-free paths with shortest distances or minimum fuel consumption.

Differential games extend optimal control from single players to multiple players. Linear-quadratic differential games are the most basic, and their solutions can be formulated as coupled Riccati equations \cite{basar1999dynamic}. 
For nonlinear systems with state and input constraints, there are a very limited number of differential games whose closed-form solutions are known, and some examples include the homicidal-chauffeur and the lady-in-the-lake games \cite{basar1999dynamic}\cite{isaacs1999differential}. 
Otherwise, numerical algorithms are desired. Existing numerical algorithms are mainly based on partial differential equations \cite{bardi2008optimal,bardi1999numerical,souganidis1999twoplayer} and viability theory \cite{cardaliaguet1999setvalued,aubin2009viability,aubin2011viability}. Noticeably, this set of papers only considers zero-sum two-player scenarios.

\emph{Contribution statement:} This paper investigates a class of multi-robot closed-loop motion planning problems where multiple robots aim to reach their respective goal regions as soon as possible.
The robots are restricted to complex dynamic constraints and need to avoid the collisions with static obstacles and other robots. 
Pareto optimality is used as the solution notion where no robot can reduce its own travelling time without extending others'. 
A numerical algorithm is proposed to identify the Pareto optimal solutions. 
It is shown that, under mild regularity conditions, the algorithm can consistently approximate the epigraph of the minimal arrival time function.
The proofs are based on set-valued numerical analysis \cite{cardaliaguet1999setvalued,aubin2009viability, aubin2011viability}, which are the first to point out the promise in extending set-valued tools to multi-robot motion planning problems.
Experiments on an indoor multi-robot platform and computer simulations on unicycle robots are conducted to demonstrate the anytime property of our algorithm; i.e., it is able to quickly return a feasible control policy that safely steers the robots to their goal regions and it keeps improving policy optimality if more time is given.
Detailed proofs are provided in Section \ref{sect:analysis}.
Preliminary results are included in \cite{zhao2018pareto} where all the proofs and experimental results are removed due to space limitation.

\end{section}

\begin{section}{Problem Formulation}\label{sect:formulation}

Consider a team of mobile robots labeled by $\mathcal{V}\triangleq\{1, ..., N\}$.
The dynamic of robot $i$ is governed by:\begin{equation}\label{eq:0}
\dot{x}_i(s)=f_i(x_i(s), u_i(s)),\quad\forall i\in\mathcal{V},
\end{equation}
where $x_i(s)\in X_i$ is the state of robot $i$ and $u_i:[0, +\infty)\to U_i$ is the control of robot $i$. 
Here, the state space and the set of all possible control values for robot $i$ are denoted by $X_i\subseteq\mathbb{R}^{d_i}$ and $U_i\subseteq\mathbb{R}^{m_i}$ respectively.
The obstacle region and goal region for robot $i\in\mathcal V$ are denoted by $ X_i^O\subseteq X_i$ and $X_i^G\subseteq X_i\setminus X_i^O$ respectively.
Denote the minimum safety distance between any two robots as $\sigma>0$.
The free region for robot $i$ is denoted by $X_i^F\triangleq\{x_i\in X_i\setminus X_i^O|\|x_i-x_j\|\geq\sigma, x_j\in X_j^G, i\neq j\}$.
Let $\textbf{X}\triangleq\prod_{i\in\mathcal V} X_i$, $\textbf{X}^G\triangleq\prod_{i\in\mathcal V} X_i^G$ and $\textbf{X}^F\triangleq\prod_{i\in\mathcal V} X_i^F$.
Assume $\|x_i-x_j\|\geq\sigma, \forall x\in \textbf{X}^G, i\neq j$.
Define the safety region as $\textbf{S}\triangleq\{x\in\textbf{X}^F|\|x_i-x_j\|\geq\sigma, i\neq j\}$.
Here $\|\cdot\|$ denotes the $2$-norm.

The sets of state feedback control policies for robot $i$ and the whole robot team are defined as $\varpi_i\triangleq\{\pi_i(\cdot):\textbf{X}\to U_i\}$ and $\varpi\triangleq\{\prod_{i\in\mathcal V}\pi_i(\cdot)| \pi_i(\cdot)\in\varpi_i\}$ respectively.
Consider the scenario where the robot team starts from $x\in\textbf{X}$ and executes policy $\pi\in\varpi$.
The induced minimal arrival time vector is characterized as $\vartheta(x, \pi)\triangleq\inf\{t\in\bar{\mathbb{R}}_{\geq0}^N|\forall i\in\mathcal{V}, x_i(0)=x_i, \dot{x}_i(s)=f_i(x_i(s), \pi_i(x(s))), x(s)\in\textbf{S}, x_i(t_i)\in X_i^G, 0\leq s\leq \max_{i\in\mathcal V}t_i\}$, where the infimum uses the partial order in footnote~\ref{ftnt:1}.
\footnotetext[1]{
Throughout this paper, product order is imposed; i.e. two vectors $a, b\in\mathbb R^N$ are said ``$a$ is less than $b$ in the Pareto sense'', denoted by $a\preceq b$, if and only if  $a_i\leq b_i, \forall i\in\{1,\cdots,N\}$. 
Similarly, strict inequality can be defined by $a\prec b\iff a_i<b_i, \forall i\in\{1, \cdots, N\}$.
\label{ftnt:1}
}
The $i$-th element of $\vartheta(x, \pi)$ represents the first time robot $i$ reaches its goal region without collisions when the robot team starts from initial state $x$ and executes policy $\pi$.
In our multi-robot motion planning problem, the minimal arrival time function $\varTheta^*:\textbf{X}\rightrightarrows\bar{\mathbb{R}}_{\geq0}^N$ is a set-valued map and is defined as $\varTheta^*(x)\triangleq\mathcal{E}[cl(\{\vartheta(x, \pi)|\pi\in\varpi\})]$, where $\mathcal{E}$ is the Pareto minimization defined as $\mathcal{E}(\mathcal{T})\triangleq\{\tau\in\mathcal{T}|\nexists \tau'\in\mathcal{T}\text{ s.t. }\tau'\neq\tau\text{ and }\tau'\preceq\tau\}$ for $\mathcal{T}\subseteq\mathbb{R}^N_{\geq0}$ and $cl(\cdot)$ is the closure.
The closure ensures the existence of $\varTheta^*(x)$ per Theorem 4.1 of \cite{hartley1978cone}.
The vectors in $\varTheta^*(x)$ indicate that no robot can unilaterally reach its goal region earlier without extending other robots' travelling times. 
The associated set of Pareto optimal solutions is defined as $\mathcal{U}^*(x)\triangleq\{\pi^*\in\varpi|\vartheta(x, \pi^*)\in\varTheta^*(x)\}$.
Note that the elements of $\vartheta(x, \pi^*)$ could be infinite, indicating that some robots cannot safely reach their goal regions.
Infinite time may cause numerical issues.
To tackle this, transformed minimal arrival time function is defined as $v^*(x)\triangleq\Psi(\varTheta^*(x))$, where Kruzhkov transform $\Psi(t)\triangleq\begin{bmatrix}1-e^{-t_1}\\1-e^{-t_2}\\\vdots\\1-e^{-t_N}\end{bmatrix}$
 for $t\in\bar{\mathbb R}^N_{\geq0}$ normalizes $[0, +\infty]$ to $[0,1]$.
Notice that Kruzhkov transform is bijective and monotonically increasing.

The objective of this paper is to identify optimal control policies in $\mathcal{U}^*(x)$ and the corresponding minimal arrival time function $\varTheta^*(x)$ (or equivalently $v^*(x)$).

\end{section}

\begin{section}{Assumptions and Notations}\label{sect:notations}
This section summarizes the assumptions, notions and notations used throughout the paper. 
Most notions and notations on sets and set-valued maps follow the presentation of \cite{aubin2009setvalued}.

The multi-robot system \eqref{eq:0} can be written in the differential inclusion form:
$\dot{x}_i(s)\in F_i(x_i(s)), \forall s\geq0,$
where the set-valued map $F_i: X_i\rightrightarrows\mathbb R^{d_i}$ is defined as $F_i(x_i)\triangleq\{f_i(x_i, u_i)| u_i\in U_i\}$.
Let $F(x)\triangleq\prod_{i\in\mathcal V}F_i(x_i)$.
The following assumptions are imposed.
\begin{assumption}\label{asmp:1}
The following properties hold for $i\in\mathcal V$:
\begin{enumerate}[label=\textbf{(A\arabic*)}, leftmargin = *, align=left]
\item\label{asmp:compact} $ X_i\text{ and } U_i$ are non-empty and compact;
\item\label{asmp:continuous} $f_i(x_i, u_i)$ is continuous over both variables;
\item\label{asmp:lineargrowth} $f_i(x_i, u_i)$ is linear growth; i.e., $\exists c_i\geq0 \text{ s.t. }\forall x_i\in X_i$ and $\forall u_i\in U_i$, $\|f_i(x_i, u_i)\|\leq c_i(\|x_i\|+\|u_i\|+1)$;
\item\label{asmp:uconvex} For each $x_i\in  X_i$, $F_i(x_i)\text{ is convex}$;
\item\label{asmp:Lipschitz} $F_i(x_i)$ is Lipschitz with Lipschitz constant $l_i$.
\end{enumerate}
\end{assumption}

Assumptions~\ref{asmp:compact} and \ref{asmp:continuous} imply $\|f_i(x_i, u_i)\|$ is bounded for each $i\in\mathcal V$.
Define $M_i\triangleq\max_{x_i\in X_i, u_i\in U_i}\|f_i(x_i, u_i)\|$ and let $M^+\triangleq\sqrt{\sum_{i\in\mathcal V}M_i^2}$ and $l^+\triangleq\sqrt{\sum_{i\in\mathcal V}l_i^2}$.
Then $F$ is bounded by $M^+$ and is $l^+$-Lipschitz.

\begin{remark}
One sufficient condition of Assumption \ref{asmp:uconvex} is that $f_i(x_i, u_i)$ is linear with respect to $u_i$ and $U_i$ is convex.
One sufficient condition of Assumption \ref{asmp:Lipschitz} is that $f_i(x_i, u_i)$ is Lipschitz continuous with respect to both variables on $X_i\times U_i$.
\end{remark}

Define the distance from a point $x\in\mathcal{X}$ to a set $A\subseteq\mathcal{X}$ as $d(x, A) \triangleq \inf \{ \|x-a\| | a \in A \}$.
A closed unit ball around $x\in\mathcal X$ in space $\mathcal X$ is denoted as $x+\mathcal{B}_{\mathcal X}\triangleq\{y\in\mathcal X|\|y-x\|\leq1\}$. 
Similarly, $\delta$ expansion of a set $A\subseteq\mathcal X$ is defined as $A+\delta\mathcal B_{\mathcal X}\triangleq\{x\in\mathcal X|d(x, A)\leq \delta\}$ for some $\delta\geq0$.
Specifically, we denote $x+\mathcal B_N\triangleq\{y\in\mathbb R^N|\|y-x\|\leq1\}$ if $x\in\mathbb R^N$.
Similar notation applies to a set $A$.
The subscript of closed unit ball may be omitted when there is no ambiguity.
The Hausdorff distance that measures the distance of two sets $A$ and $B$ is defined by $d_H(A, B)\triangleq\inf\{\delta\geq0|A\subseteq B+\delta\mathcal B, B\subseteq A+\delta\mathcal B\}$.
Kuratowski lower limit and Kuratowski upper limit of sets $\{A_n\}\subseteq\mathcal X$ are denoted by $\mathrm{Liminf}_{n \to+\infty} A_{n}=\{x\in\mathcal X | \lim_{n \to+\infty} d(x, A_{n}) = 0 \}$ and $\mathrm{Limsup}_{n \to+\infty} A_{n} = \{ x\in\mathcal X| \liminf_{n \to+\infty} d(x, A_{n}) = 0\}$ respectively. If $\mathrm{Liminf}_{n \to+\infty} A_{n}= \mathrm{Limsup}_{n \to+\infty} A_{n}$, the common limit is defined as Kuratowski limit $\mathrm{Lim}_{n \to+\infty} A_{n}$.

The Pareto frontier of a nonempty set $A\subseteq\mathcal X$ is denoted as $\mathcal{E}(A)\triangleq\{t\in A|\nexists t'\in A \text{ s.t. } t'\neq t, t'\preceq t\}$.
Let $A+B\triangleq\{a+b|a\in A, b\in B\}$ be the sum of two sets $A$ and $B$.
Denote the $n$-fold Cartesian product of a set $A$ by $A^n$.
Specifically, when $A$ is an interval; e.g., $A=[a, b]$, its $n$-fold product is denoted by $[a, b]^n$.
When $A$ is a singleton; e.g., $A=\{a\}$, its $n$-fold product is written as $\{a\}^n$.
Let $A\times\{b\}\triangleq\{(a,b)|a\in A\}$ be the Cartesian product of a set $A$ and a point $b$.
Define Hadamard product for two vectors $a,b\in\mathbb R^N$ as $a\circ b\triangleq\begin{bmatrix}a_1b_1&\cdots&a_Nb_N\end{bmatrix}^T$.
Define $a\circ B\triangleq \{a\circ b|b\in B\}$.
Denote $N$-dimensional zero vector and all-ones vector by $\textbf{0}_N$ and $\textbf{1}_N$ respectively.
The subscript may be omitted when there is no ambiguity.
The cardinality of a set is denoted as $|\cdot|$.

Define the distance between two set-valued maps $g, \bar g:\textbf{X}\rightrightarrows[0,1]^N$ by $d_{\textbf{X}}(g, \bar g)\triangleq\sup_{x\in \textbf{X}}d_H(g(x),\bar g(x))$.
\begin{definition}[Epigraph]\label{def:epigraph}
The epigraph of $\varTheta$ is defined by $\mathcal{E}pi(\varTheta)\triangleq\{(x,t)\in\mathcal X\times\mathbb R^N|\exists t'\in\varTheta(x)\text{ s.t }t\succeq t'\}$.
\end{definition}
\begin{definition}[Epigraphical Profile]\label{def:epigraphicalProfile}
The epigraphical profile of $\varTheta$ is defined by $E_\varTheta(x)\triangleq\varTheta(x)+\mathbb R_{\geq0}^N$.
\end{definition}

\begin{remark}
For a Kruzhkov transformed function $v$, we define its epigraphical profile by $E_v(x)\triangleq(v(x)+\mathbb R^N_{\geq0})\cap[0,1]^N$.
\end{remark}

\end{section}

\begin{section}{Algorithm Statement and Performance Guarantee}\label{sect:algorithm}

In this section, we present our algorithmic solution and summarize its convergence in Theorem~\ref{thm:multiRobotConvergence}.

\begin{subsection}{Algorithm statement}\label{sect:as}
The proposed algorithms, Algorithms~\ref{alg:1}, \ref{alg:SVD} and \ref{alg:VI}, are informally stated as follows.
The state space of each robot is discretized by a sequence of finite grids $\{X^p_i\}\subseteq X_i$ s.t. $X^p_i\subseteq X^{p+1}_i, \forall p\geq1$, where $p$ is the grid index and by convention $X^0_i=\emptyset$. 
The state space for the robot team is discretized by $\{\textbf{X}^p\}\subseteq \textbf{X}$ with monotonic spatial resolutions $h_p\to0$, where $\textbf{X}^p\triangleq\prod_{i\in\mathcal V}X^p_i$.
The safety region \textbf{S} is discretized as $\textbf{S}^p\triangleq(\textbf{S}+h_p\mathcal B_\textbf{X})\cap\textbf{X}^p$.
On each grid $\textbf{X}^p$, our algorithm chooses temporal resolution $\epsilon_p > 2h_p$.
Denote $\mathbb R^p_{\geq0}$ as an integer lattice on $\mathbb R_{\geq0}$ consisting of segments of length $h_p$, and $(\mathbb R^N_{\geq0})^p$ as a lattice on $\mathbb R^N_{\geq0}$.

\begin{algorithm}[H]\caption{Pareto-based anytime algorithm}\label{alg:1}
\begin{algorithmic}[1]
\STATE \textbf{Input:} System dynamics $f$, state space $\textbf{X}$, discretization grids $\{\textbf{X}^p\}_{p=1}^P$, the  associated resolutions $h_p$, $\epsilon_p$ and the number of value iterations to be executed $n_p$.
\FOR{$1\leq p\leq P$}
\item[] \emph{Grid refinement}
\STATE $\alpha_p = 2h_p+\epsilon_p h_pl^++\epsilon_p^2l^+M^+$
\STATE $\textbf{S}^p=(\textbf{S}+h_p\mathcal B_\textbf{X})\cap\textbf{X}^p$\label{alg:1:Sp}
\item[] \emph{Value function interpolation}
\FOR{$x\in\textbf{X}^{p-1}$}
\label{alg:1:step1_begin}
	\STATE $\tilde v^{p-1}(x)=v^{p-1}_{\bar n_{p-1}}(x)$
\ENDFOR\label{alg:1:tildeXp_end}
\FOR{$x\in\textbf{S}^p\setminus\textbf{X}^{p-1}$}
	\FOR{$i\in\mathcal V$}
		\IF{$d(x_i,  X^G_i)\leq M_i\epsilon_{p}+h_p$}
			\STATE $\tilde v^{p-1}_i(x)=0$
		\ELSE
			\STATE $\tilde v^{p-1}_i(x)=1$
		\ENDIF
	\ENDFOR
\ENDFOR

\FOR{$x\in\textbf{X}^p\setminus(\textbf{S}^p\bigcup\textbf{X}^{p-1})$}
	\STATE $\tilde v^{p-1}(x)=\{\textbf{1}_N\}$
\ENDFOR\label{alg:1:step1_end}
\item[] \emph{Value function initialization}
\FOR{$x\in\textbf{X}^{p-1}$}\label{alg:1:Xp_begin}
	\STATE $v^p_0(x)=\tilde v^{p-1}(x)$
\ENDFOR\label{alg:1:Xp_end}\FOR{$x\in\textbf{X}^p\setminus\textbf{X}^{p-1}$}\label{alg:1:init_begin}
	\STATE $v^p_0(x)=\bigcup_{\tilde x\in X^p_E(x)}\tilde v^{p-1}(\tilde x)$\label{alg:1:initializeValue}
\ENDFOR\label{alg:1:init_end}
\item[] \emph{Value function update}
 
\FOR{$x\in \textbf{S}^p\setminus(\textbf{X}^G+(M^+\epsilon_p+h_p)\mathcal B_\textbf{X})$}\label{alg:1:SVD_begin}
	\STATE$(\tilde X^p(x), \tilde T^p(x))=\text{\MATLABStyle{Set\_Valued\_Dynamic}}(x, \textbf{S}^p)$
\ENDFOR\label{alg:1:SVD_end}
\STATE $n=0$
\WHILE{ $n\leq n_p$ \textbf{and} $v^p_n\neq v^p_{n-1}$}\label{alg:1:np}
		\STATE $n=n+1$
		\FOR{$x\in \textbf{S}^p\setminus(\textbf{X}^G+(M^+\epsilon_p+h_p)\mathcal B_\textbf{X})$}
			\STATE $(v^p_n(x), \mathcal U^p(x))\break\text{\quad}=\text{\MATLABStyle{Value\_Iteration}}(x, \tilde X^p(x), \tilde T^p(x), v^p_{n-1})$\label{alg:1:VI}
		\ENDFOR\label{alg:1:VIloopend}
\ENDWHILE\label{alg:1:np_end}
\STATE $\bar n_p=n$
\FOR{$x\in(\textbf{X}^p\cap(\textbf{X}^G+(M^+\epsilon_p+h_p)\mathcal B_\textbf{X}))\cup(\textbf{X}^p\setminus\textbf{S}^p)$}
\STATE $v^p_{\bar n_p}(x)=\tilde v^{p-1}(x)$
\ENDFOR
\ENDFOR
\STATE \textbf{Output:} $v^p_{\bar n_p}$, $\mathcal{U}^p$\label{alg:final}
\end{algorithmic}
\end{algorithm}

With these spatial and temporal discretization, Algorithm~\ref{alg:1} leverages the idea of multi-grid methods to search for the minimal arrival time function.
Specifically, Algorithm~\ref{alg:1} iteratively executes the following two phases: initializing the solution on $\textbf{X}^p$ by utilizing the results from $\textbf{X}^{p-1}$ and partially solving a multi-robot optimal control problem on grid $\textbf{X}^p$.
We start with the second phase, which consists of two steps: construction of set-valued dynamics as Algorithm~\ref{alg:SVD} and execution of value iteration as Algorithm~\ref{alg:VI}.

\begin{algorithm}[t]\caption{\MATLABStyle{Set\_Valued\_Dynamic}$(x, \textbf{S}^p)$}\label{alg:SVD}
\begin{algorithmic}[1]
\STATE \textbf{Input:} $x, \textbf{S}^p$\\
\FOR{$i\in\mathcal V$}
\label{alg:SVD:constructDyn_begin}
	\IF{$d(x_i,  X_i^G)>M_i\epsilon_p+h_p$}
		\STATE $\tilde{T}_i^p=\epsilon_p+2h_p\mathcal{B}_1$; $\tilde{X}_i^p=x_i+\epsilon_p F_i(x_i)+\alpha_p\mathcal{B}_{X_i}$;
	\ELSE
		\STATE $\tilde{T}_i^p=\{0\}$; $\tilde{X}_i^p=\{x_i\}$;
	\ENDIF
\ENDFOR\label{alg:SVD:constructDyn_end}
\STATE $\tilde{T}^p=(\prod_{i\in\mathcal V}\tilde{T}_i^p)\cap(\mathbb{R}^N_{\geq0})^p$; $\tilde{X}^p=(\prod_{i\in\mathcal V}\tilde{X}_i^p)\cap\textbf{S}^p$;\label{alg:SVD:intersect}\\
\STATE \textbf{Output:} $\tilde X^p, \tilde T^p$
\label{code:SVD_end}
\end{algorithmic}
\end{algorithm}

\begin{algorithm}[t]\caption{\MATLABStyle{Value\_Iteration}$(x, \tilde X^p, \tilde T^p, v^p_{n-1})$}\label{alg:VI}
\begin{algorithmic}[1]
\STATE \textbf{Input:} $x, \tilde X^p, \tilde T^p, v^p_{n-1}$\\
\label{code:VI_begin}
\STATE $v^p_n(x)=\mathcal{E}(\{\tau+\tilde\tau-\tau\circ\tilde\tau|\tilde\tau=\mathcal E(\Psi(\tilde T^p)), \tilde{x}\in\tilde{X}^p,  \tau\in v^p_{n-1}(\tilde{x})\})$\;\label{alg:VI:MODP}
\label{alg:VI:Up}
\STATE $\mathcal U^p(x)=$ \{the solutions to $u$ in the above step\}\;\label{alg:VI:control}
\STATE \textbf{Output:} $v^p_n, \mathcal U^p$
\label{code:VI_end}
\end{algorithmic}
\end{algorithm}

Step 1: in lines \ref{alg:SVD:constructDyn_begin}-\ref{alg:SVD:constructDyn_end} of Algorithm \ref{alg:SVD}, the following set-valued dynamics are constructed to approximate system~\eqref{eq:0}:
\begin{align}\label{eq:tildeXi}
\tilde{X}_i^p(x_i)=
	\begin{cases}
	x_i+\epsilon_pF_i(x_i)+\alpha_{p}\mathcal{B}_{X_i},\\
	\quad\quad\text{if }d(x_i,  X_i^G)>M_i\epsilon_p+h_p;\\
	x_i,\quad\text{otherwise,}\\
	\end{cases}
\end{align} and time dynamic $\dot{t} = 1$ is approximated by:
\begin{align}\label{eq:tildeTi}
\tilde{T}_i^p(x_i)=
	\begin{cases}
	\epsilon_p+2h_p\mathcal{B}_1, &\text{if }d(x_i,  X_i^G)> M_i\epsilon_p+h_p;\\
	0, &\text{otherwise,}\\
	\end{cases}
\end{align}
where $\alpha_p\triangleq2h_p+\epsilon_ ph_pl^++\epsilon_p^2l^+M^+$. 
Let $\tilde{X}^p(x)\triangleq\prod_{i\in\mathcal V}\tilde{X}^p_i(x_i)\cap\textbf{S}^p$ and $\tilde{T}^p(x)\triangleq\prod_{i\in\mathcal V}\tilde{T}^p_i(x_i)\cap (\mathbb R_{\geq0}^N)^p$ as line \ref{alg:SVD:intersect} in Algorithm~\ref{alg:SVD}.
The balls $\alpha_{p}\mathcal{B}_{X_i}$ in~\eqref{eq:tildeXi} and $2h_p\mathcal{B}_1$ in~\eqref{eq:tildeTi} represent perturbations on the dynamics. 
The perturbations ensure that the image set of any $x$ is non-empty and the set-valued dynamic is well-defined. 
Figure $\ref{fig:dis}$ illustrates the set-valued dynamics \eqref{eq:tildeXi}, where robot $i$ at state $x_i$ takes a constant control $u_i$ for a time duration $\epsilon_p$ and transits to the red cross.
The next state of robot $i$ could be any red diamond, which lies in the intersection of the grid and the ball centered at $x_i+\epsilon_p f_i(x_i, u_i)$ with radius $\alpha_p$.
Let $\epsilon_p\to0$ and $\frac{h_p}{\epsilon_p}\to0$; i.e., the spatial resolution $h_p$ diminishes faster than the temporal resolution $\epsilon_p$. 
This ensures the validity of the approximation in three phases: when $\alpha_p$ is very small compared to $\epsilon_p$ and $h_p$, the set-valued dynamics transit on the grid $\textbf{X}^p$; since $h_p$ is diminishing faster than $\epsilon_p$, the set-valued dynamics can well approximate the discrete-time system on $\textbf{X}$ when $p$ is sufficiently large; finally, as $\epsilon_p$ converges to $0$, the discrete-time system further converges to the continuous-time system.
When $d(x_i,  X_i^G)\leq M_i\epsilon_p+h_p$, robot $i$ is considered in the goal region, and hence it could stay still and stop counting traveling time.

\begin{figure}
\includegraphics[scale=0.55]{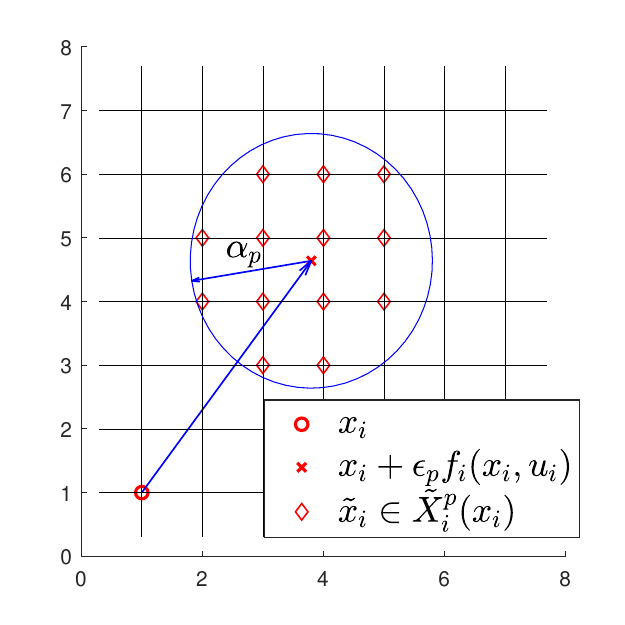}\centering\\
\caption{Set-valued discretization of robot dynamics}\label{fig:dis}
\squeezeup
\end{figure}

Step 2: given the above set-valued dynamics, Algorithm \ref{alg:VI} searches for Pareto optimal solutions of minimal arrival time vectors and stores values in $\varTheta^p_n$ and the last controls in $\mathcal{U}^p$. 
The Bellman operator in the Pareto sense is defined by 
\begin{equation}\label{eq:bellman}
\begin{split}
(\mathbb T\varTheta^p_n)(x)\triangleq&\mathcal E(\{\tilde t+t|\tilde t\in\tilde T^p(x), \tilde{x}\in \tilde{X}^p(x), t\in\varTheta^p_n(\tilde{x})\}),
\end{split}
\end{equation}
where $\varTheta^p_n:\textbf{X}^p\rightrightarrows\bar{\mathbb R}_{\geq0}^N$ is the estimate of $\varTheta^*$ after $n$ value iterations on grid $\textbf{X}^p$.
Since $\mathcal E(\tilde T^p(x))$ is a singleton, $\tilde t=\mathcal E(\tilde T^p(x))$.
When no feasible control policy exists at $x$, $\varTheta^p_n(x)$ is infinity. 
To remedy this numerical issue, we apply Kruzhkov transform on both sides of \eqref{eq:bellman} and replace $\varTheta^p_n$ with $\Psi^{-1}v^p_n$, which produces the transformed Bellman operator in the Pareto sense:
\begin{equation}\label{eq:krubellman}
\begin{aligned}
(\mathbb{G}v^p_n)(x)=
&\mathcal E(\{\tilde\tau+\tau-\tilde\tau\circ\tau|\tilde\tau=\Psi(\mathcal E(\tilde T^p(x))), \\
&\quad\tilde x\in\tilde X^p(x),\tau\in v^p_n(\tilde x)\}),
\end{aligned}
\end{equation}
where $\mathbb G\triangleq\Psi\mathbb T\Psi^{-1}$ summarizes line~\ref{alg:VI:MODP} of Algorithm \ref{alg:VI}.
Let $\mathcal{U}^p(x)$ be the set of controls which solve the last value iteration $v^p_{n}(x)=(\mathbb G v^p_{n-1})(x)$ on grid $\textbf{X}^p$.
It corresponds to line~\ref{alg:VI:control} of Algorithm~\ref{alg:VI}.

With the above two steps, Algorithm~\ref{alg:1} iteratively calls Algorithms~\ref{alg:SVD} and \ref{alg:VI} to search for the minimal arrival time function.
Denote the last estimate of minimal arrival time function on $\textbf{X}^p$ by $v^p_{\bar n_p}$, where $\bar n_p$ denotes the total number of value iterations executed on $\textbf{X}^p$.
When proceeding to grid $\textbf{X}^p$, Algorithm \ref{alg:1} first interpolates $v^{p-1}_{\bar n_{p-1}}$ to generate $\tilde v^p$ as lines \ref{alg:1:step1_begin}-\ref{alg:1:step1_end} to reuse previous computational results, then initializes value function $v^p_0$ as lines \ref{alg:1:Xp_begin}-\ref{alg:1:init_end} to reduce coupling among robots.
In particular, we maintain the estimates of minimal arrival time on the last grid $\textbf{X}^{p-1}$, assuming the fixed points on two consecutive grids are close to each other.
On new nodes $x\in\textbf{X}^p\setminus\textbf{X}^{p-1}$, $\tilde v^p(x)$ sets its $i$-th element as $0$ if robot $i$ is considered in the goal region, indicating that robot $i$ is not supposed to move and affect other robots' motions; and as $1$ otherwise, meaning no feasible solution has been found for robot $i$ yet.
Define the set of equivalent nodes $X^p_E(x)$ of $x\in\textbf{X}^p$ by
\begin{equation}\label{eq:XpE}
\begin{split}
X^p_E(x)\triangleq\{&x'\in \textbf{X}^p|x_i=x'_i, \forall i\in\mathcal V\setminus\mathcal V^G_p(x),\\
&d(x'_i, X^G_i)\leq M_i\epsilon_p+h_p,\forall i\in\mathcal V^G_p(x)\},
\end{split}\end{equation}
where $\mathcal V^G_p(x)\triangleq\{i\in\mathcal V|d(x_i, X^G_i)\leq M_i\epsilon_p+h_p\}$ denotes the set of robots which are close to or already in the goal regions.
Since robots in the goal regions never interfere with others and thus are excluded in value iterations, the values of equivalent nodes are the same.
Then the value function is initialized by $v^p_0(x)=\bigcup_{\tilde x\in X^p_E(x)}\tilde v^{p-1}(\tilde x), \forall x\in\textbf{X}^p\setminus\textbf{X}^{p-1}$; i.e., line \ref{alg:1:initializeValue} in Algorithm \ref{alg:1}.
With the initialized value function, Algorithm~\ref{alg:1} in lines~\ref{alg:1:SVD_begin}-\ref{alg:1:SVD_end} first calls Algorithm~\ref{alg:SVD} to construct set-valued dynamics and then in lines~\ref{alg:1:np}-\ref{alg:1:np_end} calls Algorithm~\ref{alg:VI} to execute value iterations for $n_p$ times or until a fixed point is reached. 
Notice that the total number of value iterations $\bar n_p$ may be less than $n_p$.
After that, Algorithm \ref{alg:1} refines the grid and begins a new cycle of updates.
\end{subsection}

\begin{subsection}{Performance guarantee}

Recall that $n_p$ at line \ref{alg:1:np} of Algorithm \ref{alg:1} is the number of value iterations to be executed on grid $\textbf{X}^p$. 
The choice of $n_p$ needs to satisfy the following assumption to ensure the convergence of Algorithm~\ref{alg:1}.

\begin{assumption}\label{asmp:DWindow}
There is a subsequence $\{D_k\}$ of the grid index sequence $\{p\}$ with $D_0=0$ s.t. $D_k-D_{k-1}\leq \bar D$ for some constant $\bar D$ and all $k\geq0$ and $\exp(-\sum_{p=D_{k-1}+1}^{D_k} n_p\kappa_p)\leq \gamma<1$ for every $k\geq0$, where $\kappa_p\triangleq(\lceil{\frac{\epsilon_p}{h_p}}\rceil-2)h_p$ is the minimum running cost.
\end{assumption}

Assumption~\ref{asmp:DWindow} implies that the distance between the estimate and the fixed point on the $D_k$-th grid reduces at least by $\gamma\in[0,1)$ over the update window length $\{D_{k-1}+1, \dots,  D_k\}$.

The choice of $\epsilon_p$ and $h_p$ should satisfy the following technical assumptions.
\begin{assumption}\label{asmp:resolutions}
The following hold for the sequences of $\{\epsilon_p\}$ and $\{h_p\}$:
\begin{enumerate}[label=\textbf{(A\arabic*)}, leftmargin = *, align=left]\setcounter{enumi}{5}
\item\label{asmp:double} $\epsilon_p>2h_p, \forall p\geq1$;
\item\label{asmp:monoConv} $\epsilon_p\to0$ and $\frac{h_p}{\epsilon_p}\to0$ monotonically as $p\to+\infty$;
\item\label{asmp:alpha} $2h_p+\epsilon_p h_pl^++\epsilon_p^2l^+M^+\geq h_{p-1}, \forall p\geq1$;
\item\label{asmp:lowerBound}$[X^G_i+(\sigma+M_i\epsilon_1+h_1)\mathcal B_{X_i}]\cap X^F_j=\emptyset, \forall i\neq j$.
\end{enumerate}
\end{assumption}
The consistent approximation of $v^*$ via Algorithm \ref{alg:1} in the epigraphical profile sense is summarized in Theorem \ref{thm:multiRobotConvergence}.
\begin{theorem}\label{thm:multiRobotConvergence}
Suppose Assumption \ref{asmp:1}, \ref{asmp:DWindow} and \ref{asmp:resolutions} hold, then the sequence $\{v^p_{\bar n_p}\}$ in Algorithm \ref{alg:1} converges to $v^*$ in the epigraphical profile sense; i.e., for any $x\in\textbf{X}$, $$E_{v^*}(x)=\Lim{p\to+\infty}\bigcup_{\tilde x\in(x+h_p\mathcal B_\textbf{X})\cap\textbf{X}^p}E_{v^p_{\bar n_p}}(\tilde x).$$
\end{theorem}
\end{subsection}

\begin{subsection}{Discussion}\label{sect:discussion}
Our proposed algorithm extends \cite{cardaliaguet1999setvalued} to multi-robot scenario.
For single robot scenario; i.e., $N=1$, if we set $\bar D = 1$ and $\gamma = 0$ and only impose Assumptions \ref{asmp:1}, \ref{asmp:double} and \ref{asmp:monoConv}, Algorithm~\ref{alg:1} and Theorem~\ref{thm:multiRobotConvergence} become Algorithm 3.2.4 on page 211 and Corollary 3.7 on page 210 of \cite{cardaliaguet1999setvalued} respectively.

However, from the analysis point of view, non-zero $\gamma$ and non-uniform lengths for update windows in the multi-robot scenario; i.e., $N\geq2$, require a set of novel analysis, which is provided in Sections~\ref{sect:analysis} and \ref{sect:proofs}.

The progress towards $v^*$ slows down or even stops as more value iterations are performed on a single grid.
A $\gamma$ close to one ensures that excessive value iterations are postponed to finer grids, and a longer update interval reduces each grid's efforts to reach the discount factor.

\end{subsection}

\end{section}

\begin{section}{Analysis}\label{sect:analysis}

In this section, we provide the major theoretic results that lead to the proof of Theorem \ref{thm:multiRobotConvergence}, which consist of four steps:

Step 1: we characterize the convergence of fixed points $v^p_\infty$ to the minimal arrival time function $v^*$; i.e., in Theorem~\ref{thm:kru1}.
The fixed point $v^p_\infty$ functions as a benchmark and we will show later that the last value function $v^p_{\bar n_p}$ on each grid $\textbf{X}^p$ can closely follow $v^p_\infty$ to converge;

Step 2: we introduce an auxiliary Bellman operator $\hat{\mathbb G}$ defined in \eqref{eq:auxBellman} to facilitate the analysis of the contraction property of the transformed Bellman operator $\mathbb G$ in the next step.
Specifically, the contraction property requires to add perturbations around all nodes in value iteration, but $\mathbb G$ imposes zero perturbation when robots are close to their goal regions.
Then $\hat{\mathbb G}$ bridges this technical gap and is equivalent to $\mathbb G$ in terms of updating value functions, which is shown in Lemma~\ref{thm:tildeHatEquivalence};

Step 3: we prove the contraction property of $\mathbb G$ via $\hat{\mathbb G}$ in Step 2 and it is summarized in Theorem \ref{thm:ContractionProperty}. 
The contraction property shows that the distance between the estimate of minimal arrival time function $v^p_n$ and the fixed point $v^p_\infty$ is exponentially discounted as value iterations are executed;

Step 4: we integrate Step 3 with Step 1 and show that $v^p_{\bar n_p}$ can closely follow $v^p_\infty$ and thus converge to $v^*$. 
In particular, the approximation errors induced by grid refinement are shown to be suppressed by sufficient value iterations and thereby the distance between $v^p_{\bar n_p}$ and $v^p_\infty$ is decreasing to zero.

This section is organized as follows.
Subsection~\ref{sect:convergenceFixedPoints} corresponds to Step 1 and introduces the convergence of fixed points; i.e., Theorem~\ref{thm:kru1}.
Subsection~\ref{sect:auxiliaryBellman} corresponds to Step 2 and confirms the equivalence of $\mathbb G$ and $\hat{\mathbb G}$ in terms of updating value functions.
Subsection~\ref{sect:propertiesHausdorff} corresponds to Step 3 and proves the contraction property of $\mathbb G$.
Step~4 is summarized in Section~\ref{sect:multiRobotMainProof}, which shows the proof of Theorem~\ref{thm:multiRobotConvergence}.
We only keep theorem statements in this section and postpone all the proofs to Section~\ref{sect:proofs}.

\begin{subsection}{Convergence of fixed points}\label{sect:convergenceFixedPoints}
The following theorem characterizes the convergence of fixed points $v^p_\infty$ to the optimal arrival time function $v^*$.

\begin{theorem}\label{thm:kru1}
Suppose Assumption \ref{asmp:1} holds and let $\epsilon_p>2h_p$, $h_p\to0$, $\frac{h_p}{\epsilon_p}\to0$.
Construct the sequence $\{v^{p}_n:\textbf{X}^p\rightrightarrows[0, 1]^N\}$ as follows:
\begin{align*}\label{eq:constructionv}
\begin{cases}
	v^{p}_0(x)=\begin{cases}	\{\textbf{\emph{0}}_N\},	&\text{ if } x\in\textbf{S}^p,\\
						\{\textbf{\emph{1}}_N\},	&\text{ otherwise},\\
						\end{cases}\\
	v^{p}_{n+1}(x)=\begin{cases}\mathbb Gv^{p}_n(x), &\text{ if } x\in\textbf{S}^p,\\
						v^p_n(x),	&\text{ otherwise},\\
						\end{cases}
\end{cases}
\end{align*}
where $\mathbb G$ is defined in \eqref{eq:krubellman}.
Then, for each $p$, there exists $v^{p}_\infty$ s.t. $\mathbb Gv^{p}_\infty=v^{p}_\infty$ and $v^p_\infty(x)=\mathrm{Lim}_{n\to+\infty}v^p_n(x), \forall x\in\textbf{X}^p$.
Further, the fixed points converge to $v^*$ in the epigraphical sense; i.e. for any $\{\eta_p\}$ s.t. $\eta_p\geq h_p$ and $\lim_{p\to+\infty}\eta_p=0$, the following holds:
\begin{equation*}
\begin{split}
\forall x\in\textbf{X}, E_{v^*}(x)=\Lim{p\to+\infty}\bigcup_{\tilde x\in(x+h_p\mathcal B_\textbf{X})\cap\textbf{X}^p}E_{v^p_\infty}(\tilde x).
\end{split}\end{equation*}
\end{theorem}

The proof of Theorem~\ref{thm:kru1} mainly follows those of Lemma 3.6 and Corollary 3.7 in \cite{cardaliaguet1999setvalued}. 
For the sake of completeness, we include the details of proofs in Section~\ref{sect:multiRobotAppendix}.
Please refer to Theorem~IX.3 and Corollary~IX.1.

\end{subsection}

\begin{subsection}{Auxiliary Bellman operator $\hat{\mathbb G}$: Lemma~\ref{thm:tildeHatEquivalence}}\label{sect:auxiliaryBellman}

In this subsection, an auxiliary Bellman operator $\hat{\mathbb G}$ is introduced as a stepping stone towards the contraction property of $\mathbb G$ in Section \ref{sect:propertiesHausdorff}.
This subsection consists of three phases:

First, $\hat{\mathbb G}$ is formally defined as \eqref{eq:auxBellman}. 
The auxiliary Bellman operator $\hat{\mathbb G}$ differs from $\mathbb G$ in the perturbations around nodes within one hop of the goal regions;

Second, the properties of $\hat{\mathbb G}$ are analyzed and it is shown that $\hat{\mathbb G}v^p_n$ is no less than $\mathbb Gv^p_n$, as Lemmas \ref{thm:fixedGridSub} and \ref{thm:multiGridSub};

Finally, $\hat{\mathbb G}v^p_n$ is no larger than $\mathbb Gv^p_n$, either, and thereby the equivalence of $\hat{\mathbb G}$ and $\mathbb G$ is established in Lemma \ref{thm:tildeHatEquivalence}.

We proceed to the first phase and derive the Bellman operator in terms of epigraphical profiles and its Kruzhkov transformed version.
We start with \eqref{eq:bellman} by adding $\mathbb R^N_{\geq0}$ to both sides:
\begin{equation*}
\begin{split}
E_{\mathbb T\varTheta^p_n}(x)&=(\mathbb T\varTheta^p_n)(x)+\mathbb R^N_{\geq0}\\
=&\{\tilde t+t|\tilde t=\mathcal E(\tilde T^p(x)),\tilde x\in\tilde X^p(x), t\in \varTheta^p_n(\tilde x)+\mathbb R^N_{\geq0}\}\\
=&\{\mathcal E(\tilde T^p(x))+t|\tilde x\in\tilde X^p(x), t\in E_{\varTheta^p_n}(\tilde x)\}\\
=&\mathcal E(\tilde T^p(x))+\bigcup_{\tilde x\in\tilde X^p(x)}E_{\varTheta^p_n}(\tilde x).
\end{split}\end{equation*}
Recall that $v^p_n(x)=(\Psi\varTheta^p_n)(x)$. Denote $\Delta\tau(x)\triangleq\Psi(\mathcal E(\tilde T^p(x))).$
Applying Kruzhkov transform to both sides yields
\begin{equation}\label{eq:epiProfileBellman}
\begin{split}
E_{\mathbb Gv^p_n}(x)
=&\Delta\tau(x)+(\textbf{1}-\Delta\tau(x))\circ \bigcup_{\tilde x\in\tilde X^p(x)}E_{v^p_n}(\tilde x).
\end{split}\end{equation}
The $i$-th element of $\Delta\tau$ can be written as
\begin{align}\label{eq:deltaTau}
\Delta\tau_i(x)=\begin{cases}
0, &\text{if } i\in\mathcal V^G_p(x);\\
1-e^{-\kappa_{p}}, &\text{otherwise,}
\end{cases}
\end{align}
where $\kappa_p$ follows the definition in Assumption \ref{asmp:DWindow}.

Now we define the auxiliary Bellman operator $\hat{\mathbb G}$ by 
\begin{align}\label{eq:auxBellman}
E_{\hat{\mathbb G}v}(x)\triangleq\Delta\tau(x)+(\textbf{1}-\Delta\tau(x))\circ\bigcup_{\hat x\in\hat X^p(x)}E_v(\hat x),
\end{align}
where $\hat X^p(x)\triangleq(\prod_{i\in\mathcal V}\hat X^p_i(x_i))\cap\textbf{S}^p$ and
\begin{equation*}\label{eq:hatXi}
\begin{split}
\hat X^p_i(x_i)\triangleq
\begin{cases}
x_i+\epsilon_p F_i(x_i)+\alpha_p\mathcal B, \text{ if }d(x_i, X^G_i)>M_i\epsilon_p+h_p;\\
x_i+\alpha_p\mathcal B, \text{ otherwise}.
\end{cases}
\end{split}\end{equation*}
If $d(x_i, X^G_i)\leq M_i\epsilon_p+h_p$, then $\tilde X^p_i(x_x)=x_i$ in $\mathbb G$ and $\hat X^p_i(x_i)=x_i+\alpha_p\mathcal B$ in $\hat{\mathbb G}$.
This is the only difference between $\mathbb G$ and $\hat{\mathbb G}$.

Before we move on to the second phase, intermediate results are required to facilitate our analysis.
The next lemma shows that the equivalent nodes of $x\in\textbf{S}^p$ are also in the safety region.
\begin{lemma}\label{thm:EquivalentSafety}
Suppose Assumption \ref{asmp:monoConv} and \ref{asmp:lowerBound} are satisfied.
Then for any $p\geq1$ and $x\in\textbf{S}^p$, it holds that $X^p_E(x)\subseteq\textbf{S}^p$.
\end{lemma}

The next lemma shows that for any robot $i\in\mathcal V^G_p(x)$, its estimate of travelling time is always $0$.
\begin{lemma}\label{thm:zeroValue}
For any $p\geq1$, the following hold:
\begin{enumerate}
\item $\mathcal V^G_p(x)\subseteq \mathcal V^G_p(\tilde x)$ for any $x\in\textbf{X}^p$ and $\tilde x\in\tilde X^p(x)$;
\item $\mathcal V^G_p(x)\supseteq \mathcal V^G_{p+1}(x)$ for any $x\in\textbf{X}^p$;
\item $\tau_i=0$ for any $x\in \textbf{S}^p$ $\tau\in v^p_n(x)$, $0\leq n\leq\bar n_p$ and $i\in\mathcal V^G_p(x)$.
\end{enumerate}
\end{lemma}

\begin{remark}
Notice that $v^p_n=\mathbb G^nv^p_0$.
Fix $p\geq1$.
It follows from the proof of the third property of Lemma \ref{thm:zeroValue} that $\tau_i=0$ for any $\tau\in\mathbb G^mv^p_0(x)$, $m\geq0$ and $i\in\mathcal V^G_p(x)$.
Specifically, by Theorem \ref{thm:kru1}, we have $v^p_\infty=\mathrm{Lim}_{n\to+\infty}v^p_n(x)=\mathrm{Lim}_{n\to+\infty}\mathbb G^nv^p_0(x)$.
Then the third property of Lemma \ref{thm:zeroValue} also applies to $v^p_\infty$.
\end{remark}
\begin{remark}
Fix $x\in\textbf{S}^p$ and $m\geq0$ and let $\mathcal V^G_p(x)=\{1, \dots, N_p\}$, where $0\leq N_p\leq N$.
By the third property of Lemma \ref{thm:zeroValue}, we have $\forall\tilde\tau\in[0, 1]^{N_p}\times\{1\}^{N-N_p}$, $\exists\tau\in \mathbb G^mv^p_0(x)$ s.t. $\tilde\tau\succeq\tau$.
This implies $[0,1]^{N_p}\times\{1\}^{N-N_p}\subseteq E_{\mathbb G^mv^p_0}(x)=(\mathbb G^mv^p_0(x)+\mathbb R^N_{\geq0})\cap[0,1]^N$.
\end{remark}

Define the set of partially perturbed state nodes $x'\in X^p_P(x)$ of $x\in\textbf{X}^p$ by 
\begin{equation*}\label{eq:XpP}
\begin{split}
X^p_P(x)\triangleq\{x'\in\textbf{S}^p|x'_i=x_i, \forall i\in\mathcal V\setminus\mathcal V^G_p(x)\}.
\end{split}\end{equation*}
The term ``partially perturbed state node'' means that $x'$ differs from $x$ only at the perturbations added to the positions of robots $i\in\mathcal V^G_p(x)$.
It is a superset of $X^p_E(x)$ in \eqref{eq:XpE}.

The following lemma shows that on a fixed grid, the partially perturbed nodes cannot have less value.
\begin{lemma}\label{thm:fixedGridSub}
Fix $p\geq1$ s.t. Assumptions \ref{asmp:monoConv} and \ref{asmp:lowerBound} are satisfied.
Consider $v^p_n:\textbf{X}^p\rightrightarrows[0,1]^N$.
If $E_{v^p_n}(x')\subseteq E_{v^p_n}(x)$ for any pair of $x\in\textbf{S}^p$ and $x'\in X^p_P(x)$, then $E_{\mathbb G^mv^p_n}(x')\subseteq E_{\mathbb G^mv^p_n}(x)$ holds for all $m\geq1$ and any pair of $x\in\textbf{S}^p$ and $x'\in X^p_P(x)$.
\end{lemma}

The next lemma extends Lemma \ref{thm:fixedGridSub} to all the iterations of Algorithm \ref{alg:1}.

\begin{lemma}\label{thm:multiGridSub}
For any pair of $x\in\textbf{S}^p$ and $x'\in X^p_P(x)$, if Assumptions \ref{asmp:monoConv} and \ref{asmp:lowerBound} are satisfied, it holds that $E_{v^p_n}(x')\subseteq E_{v^p_n}(x)$ for any $p\geq1$ and $0\leq n\leq\bar n_p$.
\end{lemma}

The next corollary shows the values of all equivalent nodes are the same.

\begin{corollary}\label{thm:XEEquivalence}
If all conditions in Lemma \ref{thm:multiGridSub} are satisfied, for any $p\geq1$, $0\leq n\leq\bar n_p$ and any pair of $x\in\textbf{S}^p$ and $x'\in X^p_E(x)$, $E_{v^p_n}(x)=E_{v^p_n}(x')$.
In addition, $E_{v^p_\infty}(x)=E_{v^p_\infty}(x')$.
\end{corollary}

Finally, we arrive at the last phase and the next lemma is the main result of this subsection that reveals the equivalence of $\mathbb G$ and $\hat{\mathbb G}$.
\begin{lemma}\label{thm:tildeHatEquivalence}
If Assumptions \ref{asmp:monoConv} and \ref{asmp:lowerBound} are satisfied, for any $p\geq1$, $0\leq n\leq\bar n_p$ and $x\in\textbf{S}^p$, it holds that $\bigcup_{\tilde x\in\tilde X^p(x)}E_{v^p_n}(\tilde x)=\bigcup_{\hat x\in\hat X^p(x)}E_{v^p_n}(\hat x)$ and $E_{\mathbb Gv^p_n}(x)=E_{\hat{\mathbb G}v^p_n}(x)$.
In addition, $\bigcup_{\tilde x\in\tilde X^p(x)}E_{v^p_\infty}(\tilde x)=\bigcup_{\hat x\in\hat X^p(x)}E_{v^p_\infty}(\hat x)$ and $E_{\mathbb Gv^p_\infty}(x)=E_{\hat{\mathbb G}v^p_\infty}(x)$.
\end{lemma}

\end{subsection}

\begin{subsection}{Contraction property of $\mathbb G$: Theorem \ref{thm:ContractionProperty}}\label{sect:propertiesHausdorff}

In this subsection, Theorem \ref{thm:ContractionProperty} shows that the transformed Bellman operator $\mathbb G$ in \eqref{eq:krubellman} is contractive with factor $e^{-\kappa_p}$.

Before we proceed to the final conclusion, the following notations are defined to facilitate our analysis.
Given a set-valued map $v:\textbf{X}^p\rightrightarrows[0,1]^N$, define the interpolation operation $\mathbb I^p$ by 
\begin{equation*}
\begin{split}
(\mathbb I^pv)(x)\triangleq\begin{cases}
v(x), &\text{if } x\in \textbf{X}^p;\\
\{V^p(x)\}, &\text{if } x\in\textbf{S}^{p+1}\setminus\textbf{X}^p;\\
\{\textbf{1}_N\}, &\text{if } x\in\textbf{X}^{p+1}\setminus(\textbf{S}^{p+1}\cup\textbf{X}^p),\\
\end{cases}
\end{split}\end{equation*}
where interpolation function $V^p:\textbf{X}^p\rightarrow\{0,1\}^N$ is defined as
\begin{equation}\label{eq:interpolationFunction}
\begin{split}
V^p_i(x)\triangleq
\begin{cases}
0, &\text{if } d(x_i,  X^G_i)\leq M_i\epsilon_{p+1}+h_{p+1};\\
1, &\text{otherwise.}\\ 
\end{cases}
\end{split}\end{equation} 
Then the interpolated value function $\tilde v^p:\textbf{X}^{p+1}\rightrightarrows [0, 1]^N$ in Algorithm \ref{alg:1} can be represented by $\tilde v^p\triangleq\mathbb I^pv^p_{\bar n_p}$.
The interpolated fixed point $\tilde v^p_\infty: \textbf{X}^{p+1}\rightrightarrows[0,1]^N$ is written as $\tilde v^p_\infty\triangleq\mathbb I^pv^p_\infty$.
Correspondingly, define the initialization operator $\mathbb P$ by 
\begin{align*}
E_{\mathbb Pv}(x)\triangleq\begin{cases}
E_{v}(x), &\text{if }x\in\textbf{X}^{p-1};\\
\bigcup_{\tilde x\in X^p_E(x)}E_v(\tilde x), &\text{if } x\in\textbf{X}^p\setminus\textbf{X}^{p-1}.
\end{cases}
\end{align*}

Define the distance between two consecutive fixed points at $x\in\textbf{X}$ by $b_p(x)\triangleq d_H(\bigcup_{\tilde x\in(x+\alpha_p\mathcal B)\cap\textbf{X}^p}E_{\mathbb P\tilde v_\infty^{p-1}}(\tilde x),\allowbreak\bigcup_{\tilde x\in(x+\alpha_p\mathcal B)\cap\textbf{X}^p}E_{v_\infty^p}(\tilde x))$.
Define $b_p\triangleq\sup_{x\in\textbf{X}}b_p(x)$.
The next lemma shows the distance diminishes.

\begin{lemma}\label{thm:Lemma9}
If Assumptions \ref{asmp:monoConv} and \ref{asmp:alpha} are satisfied, it holds that $\lim_{p\to+\infty}b_p=0$.
\end{lemma}

The following lemma shows that, under $\mathbb G$, the distance of $v^p_n$ and $v^p_\infty$ at any node $x\in\textbf{X}^p$ is discounted by $e^{-\kappa_p}$.

\begin{lemma}\label{thm:preliminaryContractionProperty}
If Assumptions \ref{asmp:monoConv} and \ref{asmp:lowerBound} are satisfied, then the following holds for any $p\geq1$, $n\geq0$ and $x\in\textbf{S}^p$:
\begin{equation*}
\begin{split}
d_H((\textbf{1}-\Delta\tau(x))\circ A, (\textbf{1}-\Delta\tau(x))\circ B)\leq e^{-\kappa_p}d_H(A, B),
\end{split}\end{equation*}
where
$
A\triangleq\bigcup_{\tilde x\in \tilde X^p(x)}E_{v^p_n}(\tilde x), 
B\triangleq\bigcup_{\tilde x\in \tilde X^p(x)}E_{v^p_\infty}(\tilde x).
$

\end{lemma}

Finally, we come to the contraction property of $\mathbb G$.

\begin{theorem}\label{thm:ContractionProperty}
If Assumptions \ref{asmp:monoConv} and \ref{asmp:lowerBound} are satisfied, the following holds for any $p\geq1$, $n\geq0$ and $x\in\textbf{S}^p$:
\begin{equation}\label{eq:contractionInequality1}
\begin{split}
&d_{\textbf{S}^p}(E_{\mathbb Gv^p_n}, E_{\mathbb Gv^p_\infty})\\
\leq&e^{-\kappa_p}d_{\textbf{X}}(\bigcup_{\tilde x\in (x+\alpha_p\mathcal B)\cap\textbf{X}^p}E_{v^p_n}(\tilde x),\bigcup_{\tilde x\in (x+\alpha_p\mathcal B)\cap\textbf{X}^p}E_{v^p_\infty}(\tilde x)).
\end{split}\end{equation}

In addition, the following is also true:
\begin{equation}\label{eq:contractionInequality2}
\begin{split}
d_{\textbf{S}^p}(E_{\mathbb Gv^p_n},  E_{\mathbb Gv^p_\infty})\leq e^{-\kappa_p}d_{\textbf{S}^p}(E_{v^p_n},E_{v^p_\infty}).
\end{split}\end{equation}
\end{theorem}
The next lemma derives a recursive relation of $d_{\textbf{X}^p}(E_{v^p_{\bar n_p}}, E_{v^p_\infty})$.

\begin{lemma}\label{thm:discountedDistanceOnOneGrid}
If Assumption \ref{asmp:1} and \ref{asmp:resolutions} are satisfied, the following inequality holds for each grid $\textbf{X}^p$:
\begin{equation}\label{eq:discountedDistanceOnOneGrid}
\begin{split}
d_{\textbf{X}^p}(E_{v^p_{\bar n_p}}, E_{v^p_\infty})
\leq \gamma_pd_{\textbf{X}^{p-1}}(E_{v^{p-1}_{\bar n_{p-1}}},E_{v^{p-1}_\infty})+b_{p},
\end{split}\end{equation}
where $\gamma_p\triangleq e^{-n_p\kappa_p}$ and $b_{p}$ is defined in Lemma \ref{thm:Lemma9}.
\end{lemma}
\end{subsection}
\end{section}

\begin{section}{Experiments and Simulations}\label{sect:simulation}

This section presents the experiments on an indoor multi-robot platform and computer simulations conducted to assess the performance of Algorithm \ref{alg:1}.
The experiment environment, shown in Figure \ref{fig:realPlat}, is a four-way intersection with no signs or signals. 
Each road is $420$mm wide and consists of two lanes of same width with opposite directions. 
Three Khepera III robots of diameters $170$mm can neither sense the environment nor communicate with each other.
A centralized computer can measure robots' locations and heading angles via Vicon system, a motion capture system, and remotely command each robot's motion via bluetooth.

\begin{figure}[ht]
\includegraphics[scale=0.3]{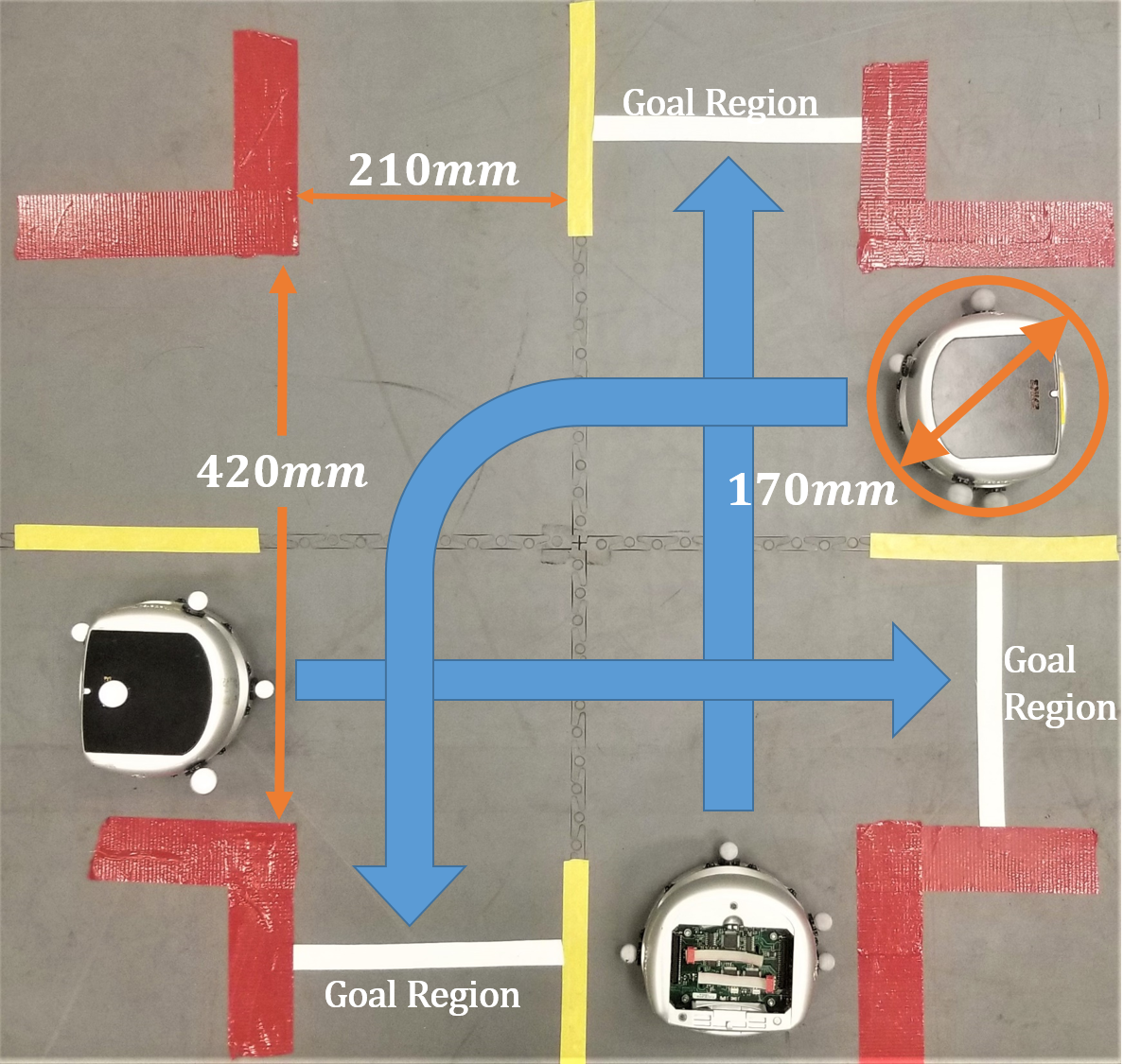}\centering
\caption{Three Khepera III robots arrive at the intersection at the same time.}\centering\label{fig:realPlat}
\squeezeup
\end{figure}

Each robot is modeled as a unicycle and its dynamic is given by $\dot p^x_i=v_i\cos \theta_i, \dot p^y_i=v_i\sin \theta_i$, where $x_i=(p^x_i, p^y_i)$ denotes the $i$-th robot's position and $u_i=(\theta_i, v_i)\in U_i=U_i^\theta\times U_i^v$ is its control including heading angle $\theta_i$ and linear speed $v_i$.
The goal for each robot is to pass the crossroads and arrive at its goal region without colliding with curbs or any other robot.
The robots stop as long as they pass their respective white goal lines in Figure~\ref{fig:realPlat}.

In practice, the allowable computational times for the robots are varying and uncertain. 
Therefore, it is desired to compute control policies, which can safely steer the robots to their goal regions within a short time and keep improving the control policies if more time is given.
This property is referred to the anytime property, which is widely adopted in robotic motion planning literature \cite{pineau2003pointbased, likhachev2005anytime, ferguson2006anytime, karaman2011anytime}.
In the following, we demonstrate that our algorithm is an anytime algorithm; i.e., it is quickly feasible and increasingly optimal.
In addition, the simulations are also used to analyze the computational complexity of our algorithm.
\begin{subsection}{Demonstration of quick feasibility}
In this subsection, an experiment on three physical robots is conducted to examine the quick feasibility of our algorithm for multiple robots. 
In our MATLAB codes, we normalize the road width to $1$ and scale robot radii to $0.2$. 
We choose $\epsilon_p=\sqrt{h_p}$.
The constraint sets of controls are given as: $U_i^v=[0, 0.25], U^\theta_1=[-\pi, -\pi/2], U^\theta_2=[-\pi/2, \pi/2]$ and $U_3^\theta=[0, \pi]$.
The dimension of state space is $6$.
For the purpose of collision avoidance, we set the inter-robot safety distance as $0.6$ and ignore perturbations added to $\textbf{S}$ in line~\ref{alg:1:Sp} of Algorithm~\ref{alg:1}; i.e., we choose $\textbf{S}^p=\textbf{S}\cap\textbf{X}^p$.
In order to efficiently address the failure of arrival caused by coarse resolutions of discrete grids, we use finer grids near goal regions.
Specifically, in the one-hop expansion of each robot's goal region $\{x\in\textbf{X}|d(x_i, X_i^G)\leq M_i\epsilon_p+h_p, i\in\mathcal V\}$, we refine the grids, perform Algorithm \ref{alg:1} on the new nodes and replace coarse controller with the refined one.
Since Algorithm~\ref{alg:1} only returns control policies on discrete grids, we need to interpolate the control policies into the continuous state space.
In particular, Unif$(\cdot)$ is used to uniformly select one control from $\mathcal U^p(x)$ for $x\in\textbf{S}^p$.
For state $x\in\textbf{X}\setminus\textbf{S}^p$, the control is interpolated by nearest neighbor method; i.e., we take $u=\text{Unif}(\mathcal U^p(\arg\min_{\hat x\in\textbf{S}^p}\|\hat x-x\|))$.
Algorithm~\ref{alg:1} is executed in MATLAB on a $3.40$ Ghz Intel Core i7 computer.

Each physical robot has inertia in changing its heading angle $\theta_i$ and is subject to $\dot\theta_i=\omega_i$, where $\omega_i$ is the angular velocity that robot $i$ can directly command.
To address this difference in dynamics, a PID controller is leveraged to modulate robots' heading angles; i.e., $\omega_i=\text{PID}(u_{i, 1}-\theta_i)$, where $u_{i, 1}$ is the returned heading angle of robot $i$.

\squeezeup
\begin{figure}[ht]
\includegraphics[scale=0.5]{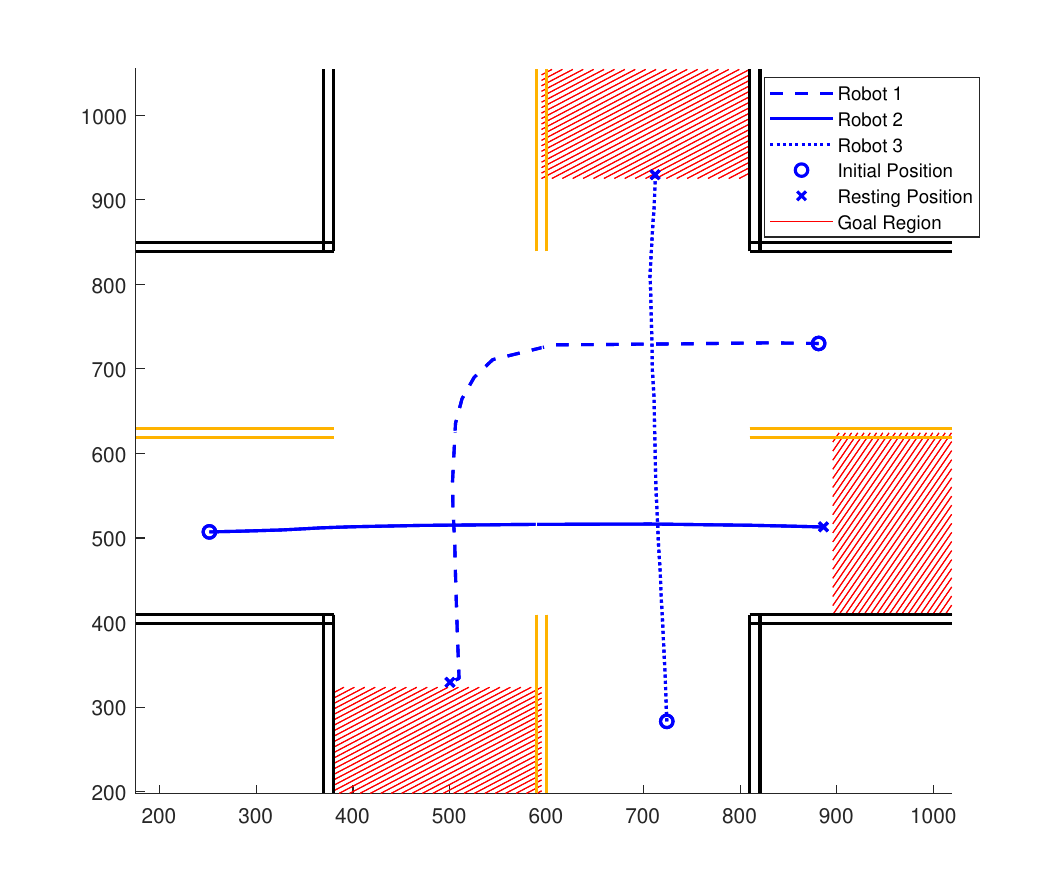}\centering
\caption{The trajectories of centers of three robots when the computation time is $1.05$s.}\centering\label{fig:1}
\squeezeup
\end{figure}
\begin{figure}[ht]
\includegraphics[scale=0.6]{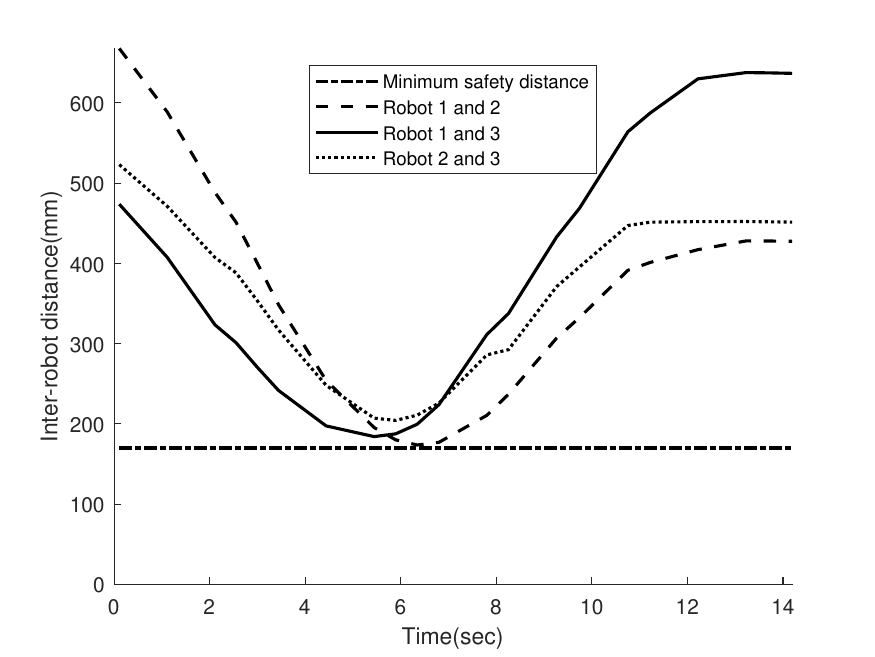}\centering
\caption{Inter-robot distances over time.}\label{fig:2}
\squeezeup
\end{figure}
\begin{figure}[ht]
\includegraphics[scale=0.6]{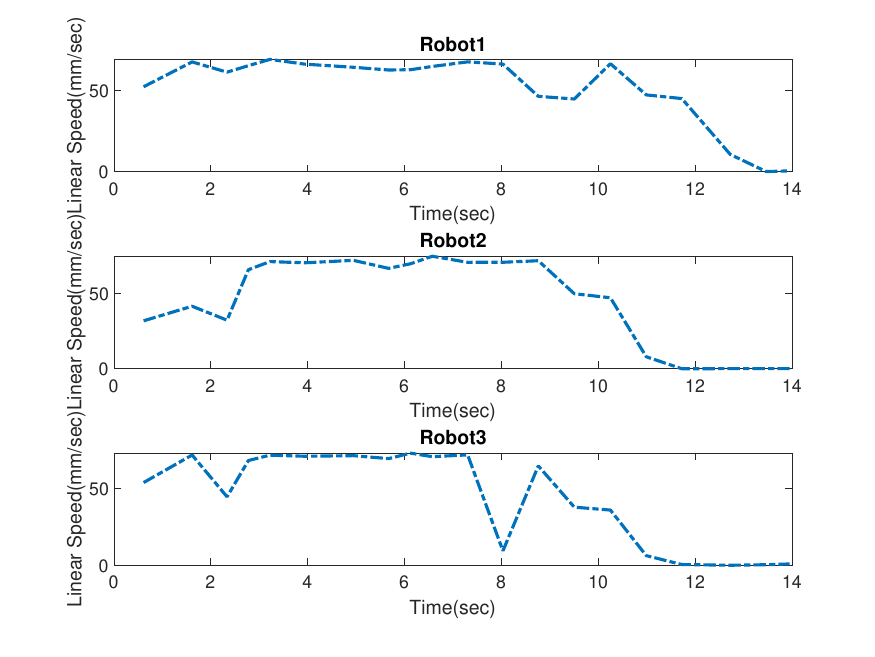}\centering
\caption{Robot linear speeds over time.}\label{fig:3}
\squeezeup
\end{figure}

Figure \ref{fig:1} shows the trajectories of the robots when they apply the interpolated control policies computed in $1.05$s.
Figure \ref{fig:2} shows the inter-robot distances over time corresponding to Figure \ref{fig:1}, indicating that no collision is caused throughout the movement of the robots.
Figure \ref{fig:3} displays the linear speeds of each robot over time.
At around $2$s, both robot $2$ and robot $3$ slow down so that robot $1$ can first pass the intersection.
At $8$s, robot $3$ is no more than one hop away from its goal region and stops owing to the coarse resolution of the grid. 
After this moment, the robots switch to the refined controller, hence robot $3$ continues to move until it rests at its goal region.
The results show that given short computational time; i.e., $1.05$s, our algorithm can already generate a feasible policy which accomplishes the planning task without violating any hard constraint.
Therefore, the quick feasibility is verified.

\end{subsection}

\begin{subsection}{Demonstration of increasing optimality}
A set of computer simulations is performed to examine the increasing optimality of Algorithm \ref{alg:1}.
The parameters are identical to the previous experiment with the differences that robot $3$ is excluded and safety distance is $0.4$.
The operating region of the robot team is discretized by the sequence of uniform square grids $\{\textbf{X}^p\}$ for $p\in\{1, \dots, 4\}$ with resolutions $h_p\in\{0.2, 0.1, 0.05, 0.025\}$, each of which contains $145, 3403, 34344, 416689$ nodes respectively.
All the grids are within the same update window.
We choose $\epsilon_p=\sqrt{h_p/M^+}$.
In computations, we only update values of nodes in the safety region $\textbf{S}^p$ as nodes in $\textbf{X}^p\setminus\textbf{S}^p$ indicate collisions and therefore are irrelevant.
In addition, we ignore the perturbation added to $\textbf{S}^p$ to avoid excessive computations.
In line~\ref{alg:1:initializeValue} of Algorithm~\ref{alg:1}, we choose any single node $x_E(x)\in X_E^p(x)\cap\textbf{X}^1$ to represent the whole equivalent set $X_E^p(x)$ as it is the minimizer of $\bigcup_{\tilde x\in X_E^p(x)}\tilde v^{p-1}(\tilde x)$.
Our algorithm refines grids if the relative difference between two consecutive value functions $v^p_n$ and $v^p_{n-1}$ is less than $10\%$ of the total difference between $v^p_n$ and $v^p_0$; i.e., $\mathcal D^p_{n-1, n}/\mathcal D^p_{0, n}\leq 10\%$, where $\mathcal D^p_{n_1, n_2}\triangleq\sqrt{\sum_{x\in \textbf{S}^p}d_H^2(v^p_{n_1}(x), v^p_{n_2}(x))}$ is the $2$-norm difference between $v^p_{n_1}$ and $v^p_{n_2}$.
The benchmark $v^\star$ is the estimate of minimal arrival time function computed on the finest grid $\textbf{S}^4$ with resolution $h_p= 0.025$.
To measure approximation errors, we use nearest neighbor method to interpolate each estimate of minimal arrival time function $v^p_n$ into $\hat v^p_n$ so that both $\hat v^p_n$ and $v^\star$ share the finest grid as their domains.
Note that $\hat v^p_n(x)\triangleq v^p_n(\arg\min_{\hat x\in\textbf{S}^p}\|\hat x-x\|)$ for every $x\in\textbf{S}^4$.
Then approximation error of $\hat v^p_n$ is measured by $\sqrt{\sum_{x\in \textbf{S}^4}d_H^2(\hat v^p_n(x), v^\star(x))}$.
Figure \ref{fig:000} shows the approximation errors over time.
The $n$-th dot from the left in Figure~\ref{fig:000} represents the total computational time after $n$ value iterations and the associated approximation error.
The peak at $2$s is caused by the nonlinearity of Kruzhkov transform, where the initial value $\textbf{1}_2$ is closer to the benchmark values.
Other than this, the approximation errors are monotonically decreasing over time.
\squeezeup
\begin{figure}[ht]
\includegraphics[scale=0.6]{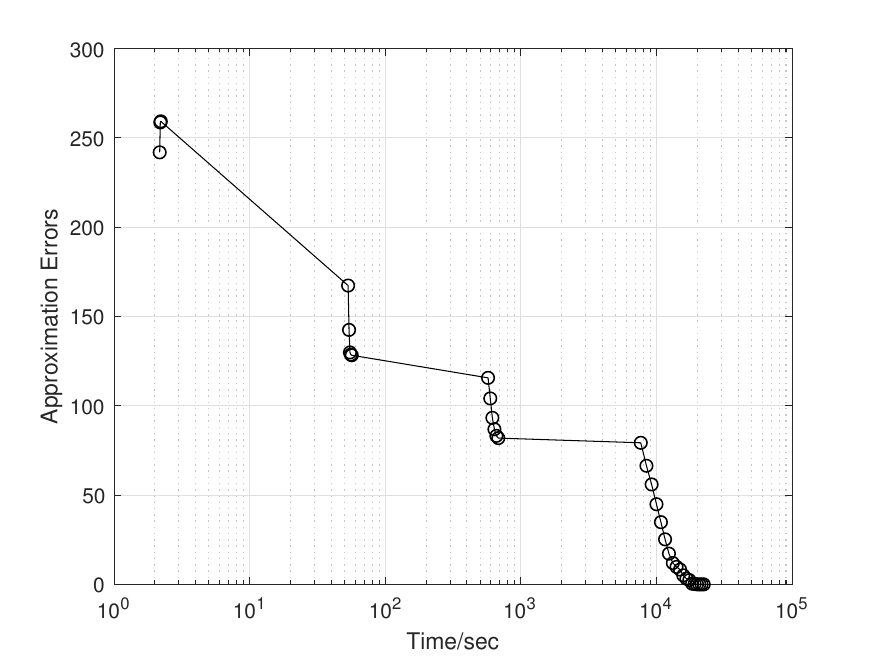}\centering
\caption{Approximation errors over time.}\label{fig:000}
\squeezeup
\end{figure}
\squeezeup

\end{subsection}

\begin{table*}[!hbp]
\centering
\begin{tabular}{|c|c|c|c|c|c|c|}
\hline
\multirow{2}{*}{Grid index} & \multirow{2}{*}{Grid size} & \multirow{2}{*}{Total time/sec} & \multicolumn{2}{c|}{Construction of set-valued dynamics} & \multicolumn{2}{c|}{Execution of value iteration}\\\cline{4-7}
 &  &  & Computational time/sec & Percentage in total time & Computational time/sec & Percentage in total time \\ \hline
1 & 145 & 2.21 & 2.14 & 96.8\% & 0.07 & 3.2\% \\ \hline
2 & 3403 & 54.35 & 50.08 & 92.1\% & 4.27 & 7.9\% \\ \hline
3 & 34344 & 624.56 & 494.03 & 79.1\% & 130.53 & 20.9\% \\ \hline
4 & 416689 & 21586.20 & 6206.59 & 28.7\% & 15379.61 & 71.3\% \\ \hline
\end{tabular}
\caption{Computational times on each grid.}\label{table:timesVsGrid}
\end{table*}

\begin{subsection}{Computational complexity}

Algorithm~\ref{alg:SVD} and Algorithm~\ref{alg:VI} correspond to two steps: construction of set-valued dynamics and execution of value iteration.
In Figure~\ref{fig:timesVsErros}, the $n$-th dot from the right represents the time to execute $n$ value iterations and the resulting approximation error except the rightmost ones around $250$.
Figure~\ref{fig:timesVsErros} shows the time to perform value iteration exponentially increases as approximation errors decrease.

\begin{figure}[t]
\includegraphics[scale=0.55]{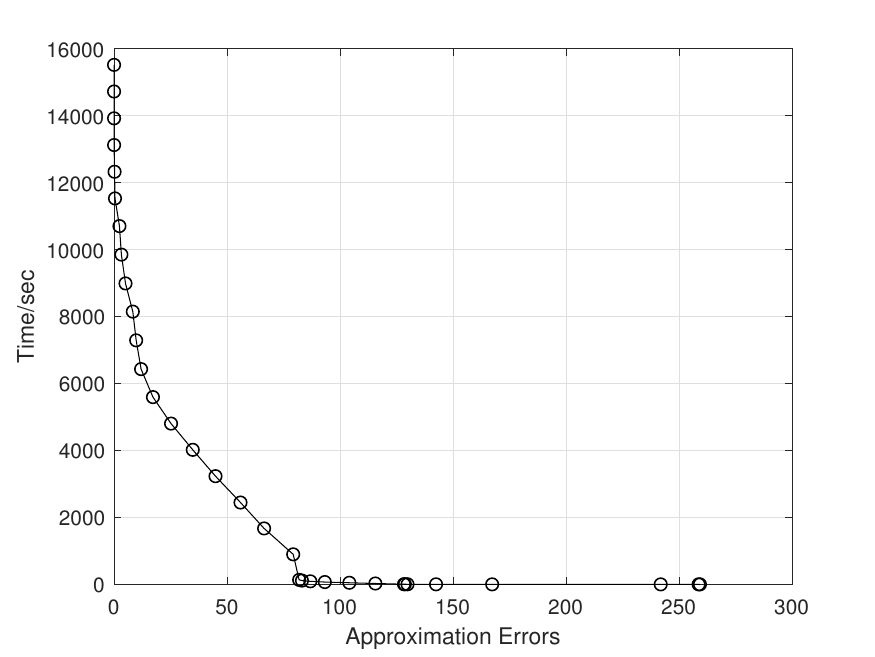}\centering
\caption{Value iteration time over approximation errors.}\centering\label{fig:timesVsErros}
\squeezeup
\end{figure}

Table~\ref{table:timesVsGrid} summarizes the total time to compute the last estimate $v^p_{\bar n_p}$ on each grid $\textbf{S}^p$ and its size.
The total computational time grows polynomially with respect to the grid size.
Specifically, the time to construct set-valued dynamics is linear with respect to the grid size while the time to execute value iteration grows polynomially.
As a result, most of the total computational time is spent on constructing set-valued dynamics on the coarse grids while the time to execute value iteration dominates on the fine grids.

\end{subsection}

\end{section}

\begin{section}{Proofs}\label{sect:proofs}
In this section, detailed proofs of theoretic results in Section \ref{sect:analysis} are provided.
\begin{subsection}{Preliminary}\label{sect:preliminary}
In this subsection, some preliminary properties of Hausdorff distance are introduced.

The following lemma shows the union of two expanded sets is the expansion of their unions.
\begin{lemma}\label{thm:unitBall}
Given two sets $A, B\subseteq\mathcal X$ and $\eta>0$, the following holds
$
(A+\eta\mathcal B)\cup(B+\eta\mathcal B)=(A\cup B)+\eta\mathcal B.
$
\end{lemma}
\begin{proof}
Fix $x\in(A+\eta\mathcal B)\cup(B+\eta\mathcal B)$.
Then $x\in A+\eta\mathcal B$ or $x\in B+\eta\mathcal B$.
We focus on the first case. 
It follows from the definition of $\mathcal B$ that $\exists y\in A$ s.t. $\|x-y\|\leq\eta$.
Since $y\in A\subseteq A\cup B$, we have $x\in(A\cup B)+\eta\mathcal B$.
Similar conclusion can be drawn for the case when $x\in B+\eta\mathcal B$.
Hence we have $(A+\eta\mathcal B)\cup(B+\eta\mathcal B)\subseteq(A\cup B)+\eta\mathcal B$.

Now consider $x\in(A\cup B)+\eta\mathcal B$.
It again follows from the definition of $\mathcal B$ that $\exists y\in A\cup B$ s.t. $\|x-y\|\leq\eta$.
If $y\in A$, we have $x\in A+\eta\mathcal B\subseteq(A+\eta\mathcal B)\cup(B+\eta\mathcal B)$; if $y\in B$, we have $x\in B+\eta\mathcal B\subseteq(A+\eta\mathcal B)\cup(B+\eta\mathcal B)$.
By either way, we have $(A\cup B)+\eta\mathcal B\subseteq(A+\eta\mathcal B)\cup(B+\eta\mathcal B)$.

Therefore, the relationship in the lemma statement is proven in both directions and the lemma is hence proven.
\end{proof}

The following lemma compares set distances given their set inclusion relationships.
\begin{lemma}\label{thm:setInequality}
Given four nonempty compact sets $A\subseteq B$ and $C\subseteq D$, the following relationships hold:
\begin{equation}\label{eq:setInequality1}
\begin{split}
&d_H(A\cup D, B)\leq d_H(D, B)\leq\max\{d_H(A, D), d_H(B, C)\}. 
\end{split}
\end{equation}
\end{lemma}
\begin{proof}
First we proceed to the proof of the first inequality of \eqref{eq:setInequality1}.
Take $\delta'>\delta\triangleq d_H(D, B)$, then it holds that
$B\subseteq D+\delta'\mathcal B,\quad  D\subseteq B+\delta'\mathcal B.$
The first relationship implies $B\subseteq D\cup A+\delta'\mathcal B$.
Since $A\subseteq B$, the second relationship implies $A\cup D\subseteq B+\delta'\mathcal B$.
Therefore $d_H(A\cup D, B)\leq\delta'$. 
Since this holds for all $\delta'>\delta$, $d_H(A\cup D, B)\leq \delta=d_H(D, B)$.

By Theorem 7.1.1 in \cite{sternberg2010dynamical}, the second inequality holds.
\end{proof}

The next lemma shows the triangle inequality holds for $d_H$.

\begin{lemma}\label{thm:TriangleInequality}
Given three set-valued maps $g^l:\mathcal X\rightrightarrows[0,1]^N$, $g^l(x)$ is compact for all $x\in\mathcal X$, $ l\in\{1,2,3\}$. It holds that $d_{\mathcal X}(g^1,g^2)\leq d_{\mathcal X}(g^1,g^3)+d_{\mathcal X}(g^3, g^2)$.
\end{lemma}
\begin{proof}
Since $g^l(x)$ is compact for $l\in\{1, 2, 3\}$, it follows from page 144 in \cite{sternberg2010dynamical} that for any $x\in\mathcal X$, $d_H(g^1(x), g^2(x))\leq d_H(g^1(x), g^3(x))+d_H(g^3(x), g^2(x))$.
Take supremum over $\mathcal X$ on both sides and we have
\begin{align*}
d_{\mathcal X}(g^1, g^2)=&\sup_{x\in\mathcal X}d_H(g^1(x), g^2(x))\\
\leq&\sup_{x\in\mathcal X}[d_H(g^1(x), g^3(x))+d_H(g^3(x), g^2(x))].
\end{align*}
In addition, splitting the sum on the right-hand side yields
\begin{align*}
d_{\mathcal X}(g^1, g^2)
\leq&\sup_{x\in\mathcal X}d_H(g^1(x), g^3(x))+\sup_{x\in\mathcal X}d_H(g^3(x), g^2(x))\\
=&d_{\mathcal X}(g^1, g^3)+d_{\mathcal X}(g^3,g^2).
\end{align*}
Hence, the lemma is proven.
\end{proof}

Lemma \ref{thm:maximumNormShrinkingPerturbation} reveals that, for two perturbed set-valued maps, the union of images of fewer nodes contributes to larger distance.

\begin{lemma}\label{thm:maximumNormShrinkingPerturbation}
Given two subsets $\mathcal X^1,\mathcal X^2\subseteq\mathcal X$, consider two set-valued maps $g_1, g_2:\mathcal X\rightrightarrows[0, 1]^N$ and perturbation radii  $\eta_l>0$  s.t. $(x+\eta_l\mathcal B)\cap\mathcal X^l\neq\emptyset, \forall x\in\mathcal X, l\in\{1, 2\}$.
The following holds for any set-valued map $Y:\mathcal X\rightrightarrows\mathcal X$ s.t. $Y(x)\neq\emptyset, \forall x\in\mathcal X$:
\begin{equation}\label{eq:shrinking}
\begin{split}
&d_{\mathcal X}(\bigcup_{\tilde x\in(Y(x)+\eta_1\mathcal B)\cap\mathcal X^1}g_1(\tilde x),\bigcup_{\tilde x\in(Y(x)+\eta_2\mathcal B)\cap\mathcal X^2}g_2(\tilde x))\\
\leq&d_{\mathcal X}(\bigcup_{\tilde x\in(x+\eta_1\mathcal B)\cap\mathcal X^1}g_1(\tilde x),\bigcup_{\tilde x\in(x+\eta_2\mathcal B)\cap\mathcal X^2}g_2(\tilde x)).
\end{split}\end{equation}
If $\mathcal X^1=\mathcal X^2\triangleq\bar{\mathcal X}$ and $\eta_1=\eta_2\triangleq\bar\eta$, we have
\begin{equation}\label{eq:shrinking2}
\begin{split}
d_{\mathcal X}(\bigcup_{\tilde x\in(x+\eta\mathcal B)\cap\bar{\mathcal X}}g_1(\tilde x),\bigcup_{\tilde x
\in(x+\eta\mathcal B)\cap\bar{\mathcal X}}g_2(\tilde x))\leq d_{\bar{\mathcal X}}(g_1,g_2).
\end{split}\end{equation}

\end{lemma}

\begin{proof}
We first proceed to prove inequality \eqref{eq:shrinking}.
Let $\delta_1\triangleq d_{\mathcal X}(\bigcup_{\tilde x\in(x+\eta_1\mathcal B)\cap\mathcal X^1}g_1(\tilde x),\allowbreak\bigcup_{\tilde x\in(x+\eta_2\mathcal B)\cap\mathcal X^2}g_2(\tilde x))$ and pick $\delta'_1>\delta_1$.
For any $\hat x\in \mathcal X$, two relationships hold:
\begin{equation}\label{eq:shrinking11}
\begin{split}
\bigcup_{\tilde x\in(\hat x+\eta_1\mathcal B)\cap\mathcal X^1}g_1(\tilde x)\subseteq\bigcup_{\tilde x\in(\hat x+\eta_2\mathcal B)\cap\mathcal X^2}g_2(\tilde x)+\delta'_1\mathcal B, \\
\bigcup_{\tilde x\in(\hat x+\eta_2\mathcal B)\cap\mathcal X^2}g_2(\tilde x)\subseteq\bigcup_{\tilde x\in(\hat x+\eta_1\mathcal B)\cap\mathcal X^1}g_1(\tilde x)+\delta'_1\mathcal B.
\end{split}\end{equation}
Fix $x\in\mathcal X$.
Notice that $\bigcup_{\tilde x\in(Y(x)+\eta_1\mathcal B)\cap\mathcal X^1}g_1(\tilde x)=\bigcup_{\hat x\in Y(x)}\bigcup_{\tilde x\in(\hat x+\eta_1\mathcal B)\cap\mathcal X^1}g_1(\tilde x)$.
Then it follows from the first relationship in \eqref{eq:shrinking11} that
\begin{align*}
\bigcup_{\tilde x\in(Y(x)+\eta_1\mathcal B)\cap\mathcal X^1}g_1(\tilde x)\subseteq\bigcup_{\hat x\in Y(x)}\bigcup_{\tilde x\in(\hat x+\eta_2\mathcal B)\cap\mathcal X^2}(g_2(\tilde x)+\delta'_1\mathcal B).
\end{align*}
By Lemma \ref{thm:unitBall}, the right-hand side of the above relationship becomes $[\bigcup_{\hat x\in Y(x)}\bigcup_{\tilde x\in(\hat x+\eta_2\mathcal B)\cap\mathcal X^2}g_2(\tilde x)]+\delta'_1\mathcal B=\bigcup_{\tilde x\in(Y(x)+\eta_2\mathcal B)\cap\mathcal X^2}g_2(\tilde x)+\delta'_1\mathcal B$.
Therefore, it renders at
\begin{align*}
\bigcup_{\tilde x\in(Y(x)+\eta_1\mathcal B)\cap\mathcal X^1}g_1(\tilde x)\subseteq\bigcup_{\tilde x\in(Y(x)+\eta_2\mathcal B)\cap\mathcal X^2}g_2(\tilde x)+\delta'_1\mathcal B.
\end{align*}
The symmetric relationship holds if $g_1$ and $g_2$ are swapped.
This implies $$d_H(\bigcup_{\tilde x\in(Y(x)+\eta_1\mathcal B)\cap\mathcal X^1}g_1(\tilde x), \bigcup_{\tilde x\in(Y(x)+\eta_2\mathcal B)\cap\mathcal X^2}g_2(\tilde x))\leq\delta'_1.$$
Since this relationship holds for all $\delta'_1>\delta_1$, $d_H(\bigcup_{\tilde x\in(Y(x)+\eta_1\mathcal B)\cap\mathcal X^1}g_1(\tilde x), \bigcup_{\tilde x\in(Y(x)+\eta_2\mathcal B)\cap\mathcal X^2}g_2(\tilde x))\leq\delta_1=\allowbreak d_{\mathcal X}(\bigcup_{\tilde x\in(x+\eta_1\mathcal B)\cap\mathcal X^1}g_1(\tilde x),\bigcup_{\tilde x\in(x+\eta_2\mathcal B)\cap\mathcal X^2}g_2(\tilde x))$.
Taking supremum for all $x\in\mathcal X$, \eqref{eq:shrinking} is established.

Then we proceed to show \eqref{eq:shrinking2}.
Let $\delta_2\triangleq d_{\bar{\mathcal X}}(g_1,g_2)$
and pick $\delta'_2>\delta_2$.
For any $\tilde x\in \bar{\mathcal X}$, two relationships hold:
$g_1(\tilde x)\subseteq g_2(\tilde x)+\delta'_2\mathcal B, \quad g_1(\tilde x)\subseteq g_2(\tilde x)+\delta'_2\mathcal B.$
Focus on the first relationship and take union over $(x+\eta\mathcal B)\cap\bar{\mathcal X}$, we have
$\bigcup_{\tilde x\in(x+\eta\mathcal B)\cap\bar{\mathcal X}}g_1(\tilde x)\subseteq\bigcup_{\tilde x\in(x+\eta\mathcal B)\cap\bar{\mathcal X}}(g_2(\tilde x)+\delta'_2\mathcal B)$.
By following the arguments towards \eqref{eq:shrinking}, the lemma is proven.
\end{proof}

Lemma \ref{thm:lemmaIX5} shows that an exponentially diminishing sequence subject to diminishing perturbations remains diminishing.
\begin{lemma}\label{thm:lemmaIX5}
A sequence $\{a_p\}\subseteq\mathbb R_{\geq0}$ satisfies $a_{p+1}\leq\gamma (a_p+c_p)$, where $\gamma\in[0,1)$, $c_p\geq0, \forall p\geq1$ and $\lim_{p\to+\infty}c_p=0$. Then $\lim_{p\to+\infty}a_p=0$.
\end{lemma}
\begin{proof}
Since $c_p\to0$, $\forall \epsilon>0$, $\exists q>0$ s.t. $\forall p\geq q$, $c_p<\epsilon$.
Fix $\epsilon$ and $q$, we take $r=q+\log_\gamma\frac{\epsilon}{a_q}$ and $p\geq \max\{r, q\}$.
With this, we have
\begin{align*}
0\leq a_p\leq&\gamma^{p-q}a_q+\sum_{l=q}^{p-1}\gamma^{p-l}c_{l}
\leq\gamma^{p-q}a_q+\sum_{l=q}^{p-1}\gamma^{p-l}\epsilon\\
=&\gamma^{p-q}a_q+\frac{\gamma(1-\gamma^{p-q})}{1-\gamma}\epsilon
\leq\epsilon+\frac{1}{1-\gamma}\epsilon.
\end{align*}
The last inequality is due to $p\geq r$.
This is true for any $\epsilon>0$. 
Therefore, $\lim_{p\to+\infty}a_p=0$.
\end{proof}
\end{subsection}

\begin{subsection}{Auxiliary Bellman operator $\hat{\mathbb G}$}
In this subsection, an auxiliary Bellman operator $\hat{\mathbb G}$ is introduced to facilitate the analysis of $\mathbb G$ in Section \ref{sect:propertiesHausdorff}.

\begin{proofOf}{Lemma~\ref{thm:EquivalentSafety}}
Fix $p\geq1$, $x\in\textbf{S}^p$ and $\tilde x\in X^p_E(x)\subseteq\textbf{X}^p$.
Without loss of generality, we denote $\mathcal V^p_G(x)=\{1, \dots, N_p\}$.
It follows from the definition of $X^p_E$ that $\mathcal V^p_G(\tilde x)=\{1, \dots, N_p\}$.
It follows from the definition of $\textbf{S}^p$ that $\exists x'\in\textbf{S}$ s.t. $\|x-x'\|\leq h_p$.
Construct $\tilde x'$ s.t. $\tilde x'_i\triangleq\begin{cases}\tilde x_i, &\text{if }i\in\{1, \dots, N_p\};\\\tilde x_i+x'_i-x_i, &\text{otherwise.}\end{cases}$
Clearly, $\|\tilde x-\tilde x'\|\leq\|x-x'\|\leq h_p$.

Now we proceed to show that $\tilde x'\in\textbf{S}$.
It again follows from the definition of $X^p_E$ that $\tilde x_i=x_i$, $\forall i\in\{N_p+1, \dots, N\}$.
Therefore, we may rewrite $\tilde x'$ as $\tilde x'_i=\begin{cases}\tilde x_i, &\text{if }i\in\{1, \dots, N_p\};\\x'_i, &\text{otherwise.}\end{cases}$
By Assumption \ref{asmp:monoConv} and \ref{asmp:lowerBound}, we have $\forall i\in\{1, \dots, N_p\}$ and $j\in\{1, \dots, N\}$, $\|\tilde x_i-\tilde x_j\|\geq\sigma$.
Since $x'\in\textbf{S}$, it follows from the definition of $\textbf{S}$ that $\|x'_i-x'_j\|\geq\sigma, \forall i\neq j$ and $i,j\in\{N_p+1, \dots, N\}$.
This indicates that $\|\tilde x'_i-\tilde x_j\|\geq\sigma, \forall i\neq j$ and $i,j\in\{N_p+1, \dots, N\}$.
In summary, we have $\|\tilde x'_i-\tilde x_j\|\geq\sigma$ holds for every $i\neq j$, which implies $\tilde x'\in\textbf{S}$.

Since $\|\tilde x-\tilde x'\|\leq h_p$ and $\tilde x\in\textbf{X}^p$, we arrive at $\tilde x\in(\textbf{S}+h_p\mathcal B)\cap\textbf{X}^p=\textbf{S}^p$ and the proof is then finished.
\end{proofOf}

\begin{proofOf}{Lemma \ref{thm:zeroValue}}
The first property follows from the definition of $\tilde X^p$.
For any pair of $i\in\mathcal V^G_p(x)$ and $\tilde x\in\tilde X^p(x)$, it holds that $\tilde x_i=x_i$, then $i\in\mathcal V^G_p(\tilde x)$.

Now we proceed to show the second property.
Since both $\epsilon_p$ and $h_p$ are monotonically decreasing, $\forall i\in\mathcal V^G_{p+1}(x)$, $d(x_i,  X_i^G)\leq M_i\epsilon_{p+1}+h_{p+1}<M_i\epsilon_p+h_p$.
It follows from the definition of $\mathcal V^G_p$ that $i\in\mathcal V^G_p(x)$.
Then the second property is proven.

We are now in a position to prove the third property.
Throughout the rest of the proof, given any $p\geq1$, $n\geq0$ and $x\in\textbf{S}^p$, define a value in $v^p_n(x)$ by $\tau^{p, n}\in v^p_n(x)$.
The $i$-th element of $\tau^{p,n}$ is denoted by $\tau^{p,n}_i$.
The grid index $p$ in $\tau^{p,n}$ may be omitted when omission causes no ambiguity.
The proof is based on induction on $p$.
Denote the induction hypothesis for $p$ by $H(p)$ as $\tau^{p,n}_i=0$ for any $x\in\textbf{S}^p, i\in\mathcal V^G_p(x), 0\leq n\leq\bar n_p$ and $\tau^{p,n}\in v^p_n(x)$.

For $p=1$, fix $x\in\textbf{S}^p$ and $i\in\mathcal V^G_p(x)$ and take $n=0$. 
Since $\textbf{X}^0=\emptyset$, $v^1_0(x)=V^0(x)$.
It follows from \eqref{eq:interpolationFunction} that $\tau^0_i=0$ for every $\tau^0\in v^1_0(x)$.
Moreover, $\tilde T_i(x_i)=0$ and $\tilde X_i(x_i)=x_i$.
Now we adopt induction on $n$ to prove that $\tau^{1, n}_i=0$ for all $x\in\textbf{S}^1\subseteq\textbf{X}^1$, $i\in\mathcal V^G_1(x)$ and $0\leq n\leq\bar n_p$.
For $n=0$, it has been proven.
Assume it holds up to $0\leq n\leq\bar n_p$. 
Then $\tau^{1, n}_i+\tilde T_i(x_i)-\tau^{1, n}_i\tilde T_i(x_i)=\tau^{1, n}_i=0$ holds for any $x\in\textbf{S}^1$, $i\in\mathcal V^G_1(x)$ and $\tau^{1, n}\in v^p_n(x)$.
Therefore, it follows from \eqref{eq:krubellman} that $\tau^{1, n+1}_i=0$.

Assume $H(p)$ holds and let us consider $p+1$.
Fix $x\in\textbf{S}^{p+1}$ and $i\in\mathcal V^G_{p+1}(x)$.
By the second property of this lemma, $i\in\mathcal V^G_p(x)$.
Take $\tau^0\in v^{p+1}_0(x)$. 
If $x\in\textbf{X}^p$, that is, $x\in\textbf{S}^{p+1}\cap\textbf{X}^p=(\textbf{S}+h_{p+1}\mathcal B)\cap\textbf{X}^{p+1}\cap\textbf{X}^p\subseteq\textbf{S}^p$, we have $v^{p+1}_0(x)=\tilde v^p(x)=v^p_{\bar n_p}(x)$.
Therefore, $\forall\tau^0\in v^{p+1}_0(x)$, it follows from $H(p)$ that $\tau^0_i=0$. 
If $x\in\textbf{S}^{p+1}\setminus\textbf{X}^p$, $v^{p+1}_0(x)=\bigcup_{\tilde x\in X^{p+1}_E(x)}\tilde v^p(\tilde x)$.
Notice that when $x\in\textbf{S}^{p+1}$, it follows from Lemma \ref{thm:EquivalentSafety} that $\tilde x\in\textbf{S}^{p+1}$.
Then if $\tilde x\in\textbf{S}^{p+1}\setminus\textbf{X}^p$, $\tilde v^p(\tilde x)=V^p(\tilde x)$; hence, it follows from the definition of $V^p$ in \eqref{eq:interpolationFunction} that we have $V^p_i(x)=0$.
If $\tilde x\in\textbf{S}^{p+1}\cap\textbf{X}^p\subseteq\textbf{S}^p$, $\tilde v^p(\tilde x)=v^p_{\bar n_p}(\tilde x)$; hence, it follows from $H(p)$ that $\forall\tilde\tau\in\tilde v^p(\tilde x)$, $\tilde\tau_i=0$.
Therefore, $\forall\tilde x\in X^{p+1}_E(x)$ and $\tilde\tau\in\tilde v^p(\tilde x)$, $\tilde\tau_i=0$.
That is to say, we have $\tau^0_i=0$ for $x\in\textbf{S}^{p+1}\setminus\textbf{X}^p$ and $\tau^0\in v^{p+1}_0(x)$.
In summary, $\tau^0_i=0$ for every $x\in\textbf{S}^{p+1}$, $i\in\mathcal V^G_{p+1}(x)$ and $\tau^0\in v^{p-1}_0(x)$.
For $1\leq n\leq\bar n_p$, we follow the arguments for $p=1$ and it holds that $\tau^n_i=0$, $\forall\tau^n\in v^{p+1}_n(x)$.
Then $H(p+1)$ is proven and the proof of the third property is finished.
\end{proofOf}

\begin{proofOf}{Lemma \ref{thm:fixedGridSub}}
Throughout the proof, we adopt the shorthand notation $v\triangleq v^p_n$.
Without loss of generality, let $\mathcal V^G_p(x)=\{1, \dots, N_p\}$ for some $0\leq N_p\leq N$ and $\mathcal V^G_p(x')=\{1, \dots, N_p'\}$ for some $0\leq N_p'\leq N_p$.
Specifically, when $N_p=0$ (resp. $N_p'=0$), $\mathcal V^G_p(x)=\emptyset$ (resp. $\mathcal V^G_p(x')=\emptyset$).

Notice that when $N_p=N$, i.e. all robots are in their goal regions at state $x$, it follows from the third property of Lemma \ref{thm:zeroValue} that $E_{\mathbb G^mv}(x)=E_v(x)=[0, 1]^N\supseteq E_{\mathbb G^mv}(x')$ for any $x'\in X^p_P(x)$, hence the lemma trivially holds.
When $N_p=0$, it holds that $N_p'=0$,  $x=x'$, and the lemma also trivially holds.
In the following proof, we restrict $1\leq N_p\leq N-1$.

The lemma is proven by induction on $m$.
Denote the induction hypothesis for $m$ by $H(m)$.
Then $H(0)$ trivially holds.
Assume $H(m)$ holds and let us consider $m+1$.
It follows from \eqref{eq:epiProfileBellman} that
\begin{align*}
&E_{\mathbb G^{m+1}v}(x)=\Delta\tau(x)+(\textbf{1}-\Delta\tau(x))\circ\bigcup_{\tilde x\in\tilde X^p(x)}E_{\mathbb G^mv}(\tilde x),\\
&E_{\mathbb G^{m+1}v}(x')=\Delta\tau(x')+(\textbf{1}-\Delta\tau(x'))\circ\bigcup_{\tilde x'\in\tilde X^p(x')}E_{\mathbb G^mv}(\tilde x').
\end{align*}
First we focus on the unions on the right-hand side, especially the one-hop neighbors $\tilde X^p$.
\begin{claim}\label{thm:fixedGridSub:1}
For all $\tilde x'\in\tilde X^p(x')$, $\exists\tilde x\in\tilde X^p(x)$ s.t. $\tilde x'\in X^p_P(\tilde x)$.
\end{claim}
\begin{proof}
Fix $\tilde x'\in\tilde X^p(x')$ and define $\tilde x\in\textbf{X}^p$ s.t. $\tilde x_i=\begin{cases}x_i, &\text{if }i\in\{1, \dots, N_p\};\\\tilde x'_i, &\text{otherwise}.\end{cases}$  
We proceed to show $\tilde x\in\tilde X^p(x)=(\prod_{i=1}^N\tilde X^p_i(x_i))\cap\textbf{S}^p$.

For $i\in\{1, \dots, N_p\}$, we have $\tilde X^p_i(x_i)=\{x_i\}$; therefore, $\tilde x_i\in\tilde X^p_i(x_i)$.
For $i\in\{N_p+1, \dots, N\}$, it follows from $x'\in X^p_P(x)$ that  $x_i=x_i'$.
Then $\tilde x_i=\tilde x_i'\in\tilde X^p_i(x'_i)=\tilde X^p_i(x_i)$.
Therefore, $\tilde x\in\prod_{i=1}^N\tilde X^p_i(x_i)$.

Notice that $\tilde x'\in\tilde X^p(x')\subseteq\textbf{S}^p$, thus $\exists \tilde y'\in\textbf{S}$ s.t. $\|\tilde x'-\tilde y'\|\leq h_p$.
Define $\tilde y$ s.t. $\tilde y_i=\begin{cases}x_i, &\text{if }i\in\{1, \dots, N_p\};\\\tilde y'_i, &\text{otherwise.}\end{cases}$
Clearly, $\|\tilde x-\tilde y\|\leq\|\tilde x'-\tilde y'\|\leq h_p$.
Since $\tilde y'\in\textbf{S}$, it holds that $\|\tilde y_i-\tilde y_j\|=\|\tilde y'_i-\tilde y'_j\|\geq\sigma, \forall i,j\in\{N_p+1, \dots, N\}, i\neq j$.
For $i\in\{1, \dots, N_p\}$, we have $\tilde x_i=x_i\in X^G_i+(M_i\epsilon_p+h_p)\mathcal B$.
Then by Assumptions \ref{asmp:monoConv} and \ref{asmp:lowerBound}, we have $\|\tilde y_i-\tilde y_j\|=\|x_i-\tilde y_j\|\geq\sigma, \forall i\in\{1, \dots, N_p\}$ and $j\in\{1, \dots, N\}$.
Therefore $\tilde y\in\textbf{S}$ and $\tilde x\in(\tilde y+h_p\mathcal B)\cap\textbf{X}^p\subseteq\textbf{S}^p$.
Thus, $\tilde x\in\tilde X^p(x)$ is proven.

Now we proceed to show $\tilde x'\in X^p_P(\tilde x)$.
By the first property of Lemma \ref{thm:zeroValue}, $\mathcal V^G_p(\tilde x)\supseteq\mathcal V^G_p(x)$.
Then for $i\in\mathcal V\setminus\mathcal V^G_p(\tilde x)\subseteq\mathcal V\setminus\mathcal V^G_p(x)=\{N_p+1, \dots, N\}$, $\tilde x'_i=\tilde x_i$.
By the definition of $X^p_P$, $\tilde x'\in X^p_P(\tilde x)$ and therefore the claim is proven.
\end{proof}

It follows from Claim \ref{thm:fixedGridSub:1} and $H(m)$ that $\forall \tilde x'\in\tilde X^p(x')$, $\exists\tilde x\in\tilde X^p(x)$ s.t. $E_{\mathbb G^mv}(\tilde x')\subseteq E_{\mathbb G^mv}(\tilde x)$.
Hence, $\bigcup_{\tilde x'\in\tilde X^p(x')}E_{\mathbb G^mv}(\tilde x')\subseteq\bigcup_{\tilde x\in\tilde X^p(x)}E_{\mathbb G^mv}(\tilde x)$ and we have

\begin{equation}\label{eq:11001}
\begin{split}
E_{\mathbb G^{m+1}v}(x')\subseteq\Delta\tau(x')+(\textbf{1}-\Delta\tau(x'))\circ\bigcup_{\mathclap{\tilde x\in\tilde X^p(x)}}E_{\mathbb G^mv}(\tilde x).
\end{split}\end{equation}
Next, we prove the right-hand side of \eqref{eq:11001} is a subset of $E_{\mathbb G^{m+1}v}(x)$.
\begin{claim}\label{thm:fixedGridSub:2}
The following relationship holds:
\begin{align*}
&\Delta\tau(x')+(\textbf{1}-\Delta\tau(x'))\circ\bigcup_{\tilde x\in\tilde X^p(x)}E_{\mathbb G^mv}(\tilde x)\\
\subseteq&\Delta\tau(x)+(\textbf{1}-\Delta\tau(x))\circ\bigcup_{\tilde x\in\tilde X^p(x)}E_{\mathbb G^mv}(\tilde x).
\end{align*}
\end{claim}
\begin{proof}
For any $\tilde x\in \tilde X^p(x)$ and $\tau\in E_{\mathbb G^mv}(\tilde x)$, construct $\hat\tau$ s.t. $\hat\tau_i=\begin{cases}(1-e^{-\kappa_p})+e^{-\kappa_p}\tau_i, &\text{if } i\in\{N_p'+1, \dots, N_p\};\\\tau_i, &\text{otherwise.}\end{cases}$
Since $\hat\tau\succeq\tau$, $\hat\tau\in E_{\mathbb G^mv}(\tilde x)$.
Recall $\Delta\tau(x')=\Psi(\mathcal E(T^p(x')))$ and $\Delta\tau_i(x')=\begin{cases}0, &\text{if }i\in\{1, \dots, N_p'\};\\1-e^{-\kappa_p}, &\text{otherwise.}\end{cases}$
Then the following holds:
\begin{align*}
&\Delta\tau(x')+(\textbf{1}-\Delta\tau(x'))\circ\tau\\
=&\begin{bmatrix}\textbf{0}_{N_p'}\\(1-e^{-\kappa_p})\textbf{1}_{N_p-N_p'}\\(1-e^{-\kappa_p})\textbf{1}_{N-N_p}\end{bmatrix}+\begin{bmatrix}\textbf{1}_{N_p'}\\e^{-\kappa_p}\textbf{1}_{N_p-N_p'}\\e^{-\kappa_p}\textbf{1}_{N-N_p}\end{bmatrix}\circ\tau\\
=&\begin{bmatrix}\textbf{0}_{N_p'}\\\textbf{0}_{N_p-N_p'}\\(1-e^{-\kappa_p})\textbf{1}_{N-N_p}\end{bmatrix}+\begin{bmatrix}\textbf{1}_{N_p'}\\\textbf{1}_{N_p-N_p'}\\e^{-\kappa_p}\textbf{1}_{N-N_p}\end{bmatrix}\circ\hat\tau\\
=&\Delta\tau(x)+(\textbf{1}-\Delta\tau(x))\circ\hat\tau.
\end{align*}
In summary, for every $\tilde x\in\tilde X^p(x)$ and $\tau\in E_{\mathbb G^mv}(\tilde x)$, there is $\hat\tau\in E_{\mathbb G^mv}(\tilde x)$ s.t. $\Delta\tau(x')+(\textbf{1}-\Delta\tau(x'))\circ\tau=\Delta\tau(x)+(\textbf{1}-\Delta\tau(x))\circ\hat\tau$.
Hence the proof of the claim is finished.
\end{proof}

Together with \eqref{eq:11001}, Claim \ref{thm:fixedGridSub:2} indicates 
$E_{\mathbb G^{m+1}v}(x')\subseteq\Delta\tau(x)+(\textbf{1}-\Delta\tau(x))\circ\bigcup_{\tilde x\in\tilde X^p(x)}E_{\mathbb G^mv}(\tilde x)=E_{\mathbb G^{m+1}v}(x).$
Then $H(m+1)$ holds and the lemma is proven.
\end{proofOf}

\begin{proofOf}{Lemma \ref{thm:multiGridSub}}
Fix a pair of $x\in\textbf{S}^p$ and $x'\in X^p_P(x)$.
Without loss of generality, let $\mathcal V^G_p(x)=\{1, \dots, N_p\}$ and $\mathcal V^G_p(x')=\{1, \dots, N'_p\}$ for some $0\leq N_p'\leq N_p\leq N$.
Specifically, when $N_p=0$ (resp. $N_p'=0$), $\mathcal V^G_p(x)=\emptyset$ (resp. $\mathcal V^G_p(x')=\emptyset$).

Notice that when $N_p=N$, by the third property of Lemma \ref{thm:zeroValue}, $E_{v^p_n}(x')\subseteq [0, 1]^N=E_{v^p_n}(x)$ for any $x'\in X^p_P(x)$ and $0\leq n\leq\bar n_p$, and the lemma trivially holds.
When $N_p=0$, it holds that $N_p'=0$ and $x=x'$, and the lemma also trivially holds.
In the following proof, we restrict $1\leq N_p\leq N-1$.

The lemma is proven by induction on $p$.
Denote the induction hypothesis for $p$ by $H(p)$ as $ E_{v^p_n}(x')\subseteq E_{v^p_n}(x)$ holds for all $x\in\textbf{S}^p$, $x'\in X^p_P(x)$ and $0\leq n\leq\bar n_p$.

For $p=1$, it follows from the definition of $V^0$ that $E_{v^1_0}(x)=E_{V^0}(x)=[0,1]^{N_1}\times\{1\}^{N-N_1}\supseteq [0,1]^{N'_1}\times\{1\}^{N-N'_1}=E_{v^1_0}(x')$, where $E_{V^1}(x)\triangleq(\{V^1(x)\}+\mathbb R^N_{\geq0})\cap[0,1]^N$ is the epigraphical profile of interpolation function $V^1(x)$.
Since this holds for every $x\in\textbf{S}^1$ and $x'\in X^1_P(x)$, it follows from Lemma \ref{thm:fixedGridSub} that $E_{v^1_n}(x)\supseteq E_{v^1_n}(x')$ holds for all $0\leq n\leq\bar n_p$.
Hence $H(1)$ holds.

Assume $H(p)$ holds for $p\geq1$.
For $p+1$, pick a pair of $x\in\textbf{S}^{p+1}$ and $x'\in X^{p+1}_P(x)$.
There are four cases:
\begin{itemize}
\item Case 1: $x, x'\in\textbf{S}^p$;
\item Case 2: $x\in\textbf{S}^{p+1}\setminus\textbf{S}^p$ and $x'\in\textbf{S}^p$;
\item Case 3: $x\in\textbf{S}^p$ and $x'\in\textbf{S}^{p+1}\setminus\textbf{S}^p$;
\item Case 4: $x, x'\in\textbf{S}^{p+1}\setminus\textbf{S}^p$.
\end{itemize}
\begin{claim}
$E_{v^{p+1}_0}(x')\subseteq E_{v^{p+1}_0}(x)$ holds for Case 1.
\end{claim}
\begin{proof}
It follows from the defintions of $v^{p+1}_0$ and $\tilde v^p$ that $E_{v^{p+1}_0}(x)=E_{\tilde v^p}(x)=E_{v^p_{\bar n_p}}(x)$ and $E_{v^{p+1}_0}(x')=E_{\tilde v^p}(x')=E_{v^p_{\bar n_p}}(x')$.
By $H(p)$, we have $E_{v^{p+1}_0}(x')=E_{v^p_{\bar n_p}}(x') \subseteq E_{v^p_{\bar n_p}}(x)=E_{v^{p+1}_0}(x)$. 
Then $E_{v^{p+1}_0}(x')\subseteq E_{v^{p+1}_0}(x)$ holds for Case 1.
\end{proof}

\begin{claim}\label{thm:multiGridSub:Case2}
$E_{v^{p+1}_0}(x')\subseteq E_{v^{p+1}_0}(x)$ holds for Case 2.
\end{claim}

\begin{proof}
Notice that $E_{v^{p+1}_0}(x)=\bigcup_{\tilde x\in X^{p+1}_E(x)}E_{\tilde v^p}(\tilde x)$ and $E_{v^{p+1}_0}(x')=E_{v^p_{\bar n_p}}(x')$.
Now we are going to construct $\tilde x\in X^{p+1}_E(x)$ s.t. $\tilde x\in\textbf{S}^p$ and prove that $E_{v^p_{\bar n_p}}(x')\subseteq E_{\tilde v^p}(\tilde x)$.
For $i\in\{N_{p+1}+1, \dots, N\}$, let $\tilde x_i=x_i$; for $i\in\{1, \dots, N_{p+1}\}$, pick $\tilde x_i\in X^G_i\cap X^p_i$.
Since $x'\in\textbf{S}^p$ and $x'\in X^{p+1}_P(x)$, we have $\tilde x_i=x_i=x'_i\in X^p_i, \forall i\in\{N_{p+1}+1, \dots, N\}$.
Therefore, we have $\tilde x\in X^{p+1}_E(x)$ and $\tilde x\in\textbf{S}^p$.

Then we show that $x'\in X^p_P(\tilde x)$.
Consider $j\in\{1, \dots, N\}$ s.t. $\tilde x_j\notin X^G_j+(M_j\epsilon_p+h_p)\mathcal B$.
By Assumption \ref{asmp:monoConv}, we have $\tilde x_j\notin X^G_j+(M_j\epsilon_{p+1}+h_{p+1})\mathcal B$.
The fact that $\tilde x\in X^{p+1}_E(x)$ implies $\tilde x_i\in X^G_i+(M_i\epsilon_{p+1}+h_{p+1})\mathcal B, \forall i\in\{1,\dots,  N_{p+1}\}$.
Therefore, $j\in\{N_{p+1}+1, \dots, N\}$ and hence $\tilde x_j=x_j$.
Moreover, since $x'\in X^{p+1}_P(x)$, it follows from $j\in\{N_{p+1}+1, \dots, N\}$ that $x'_j=x_j=\tilde x_j$.
This holds for every $j\in\mathcal V$ s.t. $\tilde x_j\notin X^G_j+(M_j\epsilon_p+h_p)\mathcal B$.
By the definition of $X^p_P$, we conclude that $x'\in X^p_P(\tilde x)$.

By utilizing $H(p)$, it follows from $x'\in X^p_P(\tilde x)$ that $E_{v^{p+1}_0}(x')=E_{v^p_{\bar n_p}}(x')\subseteq E_{v^p_{\bar n_p}}(\tilde x)$.
Since $\tilde x\in\textbf{S}^p$, $E_{v^p_{\bar n_p}}(\tilde x)=E_{\tilde v^p}(\tilde x)$.
Moreover, it follows from $\tilde x\in X^{p+1}_E(x)$ that $E_{\tilde v^p}(\tilde x)\subseteq \bigcup_{\tilde x\in X^{p+1}_E(x)}E_{\tilde v^p}(\tilde x)=E_{v^{p+1}_0}(x)$.
In summary, $E_{v^{p+1}_0}(x')\subseteq E_{v^{p+1}_0}(x)$.
Then $E_{v^{p+1}_0}(x')\subseteq E_{v^{p+1}_0}(x)$ holds for Case 2.
\end{proof}

\begin{claim}\label{thm:multiGridSub:Case3}
$E_{v^{p+1}_0}(x')\subseteq E_{v^{p+1}_0}(x)$ holds for Case 3.
\end{claim}
\begin{proof}
Notice $E_{v^{p+1}_0}(x)=E_{\tilde v^p}(x)=E_{v^p_{\bar n_p}}(x)$ and $E_{v^{p+1}_0}(x')=\bigcup_{\tilde x'\in X^{p+1}_E(x')}E_{\tilde v^p}(\tilde x')$.
For each $\tilde x'\in X^{p+1}_E(x')$, two cases arise: 

Case 3.1, $\tilde x'\in\textbf{S}^{p+1}\setminus\textbf{S}^p$:
Then $\tilde v^p(\tilde x')=\{V^p(\tilde x')\}$. 
Since $\tilde x'\in X^{p+1}_E(x')$, $E_{\tilde v^p}(\tilde x')=E_{V^p}(\tilde x')=[0, 1]^{N_{p+1}'}\times\{1\}^{N-N_{p+1}'}$. 
It follows from $N_{p+1}'\leq N_{p+1}$ that $[0, 1]^{N_{p+1}'}\times\{1\}^{N-N_{p+1}'}\subseteq[0, 1]^{N_{p+1}}\times\{1\}^{N-N_{p+1}}$.
By the third property of Lemma \ref{thm:zeroValue}, $[0, 1]^{N_p}\times\{1\}^{N-N_p}\subseteq E_{v^p_{\bar n_p}}(x)$.
Therefore, we have $E_{\tilde v^p}(\tilde x')\subseteq E_{v^{p+1}_0}(x)$.

Case 3.2, $\tilde x'\in\textbf{S}^p$: Then $E_{\tilde v^p}(\tilde x')=E_{v^p_{\bar n_p}}(\tilde x')$.
Since $\tilde x'\in X^{p+1}_E(x')$, thus for any $i\in\{N_{p+1}+1, \dots, N\}$, we have $\tilde x'_i=x'_i=x_i\notin X^G_i+(M_i\epsilon_{p+1}+h_{p+1})\mathcal B$.
Using the second property of Lemma \ref{thm:zeroValue}, we have $\tilde x'_i=x_i, \forall i\in\mathcal V\setminus\mathcal V^G_p(x)\subseteq\mathcal V\setminus\mathcal V^G_{p+1}(x)=\{N_{p+1}+1, \dots, N\}$.
Therefore, $\tilde x'\in X^p_P(x)$.
By $H(p)$, $E_{v^p_{\bar n_p}}(\tilde x')\subseteq E_{v^p_{\bar n_p}}(x)=E_{v^{p+1}_0}(x)$.
That is, $E_{\tilde v^p}(\tilde x')\subseteq E_{v^{p+1}_0}(x)$.

In summary, $\forall\tilde x'\in X^{p+1}_E(x')$, $E_{\tilde v^p}(\tilde x')\subseteq E_{v^{p+1}_0}(x)$.
Then $E_{v^{p+1}_0}(x')\subseteq E_{v^{p+1}_0}(x)$ holds for Case 3.
\end{proof}

\begin{claim}
$E_{v^{p+1}_0}(x')\subseteq E_{v^{p+1}_0}(x)$ holds for Case 4.
\end{claim}
\begin{proof}
Then $E_{v^{p+1}_0}(x)=\bigcup_{\tilde x\in X^{p+1}_E(x)}E_{\tilde v^p}(\tilde x)$ and $E_{v^{p+1}_0}(x')=\bigcup_{\tilde x'\in X^{p+1}_E(x')}E_{\tilde v^p}(\tilde x')$.
Consider $x'$, and there are two scenarios:

Case 4.1: $\exists j\in\{N_{p+1}'+1, \dots, N\}$ s.t. $x'_j\in X^{p+1}_j\setminus X^p_j$.
Then $\forall\tilde x'\in X^{p+1}_E(x')$, we have $\tilde x'_j=x'_j\in X^{p+1}_j\setminus X^p_j$.
This indicates $\tilde x'\in\textbf{S}^{p+1}\setminus\textbf{S}^p$.
Following Case 3.1, we have $E_{v^{p+1}_0}(x')=\bigcup_{\tilde x'\in X^{p+1}_E(x')}E_{\tilde v^p}(\tilde x')=\bigcup_{\tilde x'\in X^{p+1}_E(x')}([0, 1]^{N_{p+1}'}\times\{1\}^{N-N_{p+1}'})=[0, 1]^{N_{p+1}'}\times\{1\}^{N-N_{p+1}'}$ and $[0,1]^{N_{p+1}}\times\{1\}^{N-N_{p+1}}\subseteq E_{v^{p+1}_0}(x)$.
Notice that $N_{p+1}'\leq N_{p+1}$.
Then $E_{v^{p+1}_0}(x')\subseteq E_{v^{p+1}_0}(x)$.
Hence $E_{v^{p+1}_0}(x')\subseteq E_{v^{p+1}_0}(x)$ holds for Case 4.1.

Case 4.2: $\forall j\in\{N_{p+1}'+1, \dots, N\}, x'_j\in X^p_j$ while $\exists j\in\{1, \dots, N_{p+1}'\}$ s.t. $x'_j\in X^{p+1}_j\setminus X^p_j$.
We show that $\exists \tilde x\in X^{p+1}_E(x)$ s.t. $\tilde x\in\textbf{S}^p$.
Since $x'\in X^{p+1}_P(x)$, $x_i=x'_i\in X^p_i\subseteq X^{p+1}_i, \forall i\in\{N_{p+1}+1, \dots, N\}\subseteq\{N_{p+1}'+1, \dots, N\}$.
By picking $\tilde x_i\in X^G_i\cap X^p_i$ for $i\in\{1, \dots, N_{p+1}\}$ and $\tilde x_i=x_i, \forall i\in\{N_{p+1}+1, \dots, 1\}$, we have $\tilde x\in X^{p+1}_E(x)$.
In addition, since $x\in\textbf{S}^{p+1}$, we have $\exists y\in\textbf{S}$ s.t. $\|x-y\|\leq h_{p+1}\leq h_p$.
Define $\tilde y$ s.t. $\tilde y_i=\tilde x_i, \forall i\in\{1, \dots, N_{p+1}\}$ and $\tilde y_i=y_i, \forall i\in\{N_{p+1}+1, \dots, N\}$.
Then $\|\tilde y-\tilde x\|=\sqrt{\sum_{i=N_{p+1}+1}^N(y_i-x_i)^2}\leq\|y-x\|\leq h_{p+1}$.
Then $\forall i,j\in\{N_{p+1}+1, \dots, N\}$, $\|\tilde y_i-\tilde y_j\|=\|y_i-y_j\|\geq\sigma$.
Moreover, it follows from Assumptions \ref{asmp:monoConv} and \ref{asmp:lowerBound} that $\forall i\in\{1, \dots, N_{p+1}\}$ and $j\in\{1, \dots, N\}$, $\|\tilde y_i-\tilde y_j\|=\|\tilde x_i-\tilde x_j\|\geq\sigma$.
This indicates that $\tilde y\in\textbf{S}$ and hence $\tilde x\in\textbf{S}^p$.
By the definition of $X^{p+1}_E$, $X^{p+1}_E(x)=X^{p+1}_E(\tilde x)$.
This means we can replace $X^{p+1}_E(x)$ with $X^{p+1}_E(\tilde x)$ and degenerate the current case to Case 3.
Then by Claim \ref{thm:multiGridSub:Case3}, $H(p+1)$ holds for $x'\in\textbf{S}^{p+1}\setminus\textbf{S}^p$ and $\tilde{x}\in\textbf{S}^p$.
Thus, $E_{v^{p+1}_0}(x')=\bigcup_{\tilde x'\in X^{p+1}_E(x')}E_{\tilde v^p}(\tilde x')\subseteq E_{v^p_{\bar n_p}}(\tilde x)\subseteq \bigcup_{\tilde x\in X^{p+1}_E(x)} E_{\tilde v^p}(\tilde x)=E_{v^{p+1}_0}(x)$.

By the two cases discussed, $E_{v^{p+1}_0}(x')\subseteq E_{v^{p+1}_0}(x)$ holds for Case 4.
\end{proof}

By the four cases above, $E_{v^{p+1}_0}(x')\subseteq E_{v^{p+1}_0}(x)$ holds for all $x\in\textbf{S}^{p+1}$ and $x'\in X^{p+1}_P(x)$.
By Lemma \ref{thm:fixedGridSub}, $H(p+1)$ is proven.
Then the lemma is established.
\end{proofOf}

\begin{proofOf}{Corollary \ref{thm:XEEquivalence}}
Fix $p\geq1$, $0\leq n\leq \bar n_p$, $x\in\textbf{S}^p$ and $x'\in X^p_E(x)$.
On one hand, by Lemma \ref{thm:multiGridSub}, $\forall x'\in X^p_E(x)\subseteq X^p_P(x)$, $E_{v^p_n}(x')\subseteq E_{v^p_n}(x)$.
On the other hand, $x\in X^p_E(x')$; thus again by Lemma \ref{thm:multiGridSub} we also have $E_{v^p_n}(x)\subseteq E_{v^p_n}(x')$, which indicates that $E_{v^p_n}(x')= E_{v^p_n}(x)$.

Since Lemma \ref{thm:fixedGridSub} holds for every $m\geq0$, the above proof can be directly extended to $\mathbb G^mv^p_0$ for any $m>\bar n_p$.
By Theorem \ref{thm:kru1}, $v^p_\infty$ exists and $v^p_\infty(x)=\mathrm{Lim}_{m\to+\infty}\mathbb G^mv^p_0(x)$ for any $x\in\textbf{S}^p$.
Hence, the equivalence $E_{\mathbb G^mv^p_0}(x)=E_{\mathbb G^mv^p_0}(x')$ can be further extended to $E_{v^p_\infty}(x)=E_{v^p_\infty}(x)$ by taking $m\to+\infty$.
\end{proofOf}

\begin{proofOf}{Lemma \ref{thm:tildeHatEquivalence}}
We fix $p\geq1$, $0\leq n\leq\bar n_p$ and $x\in\textbf{S}^p$.
Recall that
\begin{align*}
E_{\mathbb Gv^p_n}(x)=\Delta\tau(x)+(\textbf{1}-\Delta\tau(x))\circ\bigcup_{\tilde x\in\tilde X^p(x)}E_{v^p_n}(\tilde x),\\
E_{\hat{\mathbb G}v^p_n}(x)=\Delta\tau(x)+(\textbf{1}-\Delta\tau(x))\circ\bigcup_{\hat x\in\hat X^p(x)}E_{v^p_n}(\hat x).
\end{align*}
The difference of $\mathbb G$ and $\hat{\mathbb G}$ solely depends on $\tilde X^p$ and $\hat X^p$.
The proof of either of the two equivalences automatically proves the other.

By the definitions of $\tilde X^p(x)$ and $\hat X^p(x)$, $\hat X^p_i(x_i)=\tilde X^p_i(x_i), \forall i\in\mathcal V\setminus\mathcal V^G_p(x)$.
Therefore, $\forall\hat x\in\hat X^p(x)$, $\exists\tilde x\in\prod_{i\in\mathcal V}\tilde X^p_i(x)$ s.t. $\hat x_i=\tilde x_i, \forall i\in\mathcal V\setminus\mathcal V^G_p(x)$.
It follows from the definitions of $\tilde X^p_i$ and $\mathcal V^G_p$ that $\tilde x_i=x_i\in X^G_i+(M_i\epsilon_p+h_p)\mathcal B, \forall i\in\mathcal V^G_p(x)$.
By Assumptions \ref{asmp:monoConv} and \ref{asmp:lowerBound}, we see that $\|\tilde x_i-\tilde x_j\|\geq\sigma$, $\forall i\in\mathcal V^G_p(x)$ and $j\in\mathcal V$ s.t. $i\neq j$.
It follows from $\tilde x_i=\hat x_i, \forall i\in\mathcal V\setminus\mathcal V^G_p(x)$ and $\hat x\in\hat X^p(x)\subseteq\textbf{S}^p$ that $\|\tilde x_i-\tilde x_j\|\geq\sigma$, $\forall i,j\in\mathcal V\setminus\mathcal V^G_p(x)$ s.t. $i\neq j$.
This indicates $\tilde x\in\textbf{S}^p$.
Therefore, it follows from the definition of $\tilde X^p$ that $\tilde x\in \tilde X^p(x)$.
Since $\mathcal V^G_p(\tilde x)=\mathcal V^G_p(x)$, then $\hat x\in X^p_P(\tilde x)$.
By Lemma \ref{thm:multiGridSub}, $E_{v^p_n}(\hat x)\subseteq E_{v^p_n}(\tilde x)$.
We see that this holds for all $\hat x\in\hat X^p(x)$, then $\bigcup_{\hat x\in\hat X^p(x)}E_{v^p_n}(\hat x)\subseteq \bigcup_{\tilde x\in\tilde X^p(x)}E_{v^p_n}(\tilde x)$.

In addition, $\tilde X^p(x)\subseteq\hat X^p(x)$, we have  $\bigcup_{\tilde x\in\tilde X^p(x)}E_{v^p_n}(\tilde x)\subseteq\bigcup_{\hat x\in\hat X^p(x)}E_{v^p_n}(\hat x)$.
It is concluded that $\bigcup_{\tilde x\in\tilde X^p(x)}E_{v^p_n}(\tilde x)=\bigcup_{\hat x\in\hat X^p(x)}E_{v^p_n}(\hat x)$ and $E_{\mathbb Gv^p_n}(x)=E_{\hat{\mathbb G}v^p_n}(x)$.
Since this holds for every $p\geq 0$, $0\leq n\leq n_p$ and $x\in\textbf{S}^p$, the first part of the lemma is proven.

Since Lemma \ref{thm:fixedGridSub} holds for every $m\geq0$, the above proof can be directly extended to $\mathbb G^mv^p_0$ for any $m>\bar n_p$.
By Theorem \ref{thm:kru1}, the fixed point $v^p_\infty$ exists and $v^p_\infty(x)=\mathrm{Lim}_{m\to+\infty}\mathbb G^mv^p_0(x)$ for any $x\in\textbf{S}^p$.
By the equivalence of $\mathbb G$ and $\hat{\mathbb G}$, we have $v^p_\infty(x)=\mathrm{Lim}_{m\to+\infty}\hat{\mathbb G}^mv^p_0(x)$ and $\hat{\mathbb G}v^p_\infty=v^p_\infty=\mathbb Gv^p_\infty$.
Therefore, $E_{\mathbb Gv^p_\infty}(x)=E_{\hat{\mathbb G}v^p_\infty}(x)$ and the second conclusion is proven.
\end{proofOf}

\end{subsection}

\begin{subsection}{Contraction property of $\mathbb G$}

In this subsection, Theorem \ref{thm:ContractionProperty} shows that the transformed Bellman operator $\mathbb G$ in \eqref{eq:krubellman} is contractive with factor $e^{-\kappa_p}$.

\begin{proofOf}{Lemma \ref{thm:Lemma9}}
We first consider $x\in\textbf{X}\setminus\textbf{S}$.
Since $\textbf{S}$ is closed and $\alpha_p$ is monotonically decreasing, then there exists $q>0$ s.t. $\forall p\geq q$, $(x+\alpha_p\mathcal B)\cap\textbf{S}^{p-1}=(x+\alpha_p\mathcal B)\cap(\textbf{S}+h_{p-1}\mathcal B)\cap\textbf{X}^{p-1}=\emptyset$.
This renders at $\bigcup_{\tilde x\in(x+\alpha_p\mathcal B)\cap\textbf{X}^p}E_{\mathbb P\tilde v^{p-1}_\infty}(\tilde x)=\{\textbf{1}_N\}$.
In addition, it also indicates that $(x+\alpha_p\mathcal B)\cap\textbf{S}^p=\emptyset$.
This renders at $\bigcup_{\tilde x\in(x+\alpha_p\mathcal B)\cap\textbf{X}^p}E_{v^p_\infty}(\tilde x)=\{\textbf{1}_N\}$.
Therefore, we have $b_p(x)=0$.
This holds for all $x\in\textbf{X}\setminus\textbf{S}$,

Then we fix $x\in\textbf{S}$.
The following shorthand notations are used throughout the proof:
\begin{align*}
&A^p_{11}(x)\triangleq\bigcup_{\tilde x\in(x+\alpha_p\mathcal B)\cap(\textbf{S}^p\setminus\textbf{X}^{p-1})}\bigcup_{\tilde x'\in X^p_E(\tilde x)\setminus\textbf{X}^{p-1}}E_{\tilde v_\infty^{p-1}}(\tilde x'),\\
&A^p_{12}(x)\triangleq\bigcup_{\tilde x\in(x+\alpha_p\mathcal B)\cap(\textbf{S}^p\setminus\textbf{X}^{p-1})}\bigcup_{\tilde x'\in X^p_E(\tilde x)\cap\textbf{X}^{p-1}}E_{\tilde v_\infty^{p-1}}(\tilde x'),\\
&A^p_{13}(x)\triangleq\bigcup_{\tilde x\in(x+\alpha_p\mathcal B)\cap(\textbf{X}^p\setminus(\textbf{S}^p\cup\textbf{X}^{p-1}))}E_{\tilde v_\infty^{p-1}}(\tilde x'),\\
&A^p_2(x)\triangleq\bigcup_{\tilde x\in(x+\alpha_p\mathcal B)\cap\textbf{X}^{p-1}}E_{\mathbb P\tilde v_\infty^{p-1}}(\tilde x),\\
&B^p(x)\triangleq\bigcup_{\tilde x\in(x+\alpha_p\mathcal B)\cap\textbf{X}^p}E_{v_\infty^p}(\tilde x).
\end{align*}
We drop the dependency of the above notations on $x$ for notational simplicity.
Now we are going to simplify $A^p_{11}$, $A^p_{12}$, $A^p_{13}$ and $A^p_2$.
From the definitions of $V^p$ and $\tilde v^{p-1}_\infty$, the following hold:
\begin{align*}
A^p_{11}=&\bigcup_{\tilde x\in(x+\alpha_p\mathcal B)\cap(\textbf{S}^p\setminus\textbf{X}^{p-1})}\bigcup_{\tilde x'\in X^p_E(\tilde x)\setminus\textbf{X}^{p-1}}E_{V^{p-1}}(\tilde x'),\\
=&\bigcup_{\tilde x\in(x+\alpha_p\mathcal B)\cap(\textbf{S}^p\setminus\textbf{X}^{p-1})}E_{V^{p-1}}(\tilde x)\\
A^p_{12}=&\bigcup_{\tilde x\in(x+\alpha_p\mathcal B)\cap(\textbf{S}^p\setminus\textbf{X}^{p-1})}\bigcup_{\tilde x'\in X^p_E(\tilde x)\cap\textbf{X}^{p-1}}E_{v_\infty^{p-1}}(\tilde x'),\\
A^p_{13}=&\{\textbf{1}_N\}.
\end{align*}
By the definitions of $\mathbb P$ and $\tilde v^{p-1}_\infty$, 
\begin{align*}
A^p_2=&\bigcup_{\tilde x\in(x+\alpha_p\mathcal B)\cap\textbf{X}^{p-1}}E_{\tilde v^{p-1}_\infty}(\tilde x)=\bigcup_{\tilde x\in(x+\alpha_p\mathcal B)\cap\textbf{X}^{p-1}}E_{v^{p-1}_\infty}(\tilde x).
\end{align*}

By Assumption \ref{asmp:alpha}, we have $(x+\alpha_p\mathcal B)\cap\textbf{X}^{p-1}\neq\emptyset$ and $A^p_2\neq\emptyset$.
It follows from the third property of Lemma \ref{thm:zeroValue} that $\forall \tilde x\in(x+\alpha_p\mathcal B)\cap\textbf{X}^p$, $E_{V^{p-1}}(\tilde x)\subseteq E_{v^p_\infty}(\tilde x)$.
This indicates that $A^p_{11}\subseteq B^p$.
In addition, it trivially holds that $A_{13}\subseteq B^p$.
By the first inequality of \eqref{eq:setInequality1} in Lemma~\ref{thm:setInequality}, $d_H(A^p_{11}\cup A^p_{12}\cup A^p_{13}\cup A^p_2, B^p)\leq d_H(A^p_{12}\cup A^p_2, B^p)$.

\begin{claim}\label{thm:ClaimV7A}
There is $q\geq1$ s.t. $\forall p\geq q$, if $x_i\in X^G_i$, $x_i+\alpha_p\mathcal B\subseteq X^G_i+(M_i\epsilon_p+h_p)\mathcal B$; if $x_i\notin X^G_i$, $(x_i+\alpha_p\mathcal B)\cap (X^G_i+(M_i\epsilon_p+h_p)\mathcal B)=\emptyset$.
\end{claim}
\begin{proof}
By Assumption \ref{asmp:monoConv}, we have $\exists q_i\geq0$ s.t. $\forall p\geq q_i$, $\alpha_p\leq M_i\epsilon_p+h_p$.
Then for any $p\geq q_i$, if $x_i\in X^G_i$, $x_i+\alpha_p\mathcal B\subseteq X^G_i+(M_i\epsilon_p+h_p)\mathcal B$.
It again follows from Assumption \ref{asmp:monoConv} that for each $i\in\mathcal V$ s.t. $x_i\notin X^G_i$, there exists $q_i\geq 1$ s.t. $\forall p\geq q_i$, $(x_i+\alpha_p\mathcal B)\cap (X^G_i+(M_i\epsilon_p+h_p)\mathcal B)=\emptyset$.
Then the desired $q$ is defined as $q\triangleq\max_{i\in\mathcal V}q_i$.
\end{proof}

\begin{claim}\label{thm:ClaimV7C}
For $p\geq q$ and any pair of $\tilde x\in x+\alpha_p\mathcal B$ and $i\in\mathcal V^G_p(\tilde x)$, $x_i\in X^G_i$.
\end{claim}
\begin{proof}
For every $i\in\mathcal V^G_p(\tilde x)$, $\tilde x_i\in X^G_i+(M_i\epsilon_p+h_p)\mathcal B$.
Assume $x_i\notin X^G_i$.
It follows from Claim \ref{thm:ClaimV7A} that $(x_i+\alpha_p\mathcal B)\cap(X^G_i+(M_i\epsilon_p+h_p)\mathcal B)=\emptyset$.
This contradicts the fact that $\tilde x_i\in X^G_i+(M_i\epsilon_p+h_p)\mathcal B$.
Then $x_i\in X^G_i$.
\end{proof}

Fix $p\geq q$ and $\tilde x\in(x+\alpha_p\mathcal B)\cap(\textbf{S}^p\setminus\textbf{X}^{p-1})$ s.t. $X^p_E(\tilde x)\cap\textbf{X}^{p-1}\neq\emptyset$.
Define $\hat x$ s.t. $\hat x_i=\begin{cases}\tilde x_i, &\text{if }\tilde x_i\in X^{p-1}_i;\\\arg\min_{\hat x_i\in X^{p-1}_i}\|\hat x_i-x_i\|, &\text{otherwise.}\end{cases}$
Notice that $\hat x\in\textbf{X}^{p-1}$.
It follows from the definition of $\textbf{X}^{p-1}$ that $\|\hat x_i-x_i\|\leq h_{p-1}\leq\alpha_p, \forall i$ s.t. $\tilde x_i\notin X^{p-1}_i$.
Since $\tilde x\in x+\alpha_p\mathcal B$, we have $\|\hat x_i-x_i\|=\|\tilde x_i-x_i\|\leq\alpha_p, \forall i$ s.t. $\tilde x_i\in X^{p-1}_i$.
Then it holds that $\hat x\in x+\alpha_p\mathcal B$.

\begin{claim}\label{thm:ClaimV7B}
For $p\geq q$, $\hat x\in X^{p-1}_E(\tilde x)$.
\end{claim}
\begin{proof}
Since $\exists\tilde x'\in X^p_E(\tilde x)\cap\textbf{X}^{p-1}$, it follows from the definition of $X^p_E$ that $\tilde x_i=\tilde x'_i\in X^{p-1}_i, \forall i\in \mathcal V\setminus\mathcal V^G_p(\tilde x)$.
Then the following two properties hold for $\tilde x$: (a) $\forall i\in\mathcal V\setminus\mathcal V^G_p(\tilde x), \tilde x_i\in X^{p-1}_i$; (b) $\exists i\in\mathcal V^G_p(\tilde x)$ s.t. $\tilde x_i\in X^p_i\setminus X^{p-1}_i$.
Property (b) is a result of $\tilde x\in\textbf{S}^p\setminus\textbf{X}^{p-1}$.

Fix $j\in\mathcal V$ s.t. $\hat x_j\neq\tilde x_j$.
Now we are to show $\hat x_j\in X^G_j+(M_j\epsilon_{p-1}+h_{p-1})\mathcal B$.
By properties (a)(b), $j\in\mathcal V^G_p(\tilde x)$.
It follows from Claim \ref{thm:ClaimV7C} that $x_j\in X^G_j$.
It follows from Claim \ref{thm:ClaimV7A} that $x_j+\alpha_p\mathcal B\subseteq X^G_j+(M_j\epsilon_p+h_p)\mathcal B$.
Therefore, $\hat x_j\in x_j+\alpha_p\mathcal B\subseteq X^G_j+(M_j\epsilon_p+h_p)\mathcal B$.
By Assumption \ref{asmp:monoConv}, it renders at $\hat x_j\in X^G_j+(M_j\epsilon_{p-1}+h_{p-1})\mathcal B$.

This holds for all $j\in\mathcal V$ s.t. $\hat x_j\neq\tilde x_j$.
By the definition of $X^{p-1}_E$, we have $\hat x\in X^{p-1}_E(\tilde x)$.
\end{proof}

\begin{claim}\label{thm:ClaimV7}
There is $q\geq 1$ s.t. $A^p_{12}\subseteq A^p_2$ holds for all $p\geq q$.
\end{claim}

\begin{proof}
If $A^p_{12}=\emptyset$, the claim trivially holds.
Throughout the proof, assume that $\exists\tilde x\in(x+\alpha_p\mathcal B)\cap(\textbf{S}^p\setminus\textbf{X}^{p-1})$ s.t. $X^p_E(\tilde x)\cap\textbf{X}^{p-1}\neq\emptyset$.

Pick any $\tilde x' \in X^p_E(\tilde x)\cap \textbf{X}^{p-1}$. 
It follows from Assumption \ref{asmp:monoConv} that $X^p_E(\tilde{x}) \subseteq X^{p-1}_E(\tilde{x})$; then $\tilde{x}' \in X^{p-1}_E(\tilde{x})$. 
It follows from Claim \ref{thm:ClaimV7B} that $\exists\hat x\in X^{p-1}_E(\tilde x)$.
Since $\tilde x'\in X^{p-1}_E(\tilde x)$, by the definition of $X^{p-1}_E$, we have $\tilde x'\in X^{p-1}_E(\hat x)$.
Then by Corollary \ref{thm:XEEquivalence}, $E_{v^{p-1}_\infty}(\tilde x')= E_{v^{p-1}_\infty}(\hat x)$.
Since $\hat x\in(x+\alpha_p\mathcal B)\cap\textbf{X}^{p-1}$, then $E_{v^{p-1}_\infty}(\tilde x')\subseteq A^p_2$.
This holds for every pair of $\tilde x\in(x+\alpha_p\mathcal B)\cap(\textbf{S}^p\setminus\textbf{X}^{p-1})$ and $\tilde x'\in X^p_E(\tilde x)\cap\textbf{X}^{p-1}$.
Then $A^p_{12}\subseteq A^p_2$.
\end{proof}

It follows from Lemma~\ref{thm:setInequality} and Claim \ref{thm:ClaimV7} that $b_p(x)\leq d_H(A^p_{12}(x)\cup A^p_2(x), B^p(x))= d_H(A^p_2(x), B^p(x))$ holds for $p\geq q(x)$.
Recall $\alpha_p\geq h_p$ and Assumption \ref{asmp:alpha}.
It follows from Theorem \ref{thm:kru1} that $\lim_{p\to+\infty}b_p(x)\leq\lim_{p\to+\infty} d_H(A^p_2(x), B^p(x))=0$.
Since this holds for all $x\in\textbf{X}$, the lemma is proven.
\end{proofOf}

\begin{proofOf}{Lemma \ref{thm:preliminaryContractionProperty}}
Take $\delta'>\delta \triangleq d_H(A, B)$. Then
$A\subseteq B+\delta'\mathcal B_N, B\subseteq A+\delta'\mathcal B_N.$
Focus on the first relationship and we want to show: 
\begin{equation}\label{eq:A1subsetB1}
\begin{split}
(\textbf{1}-\Delta\tau(x))\circ A\subseteq(\textbf{1}-\Delta\tau(x))\circ B+ e^{-\kappa_p}\delta'\mathcal B_N.
\end{split}\end{equation}
This is equivalent to show that $\forall a\in A$, $\exists b\in B$ s.t. $\|(\textbf{1}-\Delta\tau(x))\circ a-(\textbf{1}-\Delta\tau(x))\circ b\|\leq e^{-\kappa_p}\delta'$.

We start with $A\subseteq B+\delta'\mathcal B_N$, which implies $\forall a\in A, \exists b'\in B$ s.t. $\|a-b'\|\leq\delta'$.
Fix $a$ and $b'$.
Denote the one-hop neighbor of $x$ that attains $b'$ by $\tilde x$; i.e., $\exists\tilde x\in\tilde X^p(x)$ s.t. $b'\in E_{v^p_\infty}(\tilde x)$.
Construct $b\in[0,1]^N$ s.t. $b_i=a_i$, if $i\in\mathcal V^G_p(\tilde x)$; $b_i=b'_i$, otherwise.
Since $b'\in E_{v^p_\infty}(\tilde x)$, $\exists \tau\in v^p_\infty(\tilde x)$ s.t. $b'\succeq\tau$; that is, $b'_i\geq\tau_i$ for all $i\in\mathcal V$.
Specifically, by the third property of Lemma \ref{thm:zeroValue}, for $i\in\mathcal V^G_p(\tilde x)$, $b'_i\geq \tau_i=0$.
Since $b_i=a_i\geq0=\tau_i, \forall i\in\mathcal V^G_p(\tilde x)$ and $b_i=b_i'\geq\tau_i, \forall i\in\mathcal V\setminus\mathcal V^G_p(\tilde x)$, we have $b\succeq\tau$ and thus $b\in E_{v^p_\infty}(\tilde x)$.

Now we have $\|(\textbf{1}-\Delta\tau(x))\circ a-(\textbf{1}-\Delta\tau(x))\circ b\|^2=\sum_{i\in\mathcal V\setminus\mathcal V^G_p(\tilde x)}(1-\Delta\tau_i(x))^2(a_i-b_i')^2$.
By the first property of Lemma \ref{thm:zeroValue}, $\mathcal V\setminus\mathcal V^G_p(x)\supseteq\mathcal V\setminus\mathcal V^G_p(\tilde x)$.
Then it follows from \eqref{eq:deltaTau} that $1-\Delta\tau_i(x)=e^{-\kappa_p}, \forall i\in\mathcal V\setminus\mathcal V^G_p(\tilde x)$.
Therefore,
\begin{align*}
&\|(\textbf{1}-\Delta\tau(x))\circ a-(\textbf{1}-\Delta\tau(x))\circ b\|^2\\
=&(e^{-\kappa_p})^2\sum_{i\in\mathcal V\setminus\mathcal V^G_p(\tilde x)}(a_i-b'_i)^2\\
\leq&(e^{-\kappa_p})^2\|a-b'\|^2\leq(e^{-\kappa_p}\delta')^2.
\end{align*}
Since this holds $\forall a\in A$ and $b\in B$, then \eqref{eq:A1subsetB1} is proven.
A similar relationship for $B\subseteq A+\delta'\mathcal B_N$ can be obtained by swapping $A$ and $B$:
\begin{equation}\label{eq:B1subsetA1}
\begin{split}
(\textbf{1}-\Delta\tau(x))\circ B\subseteq(\textbf{1}-\Delta\tau(x))\circ A+ e^{-\kappa_p}\delta'\mathcal B_N.
\end{split}\end{equation}
Combining \eqref{eq:A1subsetB1} and \eqref{eq:B1subsetA1}, we arrive at $d_H((\textbf{1}-\Delta\tau(x))\circ A, (\textbf{1}-\Delta\tau(x))\circ B)\leq\delta' e^{-\kappa_p}$.
Since these two relationships hold for all $\delta'>\delta$, the lemma is then proven.
\end{proofOf}

\begin{proofOf}{Theorem \ref{thm:ContractionProperty}}
Fix $x\in\textbf{S}^p$.
For simplicity, shorthand notations listed below are used in the rest of the proof:
\begin{align*}
\tilde A(x)=\bigcup_{\tilde x\in \tilde X^p(x)}E_{v^p_n}(\tilde x), \quad
\tilde B(x)=\bigcup_{\tilde x\in \tilde X^p(x)}E_{v^p_\infty}(\tilde x), \\
\hat A(x)=\bigcup_{\hat x\in \hat X^p(x)}E_{v^p_n}(\hat x), \quad
\hat B(x)=\bigcup_{\hat x\in \hat X^p(x)}E_{v^p_\infty}(\hat x).
\end{align*}
Since translating each term in the Hausdorff distance with a common vector $\Delta\tau(x)$ does not change the distance, we focus on the discounted terms in \eqref{eq:epiProfileBellman}.
The following holds:
\begin{equation*}
\begin{split}
&d_H(E_{\mathbb Gv^p_n}(x), E_{\mathbb Gv^p_\infty}(x))\\
=&d_H((\textbf{1}-\Delta\tau(x))\circ\tilde A(x),(\textbf{1}-\Delta\tau(x))\circ\tilde B(x))\\
\leq&e^{-\kappa_p}d_H(\tilde A(x), \tilde B(x)).
\end{split}\end{equation*}
where the last inequality follows from Lemma \ref{thm:preliminaryContractionProperty}.
By Lemma \ref{thm:tildeHatEquivalence}, the right-hand of the above may be rewritten as $e^{-\kappa_p}d_H(\hat A(x), \hat B(x))$.
Taking supremum over all $x\in\textbf{S}^p$ on both sides makes the left-hand side yield to $d_{\textbf{S}^p}(E_{\mathbb Gv^p_n},  E_{\mathbb Gv^p_\infty})$.
Then the following holds:
\begin{align}\label{eq:CP1}
d_{\textbf{S}^p}(E_{\mathbb Gv^p_n}, E_{\mathbb Gv^p_\infty})\leq e^{-\kappa_p}d_{\textbf{S}^p}(\hat A(x),\hat B(x)).
\end{align}
It follows from \eqref{eq:shrinking} in Lemma~\ref{thm:maximumNormShrinkingPerturbation} that
\begin{equation*}
\begin{split}
&d_{\textbf{S}^p}(\hat A(x),\hat B(x))\\
\leq&d_{\textbf{S}^p}(\bigcup_{\tilde x\in (x+\alpha_p\mathcal B)\cap\textbf{X}^p}E_{v^p_n}(\tilde x), \bigcup_{\tilde x\in(x+\alpha_p\mathcal B)\cap\textbf{X}^p}E_{v^p_\infty}(\tilde x)).
\end{split}\end{equation*}
Notice that $\forall\tilde x\in\textbf{X}^p\setminus\textbf{S}^p$ and $\tilde x'\in\textbf{X}^p$, it holds that $E_{v^p_n}(\tilde x)=\{\textbf{1}_N\}\subseteq E_{v^p_n}(\tilde x')$.
Then the above inequality can be extended to the following one:
\begin{equation}\label{eq:CP2}
\begin{split}
&d_{\textbf{S}^p}(\hat A(x),\hat B(x))\\
\leq&d_{\textbf{S}^p}(\bigcup_{\tilde x\in (x+\alpha_p\mathcal B)\cap\textbf{X}^p}E_{v^p_n}(\tilde x), \bigcup_{\tilde x\in(x+\alpha_p\mathcal B)\cap\textbf{X}^p}E_{v^p_\infty}(\tilde x))\\
\leq&d_\textbf{X}(\bigcup_{\tilde x\in (x+\alpha_p\mathcal B)\cap\textbf{X}^p}E_{v^p_n}(\tilde x), \bigcup_{\tilde x\in(x+\alpha_p\mathcal B)\cap\textbf{X}^p}E_{v^p_\infty}(\tilde x)).
\end{split}\end{equation}
Combine \eqref{eq:CP1} and \eqref{eq:CP2}, then \eqref{eq:contractionInequality1} is proven.
Inequality \eqref{eq:contractionInequality2} is a direct result of \eqref{eq:shrinking2} in Lemma~\ref{thm:maximumNormShrinkingPerturbation}.
\end{proofOf}

\begin{proofOf}{Lemma \ref{thm:discountedDistanceOnOneGrid}}
For each grid $\textbf{X}^p$, from Line \ref{alg:1:np} of Algorithm \ref{alg:1}, one can see that the value iterations on grid $\textbf{X}^p$ terminate when (1) $n>n_p$; or (2) the fixed point $v^p_{\infty}$ is reached. 
Two cases arise.

Case 1: Value iterations terminate before the fixed point is attained; i.e., $v^p_{\bar n_p}=v^p_{n_p}$.
Notice that $\forall x\in\textbf{X}^p\setminus\textbf{S}^p$, $E_{v^p_n}(x)=E_{v^p_\infty}(x)=\{\textbf{1}_N\}$.
Then the following holds:
\begin{equation*}
\begin{split}
&d_{\textbf{X}^p}(E_{v^p_{n_p}}, E_{v^p_\infty})\\
=&\max\{d_{\textbf{S}^p}(E_{\mathbb Gv^p_{n_p-1}}, E_{\mathbb Gv^p_\infty}), d_{\textbf{X}^p\setminus\textbf{S}^p}(E_{v^p_{n_p}}, E_{v^p_\infty})\}\\
=&d_{\textbf{S}^p}(E_{\mathbb Gv^p_{n_p-1}}, E_{\mathbb Gv^p_\infty}).
\end{split}
\end{equation*}

We apply inequality \eqref{eq:contractionInequality2} in Theorem \ref{thm:ContractionProperty} for $n_p-1$ times to $d_{\textbf{X}^p}(E_{v^p_{n_p}},E_{v^p_\infty})$, then the following inequalities are obtained:
\begin{equation*}
\begin{split}
&d_{\textbf{X}^p}(E_{v^p_{n_p}}, E_{v^p_\infty})
=d_{\textbf{S}^p}(E_{\mathbb Gv^p_{n_p-1}}, E_{\mathbb Gv^p_\infty})\\
\leq&e^{-\kappa_p}d_{\textbf{S}^p}(E_{v^p_{n_p-1}}, E_{v^p_\infty})
= e^{-\kappa_p}d_{\textbf{X}^p}(E_{\mathbb Gv^p_{n_p-2}},E_{\mathbb Gv^p_\infty})\\
\leq&\cdots\leq e^{-(n_p-1)\kappa_p}d_{\textbf{S}^p}(E_{\mathbb Gv^p_0}, E_{\mathbb Gv^p_\infty})\\
\leq&e^{-n_p\kappa_p}d_{\textbf{X}}(\bigcup_{\tilde x\in(x+\alpha_p\mathcal B)\cap\textbf{X}^p}E_{v^p_0}(\tilde x), \bigcup_{\tilde x\in(x+\alpha_p\mathcal B)\cap\textbf{X}^p}E_{v^p_\infty}(\tilde x)),
\end{split}
\end{equation*}
where the last equality is a result of \eqref{eq:contractionInequality1} in Theorem \ref{thm:ContractionProperty}.

By Lemma~\ref{thm:TriangleInequality}, the right-hand side of the above becomes:
\begin{equation*}
\begin{split}
&d_{\textbf{X}}(\bigcup_{\tilde x\in(x+\alpha_p\mathcal B)\cap\textbf{X}^p}E_{\mathbb P\tilde v^{p-1}_{n_{p-1}}}(\tilde x), \bigcup_{\tilde x\in(x+\alpha_p\mathcal B)\cap\textbf{X}^p}E_{v^p_\infty}(\tilde x))\\
\leq&d_{\textbf{X}}(\bigcup_{\tilde x\in(x+\alpha_p\mathcal B)\cap\textbf{X}^p}E_{\mathbb P\tilde v^{p-1}_{n_{p-1}}}(\tilde x),\bigcup_{\tilde x\in(x+\alpha_p\mathcal B)\cap\textbf{X}^p}E_{\mathbb P\tilde v^{p-1}_\infty}(\tilde x))\\
&+d_{\textbf{X}}(\bigcup_{\tilde x\in(x+\alpha_p\mathcal B)\cap\textbf{X}^p}E_{\mathbb P\tilde v^{p-1}_\infty}(\tilde x),\bigcup_{\tilde x\in(x+\alpha_p\mathcal B)\cap\textbf{X}^p}E_{v^p_\infty}(\tilde x)),
\end{split}
\end{equation*}
where the second term is $b_p$ in Lemma \ref{thm:Lemma9}.
As for the first term, it follows from \eqref{eq:shrinking2} in Lemma~\ref{thm:maximumNormShrinkingPerturbation} that
\begin{align*}
&d_\textbf{X}(\bigcup_{\tilde x\in(x+\alpha_{p}\mathcal B)\cap\textbf{X}^{p}}E_{\mathbb P\tilde v^{p-1}_{n_{p-1}}}(\tilde x), \bigcup_{\tilde x\in (x+\alpha_{p}\mathcal B)\cap\textbf{X}^{p}}E_{\mathbb P\tilde v^{p-1}_\infty}(\tilde x))\\
\leq&d_{\textbf{X}^p}(E_{\mathbb P\tilde v^{p-1}_{n_{p-1}}},E_{\mathbb P\tilde v^{p-1}_\infty}).
\end{align*}
We focus on the right-hand side of the above inequality and proceed to show that 
\begin{align}\label{eq:discountedDistanceOnOneGrid:1}
d_{\textbf{X}^p}(E_{\mathbb P\tilde v^{p-1}_{n_{p-1}}},E_{\mathbb P\tilde v^{p-1}_\infty})\leq d_{\textbf{X}^{p-1}}(E_{v^{p-1}_{n_{p-1}}},E_{v^{p-1}_\infty}).
\end{align}
For each $x\in\textbf{X}^p$, if $x\in\textbf{X}^{p-1}$, it follows from the definition of $\mathbb P$ that $E_{\mathbb P\tilde v^{p-1}_{n_{p-1}}}(x)=E_{\tilde v^{p-1}_{n_{p-1}}}(x)$ and $E_{\mathbb P\tilde v^{p-1}_\infty}(x)=E_{\tilde v^{p-1}_\infty}(x)$.
By the definition of $\tilde v$, $E_{\tilde v^{p-1}_{n_{p-1}}}(x)=E_{v^{p-1}_{n_{p-1}}}(x)$ and $E_{\tilde v^{p-1}_\infty}(x)=E_{v^{p-1}_\infty}(x)$.
Therefore, 
\begin{align}\label{eq:discountedDistanceOnOneGrid:2}
d_{\textbf{X}^{p-1}}(E_{\mathbb P\tilde v^{p-1}_{n_{p-1}}},E_{\mathbb P\tilde v^{p-1}_\infty})= d_{\textbf{X}^{p-1}}(E_{v^{p-1}_{n_{p-1}}},E_{v^{p-1}_\infty}).
\end{align}
If $x\in\textbf{X}^p\setminus\textbf{X}^{p-1}$, it follows the definition of $\mathbb P$ that $E_{\mathbb P\tilde v^{p-1}_{n_{p-1}}}(x)=\bigcup_{x'\in X^p_E(x)}E_{\tilde v^{p-1}_{n_{p-1}}}(x')$ and $E_{\mathbb P\tilde v^{p-1}_\infty}(x)=\bigcup_{x'\in X^p_E(x)}E_{\tilde v^{p-1}_\infty}(x')$.
For each $x'\in X^p_E(x)$, if $x'\in\textbf{X}^p\setminus\textbf{X}^{p-1}$, we have $E_{\tilde v^{p-1}_{n_{p-1}}}(x')=E_{V^{p-1}}(x')$ and $E_{\tilde v^{p-1}_\infty}(x')=E_{V^{p-1}}(x')$.
Otherwise, i.e. $x'\in\textbf{X}^p$, it follows from the definition of $\tilde v$ that $E_{\tilde v^{p-1}_{n_{p-1}}}(x')=E_{v^{p-1}_{n_{p-1}}}(x')$ and $E_{\tilde v^{p-1}_\infty}(x')=E_{v^{p-1}_\infty}(x')$. 
By the third properties of Lemma \ref{thm:zeroValue}, $E_{V^{p-1}}(x')\subseteq E_{v^{p-1}_{n_{p-1}}}(x')$ and $E_{V^{p-1}}(x')\subseteq E_{v^{p-1}_\infty}(x')$.
Then we have
\begin{equation*}\label{eq:discountedDistanceOnOneGrid:3}
\begin{split}
&d_{\textbf{X}^p\setminus\textbf{X}^{p-1}}(E_{\mathbb P\tilde v^{p-1}_{n_{p-1}}},E_{\mathbb P\tilde v^{p-1}_\infty})\\
=&d_{\textbf{X}^p\setminus\textbf{X}^{p-1}}(\bigcup_{x'\in X^p_E(x)\cap\textbf{X}^p}E_{v^{p-1}_{n_{p-1}}}(x'), \bigcup_{x'\in X^p_E(x)\cap\textbf{X}^p}E_{v^{p-1}_\infty}(x'))\\
=&d_{\textbf{X}^{p-1}}(E_{v^{p-1}_{n_{p-1}}},E_{v^{p-1}_\infty}).
\end{split}
\end{equation*}
Then \eqref{eq:discountedDistanceOnOneGrid:1} is a result of \eqref{eq:discountedDistanceOnOneGrid:2} and the above inequality.
Therefore, inequality \eqref{eq:discountedDistanceOnOneGrid} is obtained for this case.

Case 2: The fixed point is reached; i.e., $v^p_{\bar n_p}=v^p_\infty$.
The left-hand side of \eqref{eq:discountedDistanceOnOneGrid} is zero and it is trivially true.

In summary, the lemma is proven.
\end{proofOf}

\end{subsection}

\begin{subsection}{Proof of Theorem \ref{thm:multiRobotConvergence}}\label{sect:multiRobotMainProof}

We set out to finish the proof of Theorem  \ref{thm:multiRobotConvergence}.
For each grid $\textbf{X}^p$, we distinguish the folllowing two cases.

Case 1: $p=D_{k+1}$ for some $k\geq0$.
We look back to $D_k$-th grid and apply Lemma \ref{thm:discountedDistanceOnOneGrid} for $D_{k+1}-D_k$ times:
\begin{equation*}\label{eq:multiSynConvergence}
\begin{split}
&d_{\textbf{X}^p}(E_{v^p_{\bar n_p}},E_{v^p_\infty})\leq\gamma_p d_{\textbf{X}^{p-1}}(E_{v^{p-1}_{\bar n_{p-1}}},E_{v^{p-1}_\infty})+b_{p}\\
\leq&\gamma_p\gamma_{p-1}d_{\textbf{X}^{p-2}}(E_{v^{p-2}_{\bar n_{p-2}}},E_{v^{p-2}_\infty})+\gamma_pb_{p-1}+b_{p}\\
\leq&(\prod_{q=D_k+1}^{D_{k+1}}\gamma_q)d_{\textbf{X}^{D_k}}(E_{v^{D_k}_{\bar n_{D_k-1}}},E_{v^{D_k}_\infty})+\sum_{q=D_k+1}^{D_{k+1}}(\prod_{r=q+1}^{D_{k+1}}\gamma_r)b_q,
\end{split}\end{equation*}
where $b_q$ is defined in Lemma \ref{thm:Lemma9}.
By Assumption \ref{asmp:DWindow}, $\prod_{q=D_k+1}^{D_{k+1}}\gamma_q=\exp(-\sum_{q=D_k+1}^{D_{k+1}}n_q\kappa_q)\leq \gamma$.
Since $D_{k+1}-D_k\leq \bar D$ and $\gamma_r\leq 1$,
$$d_{\textbf{X}^{D_{k+1}}}(E_{v^{D_{k+1}}_{\bar n_p}},E_{v^{D_{k+1}}_\infty})\leq\gamma d_{\textbf{X}^{D_k}}(E_{v^{D_k}_{\bar n_{D_k}}},E_{v^{D_k}_\infty})+\sum_{q=D_k+1}^{D_k+\bar D}b_q.$$
By Lemma \ref{thm:Lemma9}, $b_q\to0$ as $q\to+\infty$; hence $\lim_{k\to+\infty}\sum_{q=D_k+1}^{D_k+\bar D}b_q=0$. 
Therefore, by Lemma~\ref{thm:lemmaIX5}, $\lim_{k\to+\infty}d_{\textbf{X}^{D_k}}(E_{v^{D_k}_{\bar n_{D_k}}},E_{v^{D_k}_\infty})=0$.

Case 2: $p\neq D_{k+1}$ for any $k\geq0$. Then $\exists k\geq0$ s.t. $D_k+1\leq p< D_{k+1}$.
We apply Lemma \ref{thm:discountedDistanceOnOneGrid} for $p-D_k$ times:
\begin{equation*}
\begin{split}
&d_{\textbf{X}^p}(E_{v^p_{\bar n_p}},E_{v^p_\infty})\leq(\prod_{q=D_k+1}^p\gamma_q)d_{\textbf{X}^{D_k}}(E_{v^{D_k}_{\bar n_{D_k}}},E_{v^{D_k}_\infty})\\
&+\sum_{q=D_k}^p(\prod_{r=q+1}^{D_{k+1}}\gamma_r)b_q\leq d_{\textbf{X}^{D_k}}(E_{v^{D_k}_{\bar n_{D_k}}},E_{v^{D_k}_\infty})+\bar D\bar B_{D_k},
\end{split}
\end{equation*}
where $\bar B_p\triangleq\sup_{q\geq p+1}b_q$.
It follows from Lemma \ref{thm:Lemma9} that $\lim_{p\to+\infty}\bar B_p=0$.
Hence, by Lemma~\ref{thm:lemmaIX5}, $\lim_{p\to+\infty}d_{\textbf{X}^p}(E_{v^p_{\bar n_p}},E_{v^p_\infty})=0$.

Combining the above two cases, we may rewrite the result as
$\lim_{p\to+\infty}d_{\textbf{X}^p}(E_{v^p_{\bar n_p}},E_{v^p_\infty})=0$.
Pick $x\in\textbf{X}$.
By \eqref{eq:shrinking2} in Lemma~\ref{thm:maximumNormShrinkingPerturbation}, the following holds:
\begin{align*}
&d_H(\bigcup_{\tilde x\in(x+\eta_p\mathcal B)\cap\textbf{X}^p}E_{v^p_{\bar n_p}}(\tilde x), \bigcup_{\tilde x\in(x+\eta_p\mathcal B)\cap\textbf{X}^p}E_{v^p_\infty}(\tilde x))\\
\leq& d_{\textbf{X}^p}(E_{v^p_{\bar n_p}}, E_{v^p_\infty}).
\end{align*}
Take the limit $p\to+\infty$ on both sides, then the above relationship yields:
\begin{align*}
\lim_{p\to+\infty}d_H(\bigcup_{\tilde x\in(x+h_p\mathcal B)\cap\textbf{X}^p}E_{v^p_{\bar n_p}}(\tilde x), \bigcup_{\tilde x\in(x+h_p\mathcal B)\cap\textbf{X}^p}E_{v^p_\infty}(\tilde x))=0.
\end{align*}
Since this holds for all $x\in\textbf{X}$, 
$$\lim_{p\to+\infty}d_{\textbf{X}}(\bigcup_{\tilde x\in(x+h_p\mathcal B)\cap\textbf{X}^p}E_{v^p_{\bar n_p}}(\tilde x),\bigcup_{\tilde x\in(x+h_p\mathcal B)\cap\textbf{X}^p}E_{v^p_\infty}(\tilde x))=0.$$
By Theorem \ref{thm:kru1}, $\Lim{p\to+\infty}\bigcup_{\tilde x\in(x+h_p\mathcal B)\cap\textbf{X}^p}E_{v^p_\infty}(\tilde x)$ exists for any $x\in\textbf{X}$ and equals to $E_{v^*}(x)$.
Hence, it holds that $ \forall x\in\textbf{X}, \Lim{p\to+\infty}\bigcup_{\tilde x\in(x+h_p\mathcal B)\cap\textbf{X}^p}E_{v^p_{\bar n_p}}(\tilde x)=E_{v^*}(x)$.
Then the theorem is proven.

\end{subsection}

\end{section}

\begin{section}{Conclusion}
In this paper, a numerical algorithm is proposed to find the Pareto optimal solution of a class of multi-robot motion planning problems.
The consistent approximation of the algorithm is guaranteed using set-valued analysis.
A set of experiments on an indoor multi-robot platform and computer simulations are conducted to assess the anytime property.
There are a couple of interesting problems to solve in the future. 
First, the proposed algorithm is centralized. 
It is of interest to study distributed implementation. 
Second, it is interesting to find more efficient ways to construct set-valued dynamics and perform value iteration.
\end{section}

\bibliographystyle{IEEEtran}
\bibliography{MRS}

\begin{section}{Appendix}\label{sect:multiRobotAppendix}

In this section, the proof of Theorem~\ref{thm:kru1} is provided.

We show the existence of fixed point $\varTheta^p_{\infty}$ for each $p$ and the convergence of $\varTheta^p_{\infty}$ to $\varTheta^*$ in the epigraphical profile sense.
The proof extends results of approximating minimal time functions in \cite{cardaliaguet1999setvalued} to multi-robot scenario and proves the existence and pointwise convergence of fixed points $\varTheta^p_\infty=\Psi^{-1}(v^p_\infty)$ of Algorithm \ref{alg:1}.

The appendix consists of the following subsections:

\begin{itemize}
\item Subsection \ref{sect:appNotations}: Notations used in the appendix and preliminary results;
\item Subsection \ref{sect:viabkernel}: Pareto optimal solutions are reformulated in terms of viability kernels in Theorem \ref{thm:viabepi}.
This transforms the problem of constantly approxmating Pareto optimal solutions to the problem of finding viability kernel;
\item Subsection \ref{sect:discretizing}: A fully discretized scheme is developed to consistently approximate the viability kernel. The convergence is summarized in Theorem \ref{thm:fullyApprox};

\item Subsection \ref{sect:varthetaconv}: epigraphical and pointwise convergence of fixed points $\varTheta^p_\infty$ to $\varTheta^*$ is proven in Theorem \ref{thm:varthetaconv}. 

\end{itemize}

\begin{remark}
The proofs in \cite{cardaliaguet1999setvalued} are not applicable to our multi-robot setting.
In \cite{cardaliaguet1999setvalued}, the objective function is single-valued; in this paper, the image of objective function is partial ordered, meaning multiple values may all be optimal. 
This requires a new comparison that returns every optimal value and extended the Principle of Optimality based on such comparison.
Theorem \ref{thm:varthetaconv} is the extension to Theorem 2.19 in \cite{cardaliaguet1999setvalued}.
Lemma \ref{thm:zerogoaldomain_prop} is a new result showing that the estimated travel time for robots in the goal regions remains zero throughout the updates.
\end{remark}

\subsection{Further Notations and Preliminaries}\label{sect:appNotations}
Throughout the appendix, we leverage the following concepts.
\begin{definition}[Graph]\label{def:graph}
The graph of $\varTheta$ is defined by $gph(\varTheta)\triangleq\{(x, t)\in\mathcal X\times\mathbb R^N|t\in\varTheta(x)\}$.
\end{definition}
\begin{definition}[Viability kernel]
Let $\mathcal{D}$ be a closed set. 
The viability kernel of $\mathcal{D}$ for some dynamics $\Phi$ is the set $\{(x, t)\in\mathcal{D}|\exists(x(\cdot), t(\cdot))\text{ s.t. }(x(0), t(0))=(x, t), (\dot{x}(s), \dot{t}(s))\in\Phi(x(s), t(s)), (x(s), t(s))\in\mathcal{D},\forall s\in [0, +\infty)\}.$ It is denoted as $Viab_{\Phi}(\mathcal{D})$.
\end{definition}

Define spatial-temporal space and safety region-temporal space by $\mathcal H$ by $\mathcal H\triangleq\textbf{X}\times\mathbb R^N_{\geq0}$ and $\mathcal{S}=\textbf{S}\times\mathbb{R}_{\geq0}^N$ respectively.
The discrete spatial-temporal space and discrete safety region-temproal space are defined as $\mathcal H^p\triangleq\textbf{X}^p\times(\mathbb R^N_{\geq0})^p$ and $\mathcal S^p\triangleq[(\textbf{S}+h_p\mathcal B)\cap\textbf{X}^p]\times(\mathbb R^N_{\geq0})^p$ respectively.

The following lemma shows the monotonicity of Kuratowski convergence.
\begin{lemma}\label{thm:dominatedConv}
For two sequences of sets $\{A_n\}, \{B_n\}$ s.t. $A_n\subseteq B_n\subseteq\mathcal X$ for all $n\geq1$, then the following hold: 
\begin{align*}
\Limsup{n\to+\infty}A_n\subseteq\Limsup{n\to+\infty}B_n, \quad\Liminf{n\to+\infty}A_n\subseteq\Liminf{n\to+\infty}B_n.
\end{align*}
\end{lemma}
\begin{proof}
Fix $x\in\mathcal X$ and $n\geq1$.
Since $A_n\subseteq B_n$, we have $d(x, B_n)\leq d(x, A_n)$.
Then fix $x\in\mathrm{Limsup}_{n\to+\infty}A_n$.
It follows from the definition of $\mathrm{Limsup}$ that $\lim_{n\to+\infty}d(x, A_n)=0$.
Since $0\leq d(x, B_n)\leq d(x, A_n)$, we have $\lim_{n\to+\infty}d(x, B_n)=0$.
This implies $x\in \mathrm{Limsup}_{n\to+\infty}B_n$.
Since it holds for all $x\in\mathrm{Limsup}_{n\to+\infty}A_n$, the first relationship is proven.
The second one can be shown by exactly following the arguments towards the first one.
\end{proof}

A preliminary lemma is introduced to show that for $2$-norm, $N$-fold Cartesian product expands uniform perturbation by $\sqrt{N}$.
\begin{lemma}\label{thm:prodDist}
Given $A_i\subseteq\mathbb R^{d_i}$, where $d_i\geq1$ and $i\in\{1, \dots, N\}$.
Let $d\triangleq\sum_{i=1}^Nd_i$.
Then for any $\eta>0$, it holds that $\prod_{i=1}^N(A_i+\eta\mathcal B_{d_i})\subseteq\prod_{i=1}^NA_i+\sqrt{n}\eta\mathcal B_d.$
\end{lemma}
\begin{proof}
Pick $x\in\prod_{i=1}^N(A_i+\eta\mathcal B_{d_i})$.
We may rewrite $x$ as $x=\begin{bmatrix}x_1^T&\dots&x_N^T\end{bmatrix}$, where $x_i\in\mathbb R^{d_i}$.
Then it follows from the definition of $\mathcal B$ that for any $i\in\{1, \dots, n\}$, $\exists a_i\in A_i$ s.t. $\|x_i-a_i\|\leq\eta$.
That is, $\eta^2\geq\sum_{j=1}^{d_i}(x_{i,j}-a_{i,j})^2$, where $x_{i,j}$ and $a_{i,j}$ are the $j$-th element of $x_i$ and $a_j$ respectively.
Sum it up for all $i=1, \dots, N$, then we have $n\eta^2\geq\sum_{i=1}^N\sum_{j=1}^{d_i}(x_{i,j}-a_{i,j})^2=\|x-a\|^2$.
Therefore, $\|x-a\|\leq\sqrt{N}\eta$ and $x\in\prod_{i=1}^NA_i+\sqrt{N}\eta\mathcal B_d$.
The lemma is then proven.
\end{proof}

\begin{subsection}{From Pareto optimality to viability kernel: Theorem \ref{thm:viabepi}}\label{sect:viabkernel}


Consider a team of robots that each robot $i\in\mathcal V$ is equiped with an independent body-attached countdown clock with initial value $t_i$.
For each robot, it moves while its clock counts down until either the robot reaches its goal region or the clock counts to $0$.
Therefore, a non-collision trajectory for the whole team is always in the safety region coupled with ``positive time space'' $\mathcal{S}$. 
Then finding the collection of possible minimum arrival time vectors $\varTheta^*(x)$ for every state $x\in\textbf{X}$ is equivalent to finding the viability kernel of $\mathcal S$.
Rigorous reformulation is given.
First, an expanded set-valued map of $F$  is given to describe the dynamics:
\begin{align}\label{eq:phi}
\Phi(x, t)\triangleq\prod_{i\in\mathcal V}\Phi_i(x_i, t_i),
\end{align}
where
\begin{align*}
	\Phi_i(x_i, t_i)\triangleq\begin{cases}
		F_i(x_i)\times\{-1\},\quad\text{if}\;x_i\notin X_i^G;\\
		\bar{co}([F_i(x_i)\times\{-1\}]\cup[\{\textbf{0}_d\}\times\{0\}]),\\
		\quad\quad\text{if}\; x_i\in X_i^G
	\end{cases}
\end{align*}
and $\bar{co}(\cdot)$ represents the closed convex hull of a specified set.
Then the collection of minimum arrival time vectors greater than elements of $\{\vartheta(x, \pi)|\pi\in\varpi\}$ can be expressed as $\mathcal{E}pi(\varTheta^*)$.
The reformulation is summarized below.

\begin{theorem}\label{thm:viabepi}
If system \eqref{eq:0} satisfies Assumption \ref{asmp:1}, $Viab_{\Phi}(\mathcal S)=\mathcal{E}pi(\varTheta^*)$. 
\end{theorem}

\begin{proof}
The proof mainly follows the proof of Theorem 3.2 in \cite{cardaliaguet1999setvalued}.
For the sake of self-containedness, we provide the complete proof.

First, we proceed to show that $Viab_{\Phi}(\mathcal{S})\subseteq\mathcal{E}pi(\varTheta^*)$. 

Fix $(x, t)\in Viab_{\Phi}(\mathcal{S})$. 
If $x\in\textbf{X}^G$, $\dot t=\textbf{0}_N$.
Therefore, $(x, t)\in\mathcal{E}pi(\varTheta^*)$ holds for any $t\in[0, +\infty)^N$. 
If $x\notin\textbf{X}^G$, then  $\exists i\in\mathcal V\text{ s.t. } x_i\notin X^G_i$. Since $(x,t)\in Viab_{\Phi}(\mathcal{S})$, then $\exists \pi\in\varpi \text{ s.t. } $ 
$\dot{x}_i(s)\in F_i(x_i(s)),t_i(s)=t_i-s,(x(s), t(s))\in\mathcal{S},\forall s\in[0, \vartheta_i(x, \pi)]$. 
This indicates, $t_i(s)\geq 0$ and $t_i\geq\vartheta_i(x, \pi)$, i.e., $t\succeq\vartheta(x, \pi)$.
It follows from the definition of $\varTheta^*$ that $\exists\mathcal T\in\varTheta^*(x)$ s.t. $\vartheta(x, \pi)\succeq\mathcal T$.
That is, $t\succeq \mathcal T$ and $(x, t)\in\mathcal{E}pi(\varTheta^*)$.
Then we arrive at $Viab_{\Phi}(\mathcal{S})\subseteq\mathcal{E}pi(\varTheta^*).$

Second, prove $Viab_{\Phi}(\mathcal{S})\supseteq\mathcal{E}pi(\varTheta^*)$. 

Take $(x, t)\in\mathcal{E}pi(\varTheta^*)$. 
If $x\in\textbf{X}^G$, it is trivial.
If $x\notin\textbf{X}^G$, $\exists i\in\mathcal V$ s.t $x_i\notin X^G_i$.
Since $t\in E_{\varTheta^*}(x)$, then $\exists \mathcal T\in \varTheta^*(x)$ s.t. $t\succeq\mathcal T$.
Since $t$ is finite, which means $\mathcal T$ is finite, then $\exists \pi^*\in \mathcal{U}^*(x)$ s.t. $\vartheta(x, \pi^*)=\mathcal T$.
Moreover, $x(s)\in\textbf{X}, \forall s\in[0, +\infty)$ and $x_i(\vartheta_i(x, \pi^*))\in X^G_i, \forall i\in\mathcal V$. 
Denote the trajectory caused by $\pi^*$ as $(\bar{x}(s), \bar{t}(s)),\forall s\in [0, +\infty)$.
For each agent $i\in\mathcal V$, define $(x^*_i(s), t^*_i(s))$ as 
\begin{equation}\label{eq:starsys}
\begin{split}
	(x^*_i(s), t^*_i(s))=
		\begin{cases}
		(\bar{x}_i(s), \bar{t}_i-s),\text{ if } s\leq\vartheta_i(x, \pi^*); \\
		(\bar{x}_i(\vartheta_i(x, \pi^*)), \bar{t}_i-\vartheta_i(x, \pi^*)),\\
		\quad\text{if } s>\vartheta_i(x, \pi^*).
		\end{cases}
\end{split}\end{equation}
Since $t\succeq\mathcal T=\vartheta(x, \pi^*)$, then $\forall s\geq 0, t^*(s)\succeq \textbf{0}_N$.
For each agent $i$, at time $s=\vartheta_i(x, \pi^*)$, $x^*(s)\in\textbf{X}^G$. For $s\geq\vartheta_i(x, \pi^*)$, both $x^*(s)$ and $t^*(s)$ will not change over time and thus $\dot{x}_i^*=0, \dot{t}_i^*=0$. 
Therefore, \eqref{eq:starsys} is a solution to equation (1) and $(x, t)\in Viab_{\Phi}(\mathcal{S})$.
So $Viab_{\Phi}(\mathcal{S})\supseteq\mathcal{E}pi(\varTheta^*).$

Hence, $Viab_{\Phi}(\mathcal{S})\subseteq\mathcal{E}pi(\varTheta^*)$ is proved.
\end{proof}

\end{subsection}

\begin{subsection}{Consistent approximation of viability kernel: Theorem \ref{thm:fullyApprox}}\label{sect:discretizing}
Theorem \ref{thm:viabepi} connects viability kernel to $\mathcal Epi(\varTheta)$. 
Then if $Viab_\Phi(\mathcal{S})$ can be numerically obtained, $\varTheta$ is found.
In this subsection, a numerical method is developed which leverages a sequence of viability kernels in discretized spatial and temporal spaces to consistently approximate the viability kernel of interest.


We construct a good approximation of $\Phi$ via the following two steps:
\begin{itemize}
	\item Build a semi-discretization approximation map $\Phi^p(x, t)$;
	\item Build a fully discretization approximation map $\Gamma^{p}(x, t)$ to approximate $G^p(x, t)$, where $G^p(x, t)\triangleq(x, t)+\epsilon_p\Phi^p(x, t)$.
\end{itemize}

Some criteria are proposed to guarantee the validity of approximation maps. 
\begin{itemize}
	\item[\textbf{(H0)}] $\Phi^p: \mathcal H\rightrightarrows\mathcal H$ is upper semicontinuous with convex compact nonempty values;
	\item[\textbf{(H1)}] $gph(\Phi^p)\subseteq gph(\Phi)+\phi(\epsilon_p)\mathcal B_{\textbf{X}\times\mathbb R^N}$, where $\lim_{p\to+\infty}\phi(\epsilon_p)=0$;
	\item[\textbf{(H2)}] $\forall (x, t)\in\mathcal H$, $\bigcup_{\|(y, t')-(x, t)\|\leq M^+\epsilon_p}\Phi(y, t')\subseteq\Phi^p(x, t)$;
	\item[\textbf{(H3)}] $gph(\Gamma^p)\subseteq gph(G^p)+\psi(\epsilon_p, h_p)\mathcal B_{\textbf{X}\times\mathbb R^N}$, where $\lim_{p\to+\infty}\frac{\psi(\epsilon_p, h_p)}{\epsilon_p}=0^+$;
	\item[\textbf{(H4)}] $\forall (x, t)\in\mathcal H^p, \bigcup_{\|(y, t')-(x, t)\|\leq h_p}[G^p(y, t')+h_p\mathcal B_{\textbf{X}\times\mathbb R^N}]\cap\mathcal H^p\subseteq\Gamma^p(x, t)$.
\end{itemize}

The dynamics of robots are approximated by 
\begin{align}\label{eq:Gamma}
\Gamma^p(x, t)\triangleq\mathcal H^p\cap\prod_{i\in\mathcal V}\Gamma_i^p(x_i, t_i),
\end{align}
 where 
\begin{align*}
	\Gamma^p_i(x_i, t_i)\triangleq\begin{cases}
		\{x_i+\epsilon_p F_i(x_i)+\alpha_p\mathcal B_{X_i}\}\times\{t_i-\epsilon_p+2h_p\mathcal B_1\}, \\\quad\text{if } d(x_i,  X_i^G)> M_i\epsilon_p+h_p;\\
		\bar{co}([\{x_i+\epsilon_p F_i(x_i)+\alpha_p\mathcal B_{X_i}\}\\
			\quad\times\{t_i-\epsilon_p+2h_p\mathcal B_1\}]\cup[\{x_i+2h_p\mathcal B_{X_i}\}\\
			\quad\times\{t_i+2h_p\mathcal B_1\}]),\quad\text{otherwise}.
	\end{cases}
\end{align*}
Recall that $\alpha_p=2h_p+\epsilon_p h_pl^++\epsilon_p^2M^+l^+$.
The safety region-temporal space is discreitzed into a sequence $\{\mathcal S^p_n\}$ as follows:
\begin{align}\label{eq:ss}\begin{cases}
	\mathcal{S}^p_0=\mathcal{S}^p\\
	\mathcal{S}^p_{n+1}=\{(x, t)|\Gamma^p(x, t)\cap\mathcal{S}^p_n\neq\emptyset\}.
\end{cases}\end{align} 

A necessary concept is introduced before we proceed to the main theorem.

\begin{definition}[Discrete viability kernel]
Let $\mathcal{D}$ be a closed set. 
The discrete viability kernel of $\mathcal{D}$ for some dynamics $G^p$ is the set $\{(x, t)\in\mathcal{D}|\exists\{(x_n, t_n)\}_{n=0}^\infty\subseteq\mathcal{D}\text{ s.t. }(x_0, t_0)=(x, t), (x_{n+1}, t_{n+1})\in G^p(x_n, t_n),\forall n\in\mathbb{N}\}.$ It is denoted as $\overrightarrow{Viab}_{G^p}(\mathcal{D})$.
\end{definition}

\begin{theorem}\label{thm:fullyApprox}
If system \eqref{eq:0} satisfies Assumption \ref{asmp:1}, 
then $\Gamma^p$ is a fully discretization of system \eqref{eq:0} and $\mathrm{Lim}_{p\to+\infty}\overrightarrow{Viab}_{\Gamma^p}(\mathcal{S}^p)=Viab_{\Phi}(\mathcal S)$.
\end{theorem}

\begin{proof}
The proof mainly follows the proofs of Lemma 3.3 and 3.4 in \cite{cardaliaguet1999setvalued}. 
For the sake of self-containedness, we provide the complete proof.

The proof is divided into two phases: the first phase exploits Euler discretization to approximates dynamics in temporal horizon and the second phase further discretizes temporal-spatial space.

In the first phase, one approximation of $F$ is proposed by 
\begin{align}\label{eq:phieps}
	\Phi^p(x, t)=\prod_{i\in\mathcal V}\Phi^p_i(x_i, t_i),
\end{align}
where 
\begin{align*}
	\Phi^p_i(x_i, t_i)=\begin{cases}
	\{F_i(x_i)+\epsilon_p M^+l^+\mathcal{B}_{X_i}\}\times\{-1\}, \\ \quad\text{if }d(x_i,  X_i^G)> M_i\epsilon_p,\\
	\bar{co}([\{F_i(x_i)+\epsilon_p M^+l^+\mathcal{B}_{X_i}\}\times\{-1\}] \\ \quad\cup[\{\textbf{0}_d\}\times\{0\}]), \text{otherwise}.
	\end{cases}
\end{align*}
This represents the collection of one-hop moves of robot team at state $(x, t)$.
After a small time $\epsilon_p$, all the possible states of the robot team from $(x, t)$ compose a set $G^p(x, t)\triangleq(x,t)+\epsilon_p\Phi^p(x,t)$.
Now we proceed to prove that $G^p$ can properly approximate $F$.
\begin{claim}\label{thm:appSemiAppx}
$\mathrm{Lim}_{p\to+\infty}\overrightarrow{Viab}_{G^p}(\mathcal S)=Viab_{\Phi}(\mathcal S)$.
\end{claim}
\begin{proof}
First, we are going to show that equation \eqref{eq:phieps} satisfies \textbf{(H1)} with $\phi(\epsilon_p)=2\sqrt{N}\epsilon_p M^+l^+$.
Fix $(x,t)\in\mathcal H$. 
For each $i\in\mathcal V$, it follows from the proof of Lemma 3.3 in \cite{cardaliaguet1999setvalued} that $\Phi_i^p(x_i, t_i)\subseteq\Phi_i(x_i, t_i)+2M^+l^+\epsilon_p\mathcal B_{X_i\times\mathbb R}$.
Therefore, it follows from Lemma \ref{thm:prodDist} that $\Phi^p(x, t)=\prod_{i\in\mathcal V}\Phi^p_i(x_i, t_i)\subseteq\prod_{i\in\mathcal V}(\Phi_i(x_i, t_i)+2M^+l^+\epsilon_p\mathcal B_{X_i\times\mathbb R})\subseteq \prod_{i\in\mathcal V}\Phi_i(x_i, t_i)+2\sqrt{N}M^+l^+\epsilon_p\mathcal B_{\textbf{X}\times\mathbb R^N}=\Phi(x, t)+\phi(\epsilon_p)\mathcal{B}$ and thus equation \eqref{eq:phieps} satisfies \textbf{(H1)}.

Then, we show that equation \eqref{eq:phieps} satisfies \textbf{(H2)}.
For each $i\in\mathcal V$, it again follows from the proof of Lemma 3.3 in \cite{cardaliaguet1999setvalued} that $\bigcup_{\|(y_i, t'_i)-(x_i, t_i)\|\leq M^+\epsilon_p}\Phi^p_i(y_i, t'_i)\subseteq\Phi^p_i(x_i, t_i)$.
Therefore, $\bigcup_{\|(y, t')-(x, t)\|\leq M^+\epsilon_p}\Phi(y, t')\subseteq\prod_{i\in\mathcal V}\bigcup_{\|(y_i, t'_i)-(x_i, t_i)\|\leq M^+\epsilon_p}\Phi^p_i(y_i, t'_i)\subseteq\prod_{i\in\mathcal V}\Phi^p_i(x_i, t_i)=\Phi^p(x, t)$. 
Therefore, quation \eqref{eq:phieps} satisfies \textbf{(H2)}.

Therefore, \textbf{(H1)} and \textbf{(H2)} are satisfied. 
Hypothesis \textbf{(H0)} is satisfied due to Assumption \ref{asmp:1}.
Hence, it follows from Theorem 2.14 in \cite{cardaliaguet1999setvalued} that $\mathrm{Lim}_{p\to+\infty}\overrightarrow{Viab}_{G^p}(\mathcal S)=Viab_{\Phi}(\mathcal S)$.
\end{proof}

In the second phase, the spatial-temporal space is discretized and we come back to $\Gamma^p$.
The following proves the convergence of the discrete viability kernel of $\Gamma^p$ on grid $\mathcal S^p$ to the viability kernel of interest.

\begin{claim}\label{thm:appFullyUpper}
Equation \eqref{eq:Gamma} satisfies \textbf{(H3)} with $\psi(\epsilon_p, h_p)=2\sqrt{Nd}(3h_p+2\epsilon_ph_pl^+)$.
\end{claim}
\begin{proof}
Fix $(x, t)\in\mathcal H^p$ and $i\in\mathcal V$.
Two cases arise:

Case 1: $d(x_i,  X_i^G)>M_i\epsilon_p+h_p>M_i\epsilon_p$. 
Then we have
$\Gamma_i^p(x_i, t_i)=\{x_i+\epsilon_p(F_i(x_i)+\epsilon_pM^+l^+\mathcal B_{X_i})+(2h_p+\epsilon_ph_pl^+)\mathcal B_{X_i}\}\times\{t_i-\epsilon_p+2h_p\mathcal B_1\}\subseteq\{x_i+\epsilon_p(F_i(x_i)+\epsilon_pM^+l^+\mathcal B_{X_i})\}\times\{t_i-\epsilon_p\}+\sqrt{2}(2h_p+\epsilon_p h_pl^+)\mathcal B_{X_i\times\mathbb R}\subseteq G_i^p(x_i, t_i)+\sqrt{2}(3h_p+2\epsilon_p h_pl^+)\mathcal B_{X_i\times\mathbb R}$, where the first inclusion follows from Lemma \ref{thm:prodDist}.

Case 2: $d(x_i,  X_i^G)\leq M_i\epsilon_p+h_p$. Therefore, $\exists x_i'\neq x_i$ s.t. $d(x_i',  X_i^G)\leq M_i\epsilon_p$ and $\|x_i'-x_i\|\leq h_p$.
Further, it follows from Assumption \ref{asmp:Lipschitz} that
\begin{align*}
	&\Gamma_i^p(x_i, t_i)\\
\subseteq&\bar{co}([\{x_i'+\|x_i-x_i'\|\mathcal B_{X_i}+\epsilon_p(F_i(x_i')+l_i\|x_i-x_i'\|\mathcal B_{X_i})\\
&+\alpha_p\mathcal B_{X_i}\}\times\{t_i-\epsilon_p+ 2h_p\mathcal B_1\}] \\
		&\cup[\{x_i'+\|x_i-x_i'\|\mathcal B_{X_i}+2h_p\mathcal B_{X_i}\}\times\{t_i+2h_p\mathcal B_1\}]) \\
		\subseteq&\bar{co}([\{x_i'+\epsilon_p F_i(x_i')+(h_p+\epsilon_p h_pl_i+\alpha_p)\mathcal B_{X_i}\}\\
		&\times\{t_i-\epsilon_p+2h_p\mathcal B_1\}]\cup[\{x_i'+3h_p\mathcal B_{X_i}\}\times\{t_i+2h_p\mathcal B_1\}])\\
		\subseteq&\bar{co}([\{x_i'+\epsilon_p(F_i(x_i')+\epsilon_pM^+l^+\mathcal B_{X_i})+\psi_0\mathcal B_{X_i}\}\\
		&\times\{t_i-\epsilon_p+\psi_0\mathcal B_1\}]\cup[\{x_i'+\psi_0\mathcal B_{X_i}\}\times\{t_i+\psi_0\mathcal B_1\}]),
\end{align*}
where $\psi_0=3h_p+2\epsilon_p h_pl^+$.
By Lemma \ref{thm:prodDist}, the above renders at $\Gamma_i^p(x_i, t_i)\subseteq G^p_i(x_i', t_i)+\sqrt{2}\psi_0\mathcal B_{X_i\times\mathbb R}.$

Combining the above two cases, we see that 
\begin{align*}&\Gamma^p(x, t)=\prod_{i\in\mathcal V}\Gamma^p_i(x_i, t_i)\\
\subseteq&\prod_{i\in\mathcal V}[\bigcup_{\|x'_i-x_i\|\leq h_p}G^p_i(x_i', t_i)+\sqrt{2}\psi_0\mathcal B_{X_i\times\mathbb R}]\\
\subseteq&\bigcup_{\|x'_i-x_i\|\leq h_p, i\in\mathcal V}G^p(x', t)+\sqrt{2N}\psi_0\mathcal B_{\textbf{X}\times\mathbb R^N}\\
\subseteq&\bigcup_{\|x'-x\|\leq \sqrt{Nd}h_p}G^p(x', t)+\sqrt{2N}\psi_0\mathcal B_{\textbf{X}\times\mathbb R^N},\end{align*}
where the second inclusion is a direct result of Lemma \ref{thm:prodDist} applied to the whole product and the third inclusion is a result of Lemma \ref{thm:prodDist} applied to the product of $x'_i-x_i$.
Therefore, for any pair of $(x, t, \tilde x, \tilde t)\in gph(\Gamma^p)$, $\exists(x', t, \tilde x', \tilde t')\in gph(G^p)$ s.t. $\|x'-x\|\leq\sqrt{Nd}h_p<\sqrt{2Nd}\psi_0$ and $\|(\tilde x-\tilde x', \tilde t-\tilde t')\|\leq\sqrt{2N}\psi_0<\sqrt{2Nd}\psi_0$.
Again we apply Lemma \ref{thm:prodDist} and set $\psi(\epsilon_p, h_p)=2\sqrt{Nd}\psi_0=2\sqrt{Nd}(3h_p+2\epsilon_ph_pl^+)$.
Clearly, $\lim_{p\to+\infty}\frac{\psi(\epsilon_p, h_p)}{\epsilon_p}=0^+$.
Hence the claim is proven.
\end{proof}

\begin{claim}\label{thm:appFullyLower}
Equation \eqref{eq:Gamma} satisfies \textbf{(H4)}.
\end{claim}
\begin{proof}
Pick a pair of $(y, t')\in\mathcal H\text{ and }(x, t)\in\mathcal H^p$ s.t. $\|(y, t')-(x, t)\|\leq h_p$.
Fix $i\in\mathcal V$. 
Two cases arise:

Case 1: $d(y_i,  X_i^G)>M_i\epsilon_p$. Then we have
\begin{align*}
	&G_i^p(y_i, t'_i)+h_p\mathcal B_{X_i\times\mathbb R}\\
	\subseteq&\{y_i+\epsilon_p F_i(y_i)+(h_p+\epsilon_p^2M^+l^+)\mathcal B_{X_i}\}\times\{t_i'-\epsilon_p+h_p\mathcal B_1\}\\
	\subseteq&\{x_i+\|y_i-x_i\|\mathcal B_{X_i}+\epsilon_p F_i(x_i)+\epsilon_p l_i\|y_i-x_i\|\mathcal B_{X_i}\\
	&\quad+(h_p+\epsilon_p^2M^+l^+)\mathcal B_{X_i}\}\times\{t_i+\|t_i'-t_i\|\mathcal B_1-\epsilon_p+h_p\mathcal B_1\}\\
	\subseteq&\{x_i+\epsilon_p F_i(x_i)+(2h_p+\epsilon_p h_pl^++\epsilon_p^2M^+l^+)\mathcal B_{X_i}\}\\
	&\quad\times\{t_i-\epsilon_p+2h_p\mathcal B_1\}.
\end{align*}
Therefore, $G_i^p(y_i, t'_i)+h_p\mathcal B_{X_i\times\mathbb R}\subseteq\Gamma_i^p(x_i, t_i)$.

Case 2:  $d(y_i,  X_i)\leq M_i\epsilon_p$.
Then it holds that $d(x_i,  X_i)\leq d(y_i,  X_i)+\|x_i-y_i\|\leq  M_i\epsilon_p+h_p$.
Hence we have
\begin{align*}
	&G_i^p(y_i, t'_i)+h_p\mathcal B_{X_i\times\mathbb R}\\
	\subseteq&\bar{co}([\{y_i+\epsilon_p F_i(y_i)+(h_p+\epsilon_p^2M^+l^+)\mathcal B_{X_i}\}\\
	&\times\{t_i'-\epsilon_p+h_p\mathcal B_1\}]\cup[\{y_i+h_p\mathcal B_{X_i}\}\times\{t'_i+h_p\mathcal B_1\}])\\
	\subseteq&\bar{co}([\{x_i+\|y_i-x_i\|\mathcal B_{X_i}+\epsilon_p F_i(x_i)+\epsilon_pl_i\|y_i-x_i\|\mathcal B_{X_i}\\
	&+(h_p+\epsilon_p^2M^+l^+)\mathcal B_{X_i}\}\times\{t_i+\|t'_i-t_i\|\mathcal B_1-\epsilon_p+h_p\mathcal B_1\}]\\
	&\cup[\{x_i+(\|y_i-x_i\|+h_p)\mathcal B_{X_i}\}\times\{t_i+(\|t'_i-t_i\|+ h_p)\mathcal B_1])\\
	\subseteq&\bar{co}([\{x_i+\epsilon_p F_i(x_i)+(2h_p+\epsilon_p h_pl_i+\epsilon_p^2M^+l^+)\mathcal B_{X_i}\}\\
	&\times\{t_i-\epsilon_p+ 2h_p\mathcal B_1\}]\cup[\{x_i+2h_p\mathcal B_{X_i}\}\times\{t_i+2h_p\mathcal B_1\}]).
\end{align*}

Therefore, $G_i^p(y_i, t'_i)+h_p\mathcal B_{X_i\times\mathbb R}\subseteq\Gamma_i^p(x, t)$ for every $i\in\mathcal V$.
Moreover, we have
\begin{align*}
&\bigcup_{\|(y, t')-(x, t)\|\leq h_p}[G^p(y, t')+h_p\mathcal B_{\textbf{X}\times R^N}]\\
\subseteq&\bigcup_{\|(y, t')-(x, t)\|\leq h_p}[\prod_{i\in\mathcal V}G^p_i(y_i, t'_i)+h_p\mathcal B_{X_i\times R}]
\subseteq\prod_{i\in\mathcal V}\Gamma^p_i(x_i, t_i).
\end{align*}
Intersect both sides with $\mathcal H^p$, then the right-hand side renders at $\Gamma^p(x, t)$.
Hence, $\Gamma^p$ satisfies \textbf{(H4)}.
\end{proof}
Finally, all conditions of Theorem 2.19 in \cite{cardaliaguet1999setvalued} are satisified by Claim \ref{thm:appSemiAppx}, Claim \ref{thm:appFullyUpper} and Claim \ref{thm:appFullyLower}.
Hence, it follows from Theorem 2.19 in \cite{cardaliaguet1999setvalued} that Theorem \ref{thm:fullyApprox} is established.
\end{proof}

\end{subsection}

\begin{subsection}{Consistent Approximation of $\varTheta^*$: Theorem \ref{thm:varthetaconv}}\label{sect:varthetaconv}
In this subsection, we dig it deep to discover the convergence of estimates of minimal arrival time functions $\varTheta^p_n$ updated by Bellman operator $\mathbb T$.
The result is summarized in Theorem \ref{thm:varthetaconv}.

We proposed the following way to cosntruct the estimates of minimal arrival time functions $\varTheta^p_n:\textbf{X}^p\rightrightarrows(\mathbb{R}_{\geq 0}^N)^p$:
\begin{align}\label{eq:constructionvartheta}
\begin{cases}
	\varTheta^{p}_0(x)=&\begin{cases}	\{\textbf{\emph{0}}_N\},	&\text{ if } x\in\textbf{S}^p;\\
									\{+\infty\textbf{\emph{1}}_N\},	&\text{ otherwise;}\\
						\end{cases}\\
	\varTheta^{p}_{n+1}(x)=&(\mathbb T\varTheta^{p}_n)(x),
\end{cases}
\end{align}
where $\mathbb T$ follows definition of \eqref{eq:bellman}.

The convergence is proven via establishing equivalence between $\varTheta^p_n$ and $\mathcal{S}^p_n$ defined in equation \eqref{eq:ss}. 
This is shown by Lemma \ref{thm:vartheta_s}.

\begin{lemma}\label{thm:vartheta_s}
Let system \eqref{eq:0} satisfy Assumption \ref{asmp:1} and \ref{asmp:double}.
Let $\{\varTheta^p_n:\textbf{X}^p\rightrightarrows(\mathbb{R}_{\geq 0}^N)^p\}_{n=0}^\infty$ be a sequence of mappings established by equation \eqref{eq:constructionvartheta}.
Then $\mathcal{E}pi(\varTheta^p_n)=\mathcal{S}^p_n$ for all $n\in\mathbb{N}$, where $\mathcal{S}^p_n$ follows equation \eqref{eq:ss}.
\end{lemma}

Before we proceed to the proof of Lemma \ref{thm:vartheta_s}, several prelinimary results are provided.
The following claim shows the estimated travel time for robots whose distance to their goal regions are within one hop remains zero.
\begin{lemma}\label{thm:zerogoaldomain_prop}
For any $p\geq1$, $x\in\textbf{S}^p$ and any $t\in\varTheta^p_n(x)$, it holds that $t_i=0$ for all $i\in\mathcal V$ s.t. $d(x_i,  X_i^G)\leq M_i\epsilon_p+h_p$.
\end{lemma}

\begin{proof}
The proof is established by induction on $n$.
Denote the induction hypothesis for $n$ by $H(n)$.
Throughout the proof, we fix $p\geq1$ and $x\in\textbf{S}^p$.

By equation \eqref{eq:constructionvartheta}, $H(0)$ trivially holds.
Assume $H(n)$ holds and consider $n+1$.
Fix $t\in\varTheta^p_{n+1}(x)$. 
Then it follows from equation \eqref{eq:constructionvartheta} that $\exists (\tilde{x}, \tilde{\tau})\in\tilde{X}^p(x)\times\tilde{T}^p(x)$ s.t. $t=\tilde{\tau}+\tilde{\mathcal{T}}$, where $\tilde{\mathcal{T}}\in\varTheta_{n+1}^p(\tilde{x})$.
It follows from the definitions of $\tilde{T}^p(x)$ and $\tilde{X}^p(x)$ that $\tilde{\tau}_i=0$ and $\tilde{x}_i=x_i$.
Then by $H(n)$, we have $\forall t\in\varTheta^p_n(x)$, $t_i=0$, which indicates $\tilde{\mathcal{T}}_i=0$.
Thus, $t_i=0+0=0$ and $H(n+1)$ holds.

Hence the claim holds for all $n\geq0$.
\end{proof}

The following claim shows that $\mathcal Epi(\varTheta^p_n)$ is monotonically decreasing with respect to $n$.

\begin{claim}\label{thm:epimonotone}
For any $n\geq0$, $\mathcal{E}pi(\varTheta^p_{n+1})\subseteq\mathcal{E}pi(\varTheta^p_{n})$.
\end{claim}

\begin{proof}
The proof is based on induction on $n$ and follows the proof of Proposition 6.2 in \cite{guigue2014approximation}. 
For the sake of self-containedness, we provide the complete proof.
Denote the induction hypthesis for $n$ by $H(n)$ that $\mathcal{E}pi(\varTheta^p_{n+1})\subseteq\mathcal{E}pi(\varTheta^p_{n})$.

We now proceed to show $H(0)$ holds.
Pick $(x, t)\in\mathcal{E}pi(\varTheta^p_1)$. 
Therefore, $\exists (\tilde{x}, \tilde{\tau})\in(\tilde{X}^p(x)\times\tilde{T}^p(x))\cap\mathcal{S}^p$ s.t. $t\succeq\tilde{\tau}+\mathcal{T}$, where $\mathcal{T}\in\varTheta^p_{0}(\tilde{x})$.
It follows from the definition of $\tilde{T}^p(x)$ and Assumption \ref{asmp:double} that $\tilde{\tau}\succ\textbf{0}$.
By the initial condition in \eqref{eq:constructionvartheta}, we have $\mathcal{T}=\textbf{0}$.
Thus, $t\succ\textbf{0}$, and $t\succ\mathcal{T}$ for any $\mathcal{T}\in\varTheta^p_0(x)$, for any $x\in\textbf{X}^p$.
Hence, $H(0)$ is proven.

Assume $H(n)$ holds and consider $n+1$.
Pick $(x, t)\in\mathcal{E}pi(\varTheta^p_{n+2})$.
It follows from equation \eqref{eq:constructionvartheta} that $\exists (\tilde{x}, \tilde{\tau})\in(\tilde{X}^p(x)\times\tilde{T}^p(x))\cap\mathcal{S}^p$ s.t. $t\succeq\tilde{\tau}+\mathcal{T}^{n+1}$, where $\mathcal{T}^{n+1}\in\varTheta^p_{k+1}(\tilde{x})$.
By Assumption \ref{asmp:double} and the definition of $\tilde{T}^p(x)$, $\tilde{\tau}\succeq\textbf{0}$.
Therefore, $t-\tilde{\tau}\succeq\mathcal{T}^{n+1}$ and $(\tilde{x}, t-\tilde{\tau})\in\mathcal{E}pi(\varTheta^p_{n+1})\subseteq\mathcal{E}pi(\varTheta^p_{n})$.
Hence, $\exists\mathcal{T}^n\in\varTheta^p_n(\tilde{x})$ s.t. $t-\tilde{\tau}\succeq\mathcal{T}^n$, i.e. $t\succeq\tilde{\tau}+\mathcal{T}^n$.
Since $\tilde{\tau}\in\tilde{T}^p(x)\cap(\mathbb{R}^N_{\geq0})^p$ and $\mathcal{T}^n\in\varTheta^p_{n}(\tilde{x})$, where $\tilde{x}\in\tilde{X}^p(x)\cap\textbf{S}^p$, it follow from Theorem 3.2.10 in \cite{sawaragi1985theory} that 
$t\in\{\tilde{\tau}+\varTheta^p_n(\tilde{x})|(\tilde{x}, \tilde{\tau})\in\tilde{X}^p(x)\times\tilde{T}^p(x)\}+\mathbb{R}_{\geq0}^N\subseteq\mathcal{E}(\{\tilde{\tau}+\varTheta^p_n(\tilde{x})|(\tilde{x}, \tilde{\tau})\in\tilde{X}^p(x)\times\tilde{T}^p(x)\}+(\mathbb{R}_{\geq0}^N)^p=\varTheta^p_{n+1}(x)+\mathbb{R}_{\geq0}^N$.
Thus $(\tilde x, \tilde t)\in\mathcal{E}pi(\varTheta^p_{n+1})$ and $H(n+1)$ holds.

Then the claim is proven.
\end{proof}
%

\begin{proofOf}{Lemma \ref{thm:vartheta_s}}
The proof is based on induction on $n$.
Denote the induction hypothesis for $n$ by $H(n)$ as $\mathcal{E}pi(\varTheta^p_n)=\mathcal{S}^p_n$.

When $n=0$, we have $H(0)$ trivially holds.
Assume $H(n)$ holds and consider $n+1$.
We first proceed to show that $\mathcal{E}pi(\varTheta^p_{n+1})\subseteq\mathcal{S}^p_{n+1}$.
Fix $(x, t)\in\mathcal{E}pi(\varTheta^p_{n+1})$. 
Then $\exists\mathcal{T}\in\varTheta^p_{n+1}(x)$ s.t. $t\succeq\mathcal{T}$.
It follows from equation \eqref{eq:constructionvartheta} that $\exists \tilde{x}\in\tilde{X}^p(x)\cap\textbf{S}^p$ and $\tilde{\tau}\in\tilde{T}^p(x)\cap(\mathbb{R}_{\geq0}^N)^p$ s.t. $\mathcal{T}\in\tilde{\tau}+\varTheta^p(\tilde{x})$.
We rewrite $\mathcal{T}$ as $\mathcal{T}=\tilde{\tau}+\tilde{\mathcal{T}}$, where $\tilde{\mathcal{T}}\in\varTheta^p_{n}(\tilde{x})$.
Since $t-\tilde{\tau}\succeq\mathcal{T}-\tilde{\tau}=\tilde{\mathcal{T}}\in\varTheta^p_n(x)$, $(\tilde{x}, t-\tilde{\tau})\in\mathcal{E}pi(\varTheta^p_n)=\mathcal{S}^p_n$.

To prove $(x, t)\in\mathcal{S}^p_{n+1}$, two items are required to prove: (1), $(x, t)\in\mathcal{S}^p_n$; (2), $\Gamma^p(x, t)\cap\mathcal{S}^p_n\neq\emptyset$.
Condition (1) can be derived from Claim \ref{thm:epimonotone} as $(x, t)\in\mathcal{E}pi(\varTheta^p_{n+1})\subseteq\mathcal{E}pi(\varTheta^p_n)=\mathcal{S}^p_n$.
To prove condition (2), a stricter condition is proposed: (3), $(\tilde{x}, t-\tilde{\tau})\in\Gamma^p(x, t)\cap\mathcal{S}^p_{n}$.
Clearly, $(\tilde{x}, t-\tilde{\tau})\in\mathcal{S}^p_n$.
It follows from the definition of $\tilde{x}$ that $\tilde{x}\in(x+\epsilon_p F(x)+\alpha_p\mathcal{B}_\textbf{X})\cap\textbf{S}^p$.
Fix $i\in\mathcal V$.
If $d(x_i, X_i^G)>M_i\epsilon_p+h_p$, then $\tilde{\tau}_i\in\epsilon_p+2h_p\mathcal B_1$ and $t_i-\tilde{\tau}_i\in t_i-\epsilon_p+2h_p\mathcal B_1$.
Otherwise, i.e. $d(x_i, X_i^G)\leq M_i\epsilon_p+h_p$, then $\tilde{\tau}_i=0$ and $t_i-\tilde{\tau}_i=t_i\in\bar{co}((t_i-\epsilon_p+2h_p\mathcal B_1)\cup(t_i+2h_p\mathcal B_1))$.
Thus, $(\tilde{x}, t-\tilde{\tau})\in\Gamma^p(x, t)$.
Hence condition (3) is satisfied and $\mathcal{E}pi(\varTheta^p_{n+1})\subseteq\mathcal{S}^p_{n+1}$ is proven.

Next, we show that $\mathcal{E}pi(\varTheta^p_{n+1})\supseteq\mathcal{S}^p_{n+1}$.
Pick $(x, t)\in\mathcal{S}^p_{n+1}$.
By equation \eqref{eq:ss}, $\exists (\tilde{x}, \tilde{t})\in\Gamma^p(x, t)\cap\mathcal{S}^p_n$.
Since $\mathcal{S}^p_n=\mathcal{E}pi(\varTheta^p_n)$, $\exists\tilde{\mathcal{T}}\in\varTheta^p_n(\tilde{x})$ s.t. $\tilde{t}\succeq\tilde{\mathcal{T}}$.
In order to prove $(x, t)\in\mathcal{E}pi(\varTheta^p_{n+1})$, it is equivalent to prove that $\exists \mathcal{T}\in\varTheta^p_{n+1}(x)$ s.t. $t\succeq\mathcal{T}$.
A sufficient condition is to prove $\mathcal T=t-\tilde t+\tilde{\mathcal T}$, where $\tilde{t}-\tilde{\mathcal{T}}\succeq\textbf{0}$.
The following will show the construction of $(\tilde{x},\tilde{t})$ and $\tilde{\mathcal{T}}\in\varTheta^p_n(\tilde x)$.

By the definition of $\Gamma^p$, two cases arise:

Case 1, $d(x_i, X_i^G)> M_i\epsilon_p+h_p$: then we have $(\tilde{x}_i, \tilde{t}_i)\in(x_i+\epsilon_p F_i(x_i)+\alpha_p\mathcal{B}_{X_i})\times(t_i-\epsilon_p+2h_p\mathcal B_1)$.
Therefore, $\tilde{x}_i\in x_i+\epsilon_p F_i(x_i)+\alpha_p\mathcal{B}_{X_i}=\tilde{X}^p_i(x_i)$ and $t_i-\tilde{t}_i\in\epsilon_p+2h_p\mathcal B_1=\tilde{T}^p_i(x_i)$.

Case 2, $d(x_i, X_i^G)\leq M_i\epsilon_p+h_p$: hence $(\tilde{x}_i, \tilde{t}_i)\in\bar{co}((\{x_i+\epsilon_p F_i(x_i)+\alpha_p\mathcal{B}_{X_i}\}\times\{t_i-\epsilon_p+2h_p\mathcal B_1\})\cup(\{x_i+2h_p\mathcal{B}_{X_i}\}\times\{t_i+ 2h_p\mathcal B_1\}))$.
Choose $(\tilde{x}_i, \tilde{t}_i)=(x_i, t_i)$ so that $(\tilde{x}_i, t_i-\tilde{t}_i)\in\tilde{X}^p_i(x_i)\times\tilde{T}^p_i(x_i)$.
By Lemma \ref{thm:zerogoaldomain_prop} and $d(\tilde{x}_i, X_i^G)=d(x_i, X_i^G)> M_i\epsilon_p+h_p$, $\tilde{\mathcal{T}}_i=0$; hence $(\tilde x_i, \tilde t_i)=(x_i, t_i)$ is a feasible choice.

Hence, $(\tilde{x}, t-\tilde{t})\in\tilde{X}^p(x)\times\tilde{T}^p(x)$ and the corresponding $\tilde{\mathcal{T}}$ could be found.
Since $\mathcal{T}=t-\tilde t+\tilde{\mathcal T}\in\{\tilde{\tau}+\varTheta^p_n(\tilde{x})|(\tilde{x}, \tilde{\tau})\in\tilde{X}^p(x)\times\tilde{T}^p(x)\}$, it follows from Theorem 3.2.10 \cite{sawaragi1985theory} that $\mathcal{T}\in\mathcal{E}(\tilde{\tau}+\varTheta^p_n(\tilde{x})|(\tilde{x}, \tilde{\tau})\in\tilde{X}^p(x)\times\tilde{T}^p(x))=\varTheta^p_{n+1}(x)$. 
Therefore, by $t\succeq\mathcal{T}$, $(x, t)\in\mathcal{E}pi(\varTheta^p_{n+1})$. 
$\mathcal{E}pi(\varTheta^p_{n+1})\supseteq\mathcal{S}^p_{n+1}$ is proved.

Therefore, $H(n+1)$ is proven and so is the lemma.
\end{proofOf}

Finally, we can prove convergence of fixed points $\varTheta^p_\infty$.

\begin{theorem}\label{thm:varthetaconv}
Assume system \eqref{eq:0} satisfies Assumption \ref{asmp:1} and \ref{asmp:double}.
Let $\{\varTheta^p_n:\textbf{X}^p\rightrightarrows(\mathbb R^N_{\geq0})^p\}_{n\in\mathbb N}$ be a sequence of mappings established by equation \eqref{eq:constructionvartheta}.
Then there exists $\varTheta^{p}_\infty$ s.t. $\mathbb T\varTheta^{p}_\infty=\varTheta^p_\infty$, and $\mathrm{Lim}_{p\to+\infty}\mathcal{E}pi(\varTheta^{p}_\infty)=\mathcal{E}pi(\varTheta^*)$, where $\mathrm{Lim}$ denotes Kuratowski limit.
\end{theorem}

\begin{proof}
By Lemma \ref{thm:vartheta_s}, take $\mathrm{Lim}_{n\to+\infty}$ on both sides of $\mathcal{E}pi(\varTheta^p_n)=\mathcal{S}^p_n$.
By Proposition 2.18 in \cite{cardaliaguet1999setvalued}, $\mathrm{Lim}_{n\to+\infty}\mathcal{S}^p_n=\overrightarrow{Viab}_{\Gamma^p}(\mathcal{S}^p)$.
Hence $\mathrm{Lim}_{n\to+\infty}\mathcal{E}pi(\varTheta^p_n)=\overrightarrow{Viab}_{\Gamma^p}(\mathcal{S}^p)$.
Then let the resolutions of time and state go to 0, by Theorem 2.19 in \cite{cardaliaguet1999setvalued}, $\mathrm{Lim}_{p\to+\infty}(\mathrm{Lim}_{n\to+\infty}\mathcal{E}pi(\varTheta^p_n))=Viab_{\Phi}(\mathcal{S})$.
Finally, by Theorem \ref{thm:viabepi}, $\mathrm{Lim}_{p\to+\infty}(\mathrm{Lim}_{n\to+\infty}\mathcal{E}pi(\varTheta^p_n))=\mathcal{E}pi(\varTheta)$.
\end{proof}

\end{subsection}


%
%
%

The following corollary extends Corollary 3.7 in \cite{cardaliaguet1999setvalued} and shows pointwise convergence of fixed points with any diminishing perturbations.
\begin{corollary}\label{thm:varThetaPointwiseConvergence}
If the assumptions in Theorem~\ref{thm:varthetaconv} are fulfilled, $\varTheta^p_\infty$ converges to $\varTheta^*$ in the epigraphical profile sense, i.e. for $\eta_p\geq h_p$ s.t. $\lim_{p\to+\infty}\eta_p=0$, it holds that
\begin{equation*}
\begin{split}
E_{\varTheta^*}(x)=\Lim{p\to+\infty}\bigcup_{\tilde x\in(x+\eta_p\mathcal B_\textbf{X})\cap\textbf{X}^p}E_{\varTheta^p_\infty}(\tilde x), \forall x\in\textbf{X}.
\end{split}\end{equation*}

\end{corollary}

\begin{proof}
Fix $p\geq1$.
Define auxiliary value function $\tilde\varTheta^p_\infty$ s.t. $\mathcal Epi(\tilde \varTheta^p_\infty)\triangleq\mathcal Epi(\varTheta^p_\infty)+\eta_p\mathcal B_{\textbf{X}\times\mathbb R^N}$.
We may rewrite the epigraph of $\tilde\varTheta^p_\infty$ in the following way:
\begin{equation}\label{eq:epiProfileEq}
\begin{split}
\bigcup_{x\in\textbf{X}}[\{x\}\times E_{\tilde\varTheta^p_\infty}(x)]=\bigcup_{x\in\textbf{X}^p}[\{x\}\times E_{\varTheta^p_\infty}(x)]+\eta_p\mathcal B_{\textbf{X}\times\mathbb R^N}.
\end{split}\end{equation}

\begin{claim}\label{thm:epiProfileIneq1}
For any $x\in\textbf{X}$, it holds that
$E_{\tilde\varTheta^p_\infty}(x)\subseteq\bigcup_{\tilde x\in( x+\eta_p\mathcal B_\textbf{X})\cap\textbf{X}^p} E_{\varTheta^p_\infty}(x)+\eta_p\mathcal B_{\textbf{X}\times\mathbb R^N}$.
\end{claim}
\begin{proof}
Fix a pair of $x\in\textbf{X}$ and $t\in E_{\tilde\varTheta^p_\infty}(x)$.
It follows from \eqref{eq:epiProfileEq} that there is a pair of $x'\in\textbf{X}^p$ and $t'\in E_{\varTheta^p_\infty}(x')$ that $\|(x, t)-(x',t')\|\leq\eta_p$.
That is, $\|x-x'\|\leq\eta_p$ and $\|t-t'\|\leq\eta_p$.
Therefore, we have $t\in\bigcup_{x'\in(x+\eta_p\mathcal B_\textbf{X})\cap\textbf{X}^p}(E_{\varTheta^p_\infty}(x')+\eta_p\mathcal B_N)$.
Since this holds for every pair of $x\in\textbf{X}$ and $t\in E_{\tilde\varTheta^p_\infty}(x)$, then the claim is obtained.
\end{proof}

\begin{claim}\label{thm:epiProfileIneq2}
It holds that $E_{\tilde\varTheta^p_\infty}(x')\supseteq  E_{\varTheta^p_\infty}(x'), \forall x'\in\textbf{X}^p$.
\end{claim}
\begin{proof}
Fix a pair of $x'\in\textbf{X}^p$ and $t'\in E_{\varTheta^p_\infty}(x')+\eta_p\mathcal B_N$.
It follows from the definition of unit ball that there is $t\in E_{\varTheta^p_\infty}(x')$ s.t. $\|t-t'\|\leq\eta_p$.
Hence, $\|(x', t')-(x',t)\|\leq\eta_p$.
Since this holds for every $t'\in E_{\varTheta^p_\infty}(x')+\eta_p\mathcal B_N$, then 
$
(\{x'\}\times E_{\varTheta^p_\infty}(x'))+\eta_p\mathcal B_{\textbf{X}\times\mathbb R^N}\supseteq\{x'\}\times(E_{\varTheta^p_\infty}(x')+\eta_p\mathcal B_N).
$
Since this holds for every $x'\in\textbf{X}^p$, then we apply it to \eqref{eq:epiProfileEq} and have
\begin{equation*}
\begin{split}
\bigcup_{x\in\textbf{X}}[\{x\}\times E_{\tilde\varTheta^p_\infty}(x)]\supseteq\bigcup_{x\in\textbf{X}^p}[\{x\}\times(E_{\varTheta^p_\infty}(x)+\eta_p\mathcal B_N)].
\end{split}\end{equation*}
Hence, for any $x'\in\textbf{X}^p$, we take intersection with $\{x'\}\times\mathbb R^N_{\geq0}$ on both sides of the above relationship, then $E_{\tilde\varTheta^p_\infty}(x')\supseteq  E_{\varTheta^p_\infty}(x')+\eta_p\mathcal B_N\supseteq E_{\varTheta^p_\infty}(x')$.
\end{proof}

Recall that by Theorem \ref{thm:kru1}, we have $\mathrm{Lim}_{p\to+\infty}\mathcal Epi(\varTheta^p_\infty)=\mathcal Epi(\varTheta^*)$.
Since $\eta_p\to0$ as $p\to+\infty$, $\mathrm{Lim}_{p\to+\infty}\mathcal Epi(\tilde\varTheta^p_\infty)=\mathrm{Lim}_{p\to+\infty}(\mathcal Epi(\varTheta^p_\infty)+\eta_p\mathcal B)=\mathrm{Lim}_{p\to+\infty}\mathcal Epi(\varTheta^p_\infty)$.
Therefore, it holds true that $\mathrm{Lim}_{p\to+\infty}\mathcal Epi(\tilde\varTheta^p_\infty)=\mathcal Epi(\varTheta^*)$.
It follows from Theorem 5.40 in \cite{rockafellar2009variational} that $\mathrm{Lim}_{p\to+\infty}E_{\tilde\varTheta^p_\infty}(x)=E_{\varTheta^*}(x)$.

\begin{claim}\label{thm:epiProfileSup}
The following holds for all $x\in\textbf{X}$: $$\Limsup{p\to+\infty}\bigcup_{\tilde x\in(x+\eta_p\mathcal B_\textbf{X})\cap\textbf{X}^p}E_{\varTheta^p_\infty}(\tilde x)\subseteq E_{\varTheta^*}(x).$$
\end{claim}
\begin{proof}
Fix $x\in\textbf{X}$.
Since $\textbf{X}^p\subseteq\textbf{X}$, we have $\bigcup_{\tilde x\in x+\eta_p\mathcal B_\textbf{X}}E_{\tilde\varTheta^p_\infty}(\tilde x)\supseteq\bigcup_{\tilde x\in(x+\eta_p\mathcal B_\textbf{X})\cap\textbf{X}^p}E_{\tilde\varTheta^p_\infty}(\tilde x)$.
Further, it follows from Claim \ref{thm:epiProfileIneq2} that $\bigcup_{\tilde x\in(x+\eta_p\mathcal B_\textbf{X})\cap\textbf{X}^p}E_{\tilde\varTheta^p_\infty}(\tilde x)\supseteq\bigcup_{\tilde x\in(x+\eta_p\mathcal B_\textbf{X})\cap\textbf{X}^p}E_{\varTheta^p_\infty}(\tilde x)$.
Then, 
\begin{align}\label{eq:pointwiseConv1}
\bigcup_{\tilde x\in x+\eta_p\mathcal B_\textbf{X}}E_{\tilde\varTheta^p_\infty}(\tilde x)\supseteq\bigcup_{\tilde x\in(x+\eta_p\mathcal B_\textbf{X})\cap\textbf{X}^p}E_{\varTheta^p_\infty}(\tilde x).
\end{align}
Since $\eta_p\to0$, then $\forall\eta>0$, $\exists P>0$ s.t. $\forall p\geq P$, $\eta_p\leq \eta$.
Then it follows from \eqref{eq:pointwiseConv1} that
$\forall p\geq P, \bigcup_{\tilde x\in(x+\eta_p\mathcal B_\textbf{X})\cap\textbf{X}^p}E_{\varTheta^p_\infty}(\tilde x)\subseteq\bigcup_{\tilde x\in x+\eta\mathcal B_\textbf{X}}E_{\tilde\varTheta^p_\infty}(\tilde x).$
Take $\mathrm{Limsup}$ on $p$ at both sides, it follows from Lemma \ref{thm:dominatedConv} that
\begin{equation*}
\begin{split} 
&\Limsup{p\to+\infty}\bigcup_{\tilde x\in(x+\eta_p\mathcal B_\textbf{X})\cap\textbf{X}^p}E_{\varTheta^p_\infty}(\tilde x)\\
\subseteq&\Limsup{p\to+\infty}\bigcup_{\tilde x\in x+\eta\mathcal B_\textbf{X}}E_{\tilde\varTheta^p_\infty}(\tilde x)=\bigcup_{\tilde x\in x+\eta\mathcal B_\textbf{X}}E_{\varTheta^*}(\tilde x), 
\end{split}\end{equation*}
where the last equality follows from $\mathrm{Lim}_{p\to+\infty}E_{\tilde\varTheta^p_\infty}(x)=E_{\varTheta^*}(x)$.
This holds for any $\eta>0$ and $x\in\textbf{X}$, so the claim is proven.
\end{proof}

It follows from Claim \ref{thm:epiProfileIneq1} and Lemma \ref{thm:dominatedConv} that $\forall x\in\textbf{X}$,
$E_{\varTheta^*}(x)=\Liminf{p\to+\infty}E_{\tilde\varTheta^p_\infty}(x)
\subseteq\Liminf{p\to+\infty}\bigcup_{\tilde x\in(x+\eta_p\mathcal B_\textbf{X})\cap\textbf{X}^p}E_{\varTheta^p_\infty}(\tilde x)+\eta_p\mathcal B_N.$
Since $\eta_p$ is diminishing, we have
\begin{align*}\label{eq:epiProfileInf}
E_{\varTheta^*}(x)\subseteq\Liminf{p\to+\infty}\bigcup_{\tilde x\in(x+\eta_p\mathcal B_\textbf{X})\cap\textbf{X}^p}E_{\varTheta^p_\infty}(\tilde x), \forall x\in\textbf{X}.
\end{align*}
By the above inequality and Claim \ref{thm:epiProfileSup}, it is concluded that $\mathrm{Lim}_{p\to+\infty}\bigcup_{\tilde x\in(x+\eta_p\mathcal B_\textbf{X})\cap\textbf{X}^p}E_{\varTheta^p_\infty}(\tilde x)= E_{\varTheta^*}(x)$ and the corollary is proven.
\end{proof}

\end{section}

\end{document}